\documentclass[11pt]{amsart}
\usepackage[T1]{fontenc}
\usepackage[latin9]{inputenc}
\usepackage{geometry}
\usepackage{color}
\usepackage[english,british]{babel}
\usepackage{verbatim}
\usepackage{float}
\usepackage{mathrsfs}
\usepackage{amstext}
\usepackage{amsthm}
\usepackage{amssymb}
\usepackage{stmaryrd}
\makeindex
\usepackage[unicode=true,pdfusetitle,
 bookmarks=true,bookmarksnumbered=false,bookmarksopen=false,
 breaklinks=false,pdfborder={0 0 0},pdfborderstyle={},backref=false,colorlinks=false]
 {hyperref}

\makeatletter

\providecommand{\tabularnewline}{\\}

\numberwithin{equation}{section}
\numberwithin{figure}{section}
\theoremstyle{plain}
\newtheorem{thm}{\protect\theoremname}[section]
\theoremstyle{definition}
\newtheorem{defn}[thm]{\protect\definitionname}
\theoremstyle{definition}
\newtheorem{example}[thm]{\protect\examplename}
\theoremstyle{plain}
\newtheorem{lem}[thm]{\protect\lemmaname}
\theoremstyle{plain}
\newtheorem{prop}[thm]{\protect\propositionname}
\theoremstyle{plain}
\newtheorem{cor}[thm]{\protect\corollaryname}
\theoremstyle{remark}
\newtheorem{rem}[thm]{\protect\remarkname}

\makeatother

\addto\captionsbritish{\renewcommand{\corollaryname}{Corollary}}
\addto\captionsbritish{\renewcommand{\definitionname}{Definition}}
\addto\captionsbritish{\renewcommand{\examplename}{Example}}
\addto\captionsbritish{\renewcommand{\lemmaname}{Lemma}}
\addto\captionsbritish{\renewcommand{\propositionname}{Proposition}}
\addto\captionsbritish{\renewcommand{\remarkname}{Remark}}
\addto\captionsbritish{\renewcommand{\theoremname}{Theorem}}
\addto\captionsenglish{\renewcommand{\corollaryname}{Corollary}}
\addto\captionsenglish{\renewcommand{\definitionname}{Definition}}
\addto\captionsenglish{\renewcommand{\examplename}{Example}}
\addto\captionsenglish{\renewcommand{\lemmaname}{Lemma}}
\addto\captionsenglish{\renewcommand{\propositionname}{Proposition}}
\addto\captionsenglish{\renewcommand{\remarkname}{Remark}}
\addto\captionsenglish{\renewcommand{\theoremname}{Theorem}}
\providecommand{\corollaryname}{Corollary}
\providecommand{\definitionname}{Definition}
\providecommand{\examplename}{Example}
\providecommand{\lemmaname}{Lemma}
\providecommand{\propositionname}{Proposition}
\providecommand{\remarkname}{Remark}
\providecommand{\theoremname}{Theorem}

\begin{document}
\selectlanguage{british}
\title{Lie algebras graded by the weight system $(\Theta_{n},\mathit{sl}_{n})$ }
\author{Alexander Baranov}
\address{Department of Mathematics, University of Leicester, Leicester, UK}
\email{baranov@le.ac.uk}
\thanks{Supported by University of Leicester.}
\author{Hogir M. Yaseen}
\address{Department of Mathematics, Salahaddin University-Erbil, Iraq.}
\email{hogr.yaseen@su.edu.krd}
\thanks{Supported by the Ministry of Higher Education and Scientific Research
in Kurdistan Regional Government-Iraq.}
\begin{abstract}
A Lie algebra $L$ is said to be $(\Theta_{n},\mathit{sl}_{n})$-graded
if it contains a simple subalgebra $\mathfrak{g}$ isomorphic to $sl_{n}$
such that the $\mathfrak{g}$-module $L$ decomposes into copies of
the adjoint module, the trivial module, the natural module $V$, its
symmetric and exterior squares $S^{2}V$ and $\wedge^{2}V$ and their
duals. We describe the multiplicative structures and the coordinate
algebras of $(\Theta_{n},\mathit{sl}_{n})$-graded Lie algebras for
$n\ge5$, classify these Lie algebras and determine their central
extensions. 
\end{abstract}

\maketitle
\global\long\def\bbF{\mathbb{F}}%

\global\long\def\ffg{\mathfrak{g}}%

\global\long\def\ffa{\mathfrak{a}}%

\global\long\def\ffb{\mathfrak{b}}%

\global\long\def\Ker{\operatorname{\rm Ker}}%

\global\long\def\ad{\operatorname{\rm ad}}%

\global\long\def\Der{\operatorname{\rm Der}}%

\global\long\def\Rad{\operatorname{\rm Rad}}%

\global\long\def\tra{\operatorname{\rm tr}}%

\global\long\def\hfb{\operatorname{\rm HF}}%

\global\long\def\Hom{\operatorname{\rm Hom}}%

\global\long\def\End{\operatorname{\rm End}}%

\global\long\def\hom{\operatorname{\rm Hom_{\mathfrak{g}}}}%

\global\long\def\dim{\operatorname{\rm dim}}%

\global\long\def\ki{\operatorname{\rm B}}%

\allowdisplaybreaks

\section{Introduction}

Root graded Lie algebras were introduced by Berman and Moody in 1992
to study toroidal Lie algebras and Slodowy intersection matrix algebras.
They classified the Lie algebras graded by simply-laced root systems
up to central isogeny \cite{berman1992lie}. The case of double-laced
root systems was settled by Benkart and Zelmanov \cite{benkart1996lie}.
Central extensions of root graded Lie algebras in terms of the homology
of its coordinate algebra were determined and described up to isomorphism
by Allison, Benkart and Gao in \cite{allison2000central}.  Non-reduced
systems $BC_{n}$ were considered by Allison, Benkart and Gao \cite{allison2002lie}
(for $n\text{\ensuremath{\ge}}2$) and by Benkart and Smirnov \cite{benkart2003lie}
(for $n=1$). It became clear at that time that this notion can be
generalized further by considering Lie algebras graded by finite weight
systems.

Throughout the paper, the ground field $\bbF$ is of characteristic
zero, $\mathfrak{g}$ is a non-zero split finite dimensional semisimple
Lie algebra over $\bbF$ with root system $\Delta$ and $\Gamma$
is a finite set of integral weights of $\mathfrak{g}$. We say that
a Lie algebra $L$ over $\bbF$ is $(\Gamma,\mathfrak{g})$-\emph{graded},
or simply \emph{$\Gamma$-graded}, if $L$ contains a subalgebra isomorphic
to $\mathfrak{g}$, the $\mathfrak{g}$-module $L$ is the direct
sum of its weight subspaces $L_{\alpha}$ ($\alpha\in\Gamma$) and
$L$ is generated by all $L_{\alpha}$ with $\alpha\ne0$ as a Lie
algebra (see also Definition \ref{def of gamma}). Unless otherwise
stated, we assume that $\mathfrak{g}$ is the grading subalgebra of
the $(\Gamma,\mathfrak{g})$-graded $L$. If $\mathscr{\mathfrak{g}}$
is simple and $\Gamma=\Delta\cup\{0\}$ then $L$ is said to be \emph{root-graded}.
If $\Gamma=BC_{n}\cup\{0\}$ and $\mathfrak{g}$ is of type $B_{n}$,
$C_{n}$ or $D_{n}$, then $L$ is\emph{ $BC_{n}$-graded}. Let $\mathfrak{g}\cong sl_{n}$
and $\Theta_{n}=\{0,\pm\varepsilon_{i}\pm\varepsilon_{j},\pm\varepsilon_{i},\pm2\varepsilon_{i}\mid1\leq i\neq j\leq n\}$
where $\{\varepsilon_{1},\dots,\varepsilon_{n}\}$ is the set of weights
of the natural $sl_{n}$-module. The aim of this paper is to describe
the multiplicative structures and the coordinate algebras of $(\Theta_{n},sl_{n})$-graded
Lie algebras, classify these Lie algebras and determine their central
extensions. 

Various generalizations of root graded Lie algebras were considered.
Neher switched from fields of characteristic zero to rings containing
$1/6$ and working with locally finite root systems instead of finite
\cite{neher1996lie}. Elduque \cite{elduque2013fine} and Draper and
Elduque \cite{draper2016overview} related root gradings with fine
grading. Root graded Lie superalgebras were studied in \cite{benkart2002lie,benkart2003lieA,benkart2003lieB,benkart2004n,martinez2003lie,yousofzadeh2015root}.
  Apart from the $BC_{n}$-graded Lie algebras, other classes of
$\Gamma$-graded Lie algebras with $\Gamma\neq\Delta$ appeared in
the literature. Certain weight-graded Lie algebras were considered
by Neeb \cite{neeb2002locally} (with $\Gamma\setminus\{0\}$ a finite
irreducible root system and $\Delta$ a sub-root system of $\Gamma\setminus\{0\}$)
in a topological setting of locally convex Lie algebras to study some
classes of Lie algebras arising in mathematical physics, operator
theory, and geometry. Let $\mathfrak{g}=sl_{n}$ and $\Gamma_{V}=\Delta\cup V\cup\{0\}$
where $\Delta=A_{n-1}$ and $V$ is the set of weights of the natural
and conatural (the dual of natural) $\mathfrak{g}$-modules. Bahturin
and Benkart \cite{bahturin2004some} (for $n>3$) and Benkart and
Elduque \cite{benkart2012lie} (for $n=3$) described the multiplicative
structure of the $(\Gamma_{V},\mathfrak{g})$-graded Lie algebras.
Note that a Lie algebra is $(\Gamma_{V},\mathfrak{g})$-graded if
and only if  it decomposes as a $\mathfrak{g}$-module into (possibly
infinitely many) copies of the adjoint, natural, conatural and trivial
modules. We believe that the set $\Gamma_{V}$ should be enlarged
further by adding the weights of the symmetric and exterior squares
of the natural and conatural modules, which brings it to $\Theta_{n}$.
Note that a Lie algebra $L$ is $(\Theta_{n},\mathfrak{g})$-graded
if and only if $L$ is generated by $\mathfrak{g}$ as an ideal and
the $\mathfrak{g}$-module $L$ decomposes into copies of the adjoint
module (we will denote it by the same letter $\mathfrak{g}$), the
natural module $V$, its symmetric and exterior squares $S:=S^{2}V$
and $\Lambda:=\wedge^{2}V$, their duals $V',S',\Lambda'$ and the
one dimensional trivial $\mathfrak{g}$-module $T$. Thus, by collecting
isotypic components, we get the following decomposition of the $\mathfrak{g}$-module
$L$:
\begin{equation}
L=(\mathfrak{g}\otimes A)\oplus(V\otimes B)\oplus(V'\otimes B')\oplus(S\otimes C)\oplus(S'\otimes C')\oplus(\Lambda\otimes E)\oplus(\Lambda'\otimes E')\oplus D\label{eq1}
\end{equation}
where $A,B,B',C,C',E,E'$ are vector spaces and $D$ is the sum of
the trivial $\mathfrak{g}$-modules. 

The $\Theta_{n}$-graded Lie algebras did appear in the literature
previously in various contexts. Finite dimensional $\Theta_{n}$-graded
Lie algebras and their representations were studied in \cite{baranov2001plain,baranov2001quasiclassical}.
It was also proved in \cite[4.3]{bahturin2004simple} that a simple
locally finite Lie algebra is $\Theta_{n}$-graded if and only if
it is of diagonal type.

Denote $\mathfrak{a}:=A^{+}\oplus A^{-}\oplus C\oplus E\oplus C'\oplus E'$
and $\mathfrak{b}:=\mathfrak{a}\oplus B\oplus B'$ where $A^{-}$
and $A^{+}$ are two copies of the vector space $A$. We show that
the product in $L$ induces algebra structures on both $\mathfrak{a}$
and $\mathfrak{b}$ (the latter being called the \emph{coordinate
algebra} of $L$). Moreover, $\mathfrak{a}$ is associative if $n\ge7$
or $n=5,6$ and the following conditions on multiplication in $L$
hold:
\begin{align}
 & [\Lambda\otimes E,\Lambda\otimes E]=[\Lambda'\otimes E',\Lambda'\otimes E']=0\text{ for }n=6;\label{eq:MainAssumptions}\\
 & [\Lambda\otimes E,(\Lambda\otimes E)\oplus(V\otimes B)]=[\Lambda'\otimes E',(\Lambda'\otimes E')\oplus(V'\otimes B')]=0\text{ for }n=5.\nonumber 
\end{align}
Note that the conditions (\ref{eq:MainAssumptions}) automatically
hold for $n\ge7$ (see Table \ref{t1}) and are required only because
of irregularities in the tensor product decompositions of the specified
modules for small ranks. We do not consider the case of $n\le4$ in
this paper because of additional technicalities (e.g. $\Lambda\cong\Lambda'$
for $A_{3}$ and $\Lambda\cong V'$ and $\Lambda'\cong V$ for $A_{2}$). 

Our main goal of classification of $\Theta_{n}$-graded Lie algebras
$L$ is achieved in the following steps. 
\begin{enumerate}
\item The computation of all spaces $\hom(X\otimes Y,Z)$ where $X,Y,Z\in\{\mathfrak{g},V,V',S,\Lambda,S',\Lambda',T\}$,
see (\ref{t2}).
\item The determination of the system of products on the components of (\ref{eq1})
induced by multiplication in $L$, see (\ref{main for}).
\item Description of the coordinate algebra $\mathfrak{b}$ of $L$ (Theorem
\ref{structure}).
\item For a given algebra $\mathfrak{b}$, we construct an explicit model
of a $\Theta_{n}$-graded Lie algebra $\mathfrak{u}=\mathfrak{u}(\mathfrak{b})$
with coordinate algebra $\mathfrak{b}$ (Example \ref{exa:Model x}). 
\item We define a centerless algebra $\mathcal{L}(\mathfrak{b})$ with coordinate
algebra $\mathfrak{b}$ (as $\mathcal{L}(\mathfrak{b})=\mathfrak{u}/Z(\mathfrak{u})$)
and we show that every $\Theta_{n}$-graded Lie algebra $L$ with
coordinate algebra $\mathfrak{b}$ is a cover of $\mathcal{L}(\mathfrak{b})$,
i.e. $L/Z(L)\cong\mathcal{L}(\mathfrak{b})$ (Theorem \ref{L(b) center-1}).
\item We prove that $L$ is centrally isogenous to $\mathfrak{u}(\mathfrak{b})$
(i.e. $L/Z(L)\cong\mathfrak{u}(\mathfrak{b})/Z(\mathfrak{u}(\mathfrak{b}))$.
In particular, $L$ is uniquely determined (up to central isogeny)
by its coordinate algebra $\mathfrak{\mathfrak{b}}$ (Theorem \ref{model theorem}).
This completes the classification up to central extensions. 
\item We find the universal central extension $\widehat{\mathcal{L}(\mathfrak{b})}$
of $\mathcal{L}(\mathfrak{b})$ and show that its center is $\hfb(\mathfrak{b})$,
the full skew-dihedral homology group of $\mathfrak{b}$ (Theorem
\ref{universal covering}). We prove that every $\Theta_{n}$-graded
Lie algebra with coordinate algebra $\mathfrak{b}$ is isomorphic
to $\mathcal{L}(\mathfrak{b},X)=\widehat{\mathcal{L}(\mathfrak{b})}/X$
for some subspace $X$ of $\hfb(\mathfrak{b})$, which classifies
the $\Theta_{n}$-graded Lie algebras up to isomorphisms (Theorem
\ref{main universal}).
\end{enumerate}
The paper is organized as follows. In Section \ref{chap:3}  we review
main concepts and results of the theory of Lie algebras graded by
finite root systems and establish general properties of $\Gamma$-graded
Lie algebras. In Section \ref{sec:Multiplication-in--graded} we describe
the multiplicative structures of $\Theta_{n}$-graded Lie algebras.
The coordinate algebra of a $\Theta_{n}$-graded Lie algebra and its
properties are analyzed in Section \ref{chap:4}. In Section \ref{chap:5}
we classify $\Theta_{n}$-graded Lie algebras up to central extensions
and isomorphisms. 

\section{\label{chap:3} $\Gamma$-graded Lie algebras}

This section is organized as follows. First we review main concepts
and results of the theory of Lie algebras graded by finite root systems
and non-reduced systems $BC_{n}$ ($n\geq2$). Then we study general
properties of $\Gamma$-graded Lie algebras. After that we discuss
the similarities between the $\Theta_{n}$-graded and $BC_{n}$-graded
Lie algebras by showing that every $\Theta_{n}$-graded Lie algebra
is $BC_{r}$-graded with $r=\lfloor\frac{n}{2}\rfloor$ and every
$BC_{n}$-graded Lie algebra is $\Theta_{n}$-graded, see Theorems
\ref{bc} and \ref{bc-converse}. We show that our approach gives
a ``finer'' multiplicative and coordinate algebra structure on $L$
as we have more components in the decomposition of $L$ (see Remark
\ref{rem:finer}).

\subsection{Main definition and examples}

We start with the general definition of Lie algebras graded by finite
weight systems \cite{bahturin2004simple}. 
\begin{defn}
\label{def of gamma} Let $\Delta$ be a root system and let $\Gamma$
be a finite set of integral weights of $\Delta$ containing $\Delta$
and $\{0\}$. A Lie algebra $L$ is called $(\Gamma,\mathfrak{g})$-\emph{graded
}(or simply\emph{ $\Gamma$-graded}) if

$(\Gamma1)$ $L$ contains a non-zero finite-dimensional split semisimple
Lie subalgebra $\mathfrak{g}=\mathfrak{h}\oplus\bigoplus_{\alpha\in\Delta}\mathit{\mathrm{\mathfrak{g}}}_{\alpha}$
whose root system is $\Delta$ relative to a split Cartan subalgebra
$\mathfrak{h}=\mathfrak{g}_{0}$;

$(\Gamma2)$ $L=\bigoplus_{\alpha\in\Gamma}L_{\alpha}$ where $L_{\alpha}=\left\{ x\in\mathrm{\mathit{L}\mid\left[\mathit{h,x}\right]=\mathit{\alpha\left(h\right)x\text{ for all }h\in\mathrm{\mathit{\mathfrak{h}}}}}\right\} $;

$(\Gamma3)$ $L_{0}=\sum_{\alpha,-\alpha\in\Gamma\setminus\{0\}}\left[L_{\alpha},L_{-\alpha}\right]$.
\end{defn}

The subalgebra $\mathfrak{g}$ is called the \emph{grading subalgebra}
of $L$. A Lie algebra $L$ is called $(\Gamma,\mathfrak{g})$-\emph{pregraded}
if it satisfies $(\Gamma1)$ and $(\Gamma2)$ (but not necessarily
$(\Gamma3)$). Note that the condition $(\Gamma2)$ yields $[L_{\mu},L_{\nu}]\subseteq L_{\mu+\nu}$
if $\mu+\nu\in\Gamma$ and $[L_{\mu},L_{\nu}]=0$ otherwise. Note
that a $(\Gamma,\mathfrak{g})$-pregraded Lie algebra $L$ is $(\Gamma,\mathfrak{g})$-graded
if and only if the ideal generated by $\mathscr{\mathfrak{g}}$ coincides
with $L$, see Proposition \ref{pre-4 equivalent}.

Suppose that the grading subalgebra $\mathscr{\mathfrak{g}}$ is simple.
If $\Gamma=\Delta\cup\{0\}$ then $L$ is said to be \emph{root-graded}.
If $\Gamma=BC_{n}\cup\{0\}$ and $\mathfrak{g}$ is of type $B_{n}$,
$C_{n}$ or $D_{n}$, then $L$ is said to be $BC_{n}$-\emph{graded}.
If $\Gamma=\Theta_{n}$ and $\mathfrak{g}$ is of type $A_{n-1}$
then $L$ is said to be $\Theta_{n}$-\emph{graded}. Clearly, any
Lie algebra which is graded by a finite root system of type $B_{r},C_{r},$
or $D_{r}$ is also $BC_{r}$-graded. Moreover, note that every $BC_{n}$-graded
Lie algebra ($n\ge2$) is $\Theta_{n}$-graded, see Theorem \ref{bc-converse}.

\begin{example}
Let $A$ be an associative commutative $\mathbb{F}$-algebra with
unit $1$ and let $\mathfrak{g}$ be a split simple Lie algebra of
type $\Delta$. Then $L=\mathfrak{g}\otimes A$ is a $(\Delta,\mathfrak{g}\otimes1)$-graded
Lie algebra with respect to the bracket $[x\varotimes a,y\varotimes b]=[x,y]\varotimes ab$
for all $x,y\in\mathfrak{g}$ and $a,b\in A.$ More generally, any
perfect central extension of $\mathfrak{g}\otimes A$ is also $(\Delta,\mathfrak{g}\otimes1)$-graded.
The universal covering algebra of $\mathfrak{g}\otimes A$ is a generalization
of the affine Kac-Moody algebra determined by $\mathfrak{g}$ \cite[0.5]{benkart1996lie}
.
\end{example}

\begin{example}
\cite{allison2002lie} \label{example ka}(1) Affine Lie algebras
(or more precisely their derived algebras) which have realization
as $\mathfrak{g}^{aff}=(\mathfrak{g}\otimes\mathbb{F}[t^{\pm1}])\oplus\mathbb{F}z$
where $\mathbb{F}[t^{\pm}]$ is the algebra of Laurent polynomials
in $t$ over $\mathbb{F}$ and $\mathbb{F}z$ is a one dimensional
(non split) center, are $\Delta$-graded.

(2) Toroidal Lie algebras, which can be realized as $\mathfrak{g}^{aff}=(\mathfrak{g}\otimes\mathbb{F}[t_{1}^{\pm1},\cdots,t_{n}^{\pm1}])\oplus Z$
where $Z$ is an infinite dimensional non-split center, are $\Delta$-graded.

(3) The twisted affine algebras $(\mathfrak{g}\otimes F[t^{\pm2}])\oplus(W\otimes tF[t^{\pm2}])\oplus Fz$
with $\Delta=B_{r},C_{r},F_{4}$ and their toroidal counterparts are
graded by the root system of $\mathfrak{g}$ ($W$ is the irreducible
$\mathfrak{g}$-module whose highest weight is the highest short root).

(4) The Tits-Kantor-Koecher Lie algebra $K(A)=(sl_{2}\otimes A)\oplus[L_{A},L_{A}]$
of a unital Jordan algebra $A$ where $L_{A}$ denotes left multiplication
by $a\in A$, is graded by $\Delta=A_{1}$.
\end{example}

\begin{example}
Let $L=\mathfrak{g}_{1}\oplus\mathfrak{g}_{2}$ where $\mathfrak{g}_{1}$
and $\mathfrak{g}_{2}$ are ideals of $L$ isomorphic to $sl_{n}$
and let $\mathfrak{g}$ be the diagonal subalgebra of $L$ isomorphic
to $sl_{n}$. Then $L$ is $(A_{n-1},\mathfrak{g})$-graded. Note
that $L$ is also $(A_{n-1},\mathfrak{g}_{i})$-\emph{pregraded},
but not $(A_{n-1},\mathfrak{g}_{i})$-graded as it fails to satisfy
condition $(\Gamma3)$ in the definition. 
\end{example}

\begin{example}
Let $L=sl_{n+k}$ and let $\mathfrak{g}$ be the copy of $sl_{n}$
in the northwest corner. We consider the adjoint action of $\mathfrak{g}$
on $L$. Then the $\mathfrak{g}$-module $L$ decomposes into $k$
copies of the natural module $V=\mathbb{F}^{n}$, $k$ copies of the
dual module $V'=\Hom(V,\mathbb{F})$, an adjoint module $\mathfrak{g}$
and one dimensional trivial $\mathfrak{g}$-modules in its southeast
corner. Then $L=\mathfrak{g}\oplus V^{\oplus k}\oplus V'^{\oplus k}\oplus D$
where $D$ is the sum of the trivial $sl_{n}$-modules. As a result,
we may write 
\[
L=\mathfrak{g}\oplus(V\otimes B)\oplus(V'\otimes B')\oplus D
\]
where $B\cong B'\cong\mathbb{F}^{k}$. Then $L$ is $(\Theta_{n},\mathfrak{g})$-graded.
Bahturin and Benkart \cite{bahturin2004some} (for $n>3$) and Benkart
and Elduque \cite{benkart2012lie} (for $n=3$) described the multiplicative
structure of this type of Lie algebras. Note that $L$ is also $(A_{n+k-1},L)$-graded.
This shows that Lie algebras can be weight graded in different ways. 
\end{example}

\begin{example}
Let $L=sl_{2n+1}$ and $\mathfrak{g}=\left\{ \left[\begin{array}{ccc}
x & 0 & 0\\
0 & -x^{t} & 0\\
0 & 0 & 0
\end{array}\right]\mid x\in sl_{n}\right\} \subset L.$ We consider the adjoint action of $\mathfrak{g}$ on $L$. Then $L$
is $(\Theta_{n},\mathfrak{g})$-graded. Moreover, one can check that
all the componets in the decomposition (\ref{eq1}) for this algebra
are non-zero (see \cite[Example 3.3.3]{yaseen2018generalized}).
\end{example}

\begin{example}
Let $L=\mathfrak{g}\oplus R$ where $R=\Rad L$ and $\mathfrak{g}$
is a simple subalgebra of $L$ isomorphic to $sl_{n}$. Suppose $[R,R]=0$
and $R$ is a finite dimensional simple $\mathfrak{g}$-module with
highest weight $\lambda$ under the adjoint action of $\mathfrak{g}$.
Then $L$ is $(\Theta_{n},\mathfrak{g})$-graded if and only if $\lambda\in\Theta_{n}$.
\end{example}

\subsection{Multiplicative structure of the root graded and $BC_{r}$-graded
Lie algebras}

In this subsection we briefly recall multiplicative structures and
coordinate algebras of the root graded and $BC_{r}$-graded Lie algebras
$L$. Let $\mathfrak{g}$ be the grading subalgebra of $L$ and let
$\Delta$ be its root system. Then $\Gamma=\Delta\cup\{0\}$ or $BC_{r}\cup\{0\}$.
The multiplicative structure and the coordinate algebra of $L$ is
obtained as follows.
\begin{itemize}
\item[(1)]  $\Gamma=\Delta\cup\{0\}$ and $\Delta=A_{n-1}$ with $n\ge3$ (\cite{berman1992lie}
and \cite[4.14]{allison2000central}). Note that the Lie algebra $L$
in this case is also $\Theta_{n}$-graded, so $L\cong(\mathfrak{g}\varotimes A)\oplus D$
with the same multiplication as in (\ref{main for}) with $B=B'=C=C'=E=E=\{0\}$.
Here $A$ is an associative (if $n\ge4$) or alternative (if $n=3$)
algebra over $\mathbb{F}$ and $D$ is the sum of trivial $\mathfrak{g}$-modules
(acting by derivations on $A$).
\item[(2)]  $\Gamma=\Delta\cup\{0\}$ and $\Delta=E_{r}$ ($r=6,7,8$) or $\Delta=A_{1}$
(\cite{berman1992lie} and \cite[2.34]{allison2000central}). Then
there is a commutative associative algebra $A$ (or Jordan algebra
$A$ if $\Delta=A_{1}$) over $\mathbb{F}$ such that $L\cong(\mathfrak{g}\varotimes A)\oplus D,$
with $[x\varotimes a,d]=x\varotimes ad$, and $[x\varotimes a,y\varotimes a']=[x,y]\varotimes aa'+(x\mid y)\langle a,a'\rangle$
where $x,y\in\mathfrak{g}$, $a,a'\in A$ and $d,\langle a,a'\rangle\in D$. 
\item[(3)]  $\Gamma=\Delta\cup\{0\}$ and $\Delta=B_{r}$, $C_{r},$ or $D_{r}$
with $r\ge2$, see \cite{benkart1996lie}. Note that $L$ is also
$BC_{r}$-graded so (5) can be used instead. 
\item[(4)] $\Gamma=\Delta\cup\{0\}$ and $\Delta=F_{4},G_{2}$, see \cite{benkart1996lie}.
\item[(5)]  $\Gamma=BC_{r}\cup\{0\}$, $\Delta=B_{r}$, $C_{r},$ or $D_{r}$,
$r\ge3$, and $\Delta\ne D_{3}$, see \cite{allison2002lie}. Then
there exists an $\bbF$-algebra $\mathfrak{a}$ with involution $\eta$
having symmetric elements $A$ and skew symmetric elements $B$ relative
to $\eta$, an $\mathfrak{a}$-module $C$, an $\mathfrak{a}$-sesquilinear
form $\chi(\,,\,)$ on $C$ so that
\begin{itemize}
\item[(a)]  $\mathfrak{a}$ is associative unless $r=3$ and $\mathfrak{g}$-has
type $C_{3}$ in which case $\mathfrak{a}$ is alternative and $A$
is contained in the nucleus (associative center) of $\mathfrak{a};$
\item[(b)]  $C$ is an associative $\mathfrak{a}$-module and $\chi(\,,\,)$
is hermitian (skew-hermitian) if the form on the natural $\mathfrak{g}$-module
$V$ is symmetric (skew-symmetric);
\item[(c)]  $L=(\mathfrak{g}\otimes A)\oplus(\mathfrak{s}\otimes B)\oplus(V\otimes C)\oplus D$
where $D$ is the centralizer of $\mathfrak{g}$ in $L$ and $\mathfrak{s}$
is the irreducible $\mathfrak{g}$-module of highest weight $2\omega_{1}$
(if $\mathfrak{g}$ is of type $B_{r}$ or $D_{r}$) or $\omega_{2}$
(if $\mathfrak{g}$ is of type $C_{r}$), see Proposition \ref{prop:4om}.
 Moreover, we may suppose that there exist several products/mappings
between the components of the coordinate algebra, which induce multiplication
in $L$, see \cite{allison2002lie} for details . 
\end{itemize}
\end{itemize}

\subsection{Basic properties of $\Gamma$-graded Lie algebras}
\begin{lem}
\label{Le- pregra} Let $L$ be a Lie algebra containing a non-zero
split semisimple subalgebra $\mathfrak{g}$ and let $V(\lambda)$
denotes the simple $\mathfrak{g}$-module with highest weight $\lambda$.
Then $L$ is $(\Gamma,\mathfrak{g})$-pregraded for some finite set
$\Gamma$ if and only if there exists a finite set $Q$ of dominant
weights of $\mathfrak{g}$ such that $L$ is the direct sum of finite-dimensional
irreducible $\mathfrak{g}$-modules whose highest weights are in $Q$,
i.e. as a $\mathfrak{g}$-module, $L\cong\underset{\lambda\in Q}{\bigoplus}V(\lambda)\otimes W_{\lambda}$
for some vector spaces $W_{\lambda}$ (the ~vector space $W_{\lambda}$
indexes the copies of $V(\lambda)$ and the $\mathfrak{g}$-action
is given by $x.(v_{\lambda}\otimes w_{\lambda})=[x,v_{\lambda}\otimes w_{\lambda}]=x.v_{\lambda}\otimes w_{\lambda}$
for $x\in\mathfrak{g}$, $v_{\lambda}\in V(\lambda)$ and $w_{\lambda}\in W_{\lambda}$).
\end{lem}

\begin{proof}
The ``if'' part is obvious with $\Gamma$ being the union of the
weights of the modules $V(\lambda)$, $\lambda\in Q$. To prove the
``only if'' part it is enough to note that $L$ is locally finite
as a $\mathfrak{g}$-module (i.e every finitely generated submodule
is finite-dimensional), so $L$ is semisimple as a $\mathfrak{g}$-module
(see for example \cite[Lemma 2.2]{benkart2003lieB}).
\end{proof}
\begin{prop}
\label{pre-4 equivalent} Let $\mathfrak{g}$ be a split semisimple
subalgebra of a Lie algebra $L$ and suppose $L$ is $(\Gamma,\mathfrak{g})$-pregraded.
Let $G$ be the ideal of $L$ generated by $\mathfrak{g}$. Then $G=\bigoplus_{\alpha\in\Gamma\setminus\{0\}}L_{\alpha}+\sum_{\alpha,-\alpha\in\Gamma\setminus\{0\}}\left[L_{\alpha},L_{-\alpha}\right]$
and it is $(\Gamma,\mathfrak{g})$-graded. In particular, $L$ is
$(\Gamma,\mathfrak{g})$-graded if and only if $G=L$.
\end{prop}

\begin{proof}
Note that $L_{\alpha}=[\mathfrak{g}_{0},L_{\alpha}]\subseteq G$ for
all $\alpha\neq0$, so $G$ contains the subalgebra $G'=\bigoplus_{\alpha\in\Gamma\setminus\{0\}}L_{\alpha}+\sum_{\alpha,-\alpha\in\Gamma\setminus\{0\}}\left[L_{\alpha},L_{-\alpha}\right]$,
which is clearly $(\Gamma,\mathfrak{g})$-graded. Note that $[G',L_{0}]\subseteq G'$
so $G'$ is an ideal of $L$ containing $\mathfrak{g}$. Therefore
$G=G'$, as required. 
\end{proof}
\begin{cor}
\label{Th-simple and G} Let $\mathfrak{g}$ be a split semisimple
finite dimensional subalgebra of a simple Lie algebra $L$. Suppose
$L$ is finite dimensional or $(\Gamma,\mathfrak{g})$-pregraded for
some $\Gamma$. Then $L$ is $(\Gamma,\mathfrak{g})$-graded.
\end{cor}

\begin{prop}
\label{Th-transtivity of gradua} Suppose $L$ is $(\Gamma_{1},\mathfrak{g}_{1})$-graded
and $\mathfrak{g}_{1}$ is $(\Gamma_{2},\mathfrak{g}_{2})$-graded.
Then $L$ is $(\Gamma_{3},\mathfrak{g}_{2})$-graded where $\Gamma_{3}$
is the set of all weights of the $\mathfrak{g}_{2}$-module $L.$
\end{prop}

\begin{proof}
Clearly, $L$ is $(\Gamma_{3},\mathfrak{g}_{2})$-pregraded. By Lemma
\ref{pre-4 equivalent}, $L$ is generated by $\mathfrak{g}_{1}$
as an ideal and $\mathfrak{g}_{1}$ itself is generated by $\mathfrak{g}_{2}$
as an ideal, so the ideal of $L$ generated by $\mathfrak{g}_{2}$
coincides with $L$. Hence by Lemma \ref{pre-4 equivalent} $L$ is
$(\Gamma_{3},\mathfrak{g}_{2})$-graded.
\end{proof}
\begin{lem}
\label{Le-extenrnal} Let $L_{i}$ be $(\Gamma_{i},\mathfrak{g}_{i})$-graded
for $i=1,2$. Suppose that $\mathfrak{g}_{1}\cong\mathfrak{g}_{2}$.
Let $\mathfrak{g}=\{(x,x)\mid x\in\mathfrak{g}_{1}\}\cong\mathfrak{g}_{1}$
be the diagonal subalgebra of $\mathfrak{g}_{1}\oplus\mathfrak{g}_{2}\subseteq L_{1}\oplus L_{2}$.
Then $L_{1}\oplus L_{2}$ is $(\Gamma_{1}\cup\Gamma_{2},\mathfrak{g})$-graded. 
\end{lem}

\begin{proof}
Clearly, $L_{1}\oplus L_{2}$ is $(\Gamma_{1}\cup\Gamma_{2},\mathfrak{g})$-pregraded.
It remains to note that the ideal generated by $\mathfrak{g}$ coincides
with $L_{1}\oplus L_{2}$ and use Proposition \ref{pre-4 equivalent}.
\end{proof}
\begin{lem}
\label{le-semi} Let $L$ be a non-zero finite-dimensional split semisimple
Lie algebra. Then $L$ is $(\Gamma,sl_{2})$-graded for some $\Gamma$.
\end{lem}

\begin{proof}
Let $L=S_{1}\oplus S_{2}\oplus\cdots\oplus S_{k}$ where $S_{i}$
are split simple ideals. Note that each $S_{i}$ is $(\Gamma,sl_{2})$-graded
(just fix any subalgebra $\mathfrak{g}_{i}\cong sl_{2}$ of $S_{i}$
and use Corollary \ref{Th-simple and G}). It remains to apply Lemma
\ref{Le-extenrnal}.
\end{proof}
\begin{thm}
\label{Th-per is sl2} Let $L$ be a non-zero finite-dimensional perfect
Lie algebra over an algebraically closed field of characteristic zero
and let $Q$ be any Levi subalgebra of $L$. Then

$(1)$ $L$ is $(\Gamma,Q)$-graded for some $\Gamma$.

$(2)$ $L$ is $(\Gamma,sl_{2})$-graded for some $\Gamma$.
\end{thm}

\begin{proof}
(1) Let $R$ be the solvable radical of $L$. Then $L=Q\oplus R$.
Note that $L$ is $(\Gamma,Q)$-pregraded where $\Gamma$ is the set
of weights of the $Q$-module $L$. Denote by $P$ the ideal generated
by $Q$ in $L$. Since $R$ is solvable, $L/P=(P+R)/P\cong R/(P\cap R)$
is solvable. But $L/P$ is perfect, so $L/P=\left\{ 0\right\} $ and
$L=P$. By Proposition \ref{pre-4 equivalent}, $L$ is $(\Gamma,Q)$-graded.

(2) This follows from Lemma \ref{le-semi} and Proposition \ref{Th-transtivity of gradua}.
\end{proof}

\subsection{$\Theta_{n}$-graded and $BC_{n}$-graded Lie algebras}

In this subsection we discuss the relationship between $\Theta_{n}$-graded
and $BC_{n}$-graded Lie algebras. Let $\mathfrak{g}$ be a split
simple Lie algebra of classical type $A_{n}$, $B_{n}$, $C_{n}$
or $D_{n}$. Throughout this paper, $\{\omega_{1},\dots,\omega_{n}\}$
is the set of the fundamental weights of $\mathfrak{g}$; $V_{\mathfrak{g}}(\omega)$
(or simply $V(\omega)$) denotes the highest weight $\mathfrak{g}$-module
of weight $\omega$; $V_{\mathfrak{g}}:=V_{\mathfrak{g}}(\omega_{1})$
(or simply $V$) is the natural $\mathfrak{g}$-module; if $M$ is
a $\mathfrak{g}$-module then $M'$ is its dual and $\mathcal{W}(M)$
is the set of weights of $M$. If $\mathfrak{g}$ is of type $A_{n-1}$,
we will use the following notations for the $\mathfrak{g}$-modules
below:
\[
\mathfrak{g}:=V(\omega_{1}+\omega_{n-1}),\ V:=V(\omega_{1}),\ S:=V(2\omega_{1}),\ \Lambda:=V(\omega_{2})\text{ and }T:=V(0).
\]
Recall that for type $A_{n-1}$, $V'\cong V(\omega_{n-1}),$ $S'\cong V(2\omega_{n-1})$,
$\Lambda'\cong V(\omega_{n-2})$ and $\omega_{i}=\varepsilon_{1}+\varepsilon_{2}+\cdots+\varepsilon_{i}$
for $i=1,\dots,n-1$ were $\{\varepsilon_{1},\dots,\varepsilon_{n}\}$
is the set of the weights of the natural $sl_{n}$-module. 

Recall that a Lie algebra $L$ is $(\Gamma,\mathfrak{g})$-pregraded
if it satisfies $(\Gamma1)$ and $(\Gamma2)$ of Definition \ref{def of gamma}.
It is easy to see that $BC_{n}$-pregraded Lie algebras have the following
decomposition, see for example \cite[2.5]{allison2002lie}.
\begin{prop}
\label{prop:4om} Let $L$ be a Lie algebra and let $\mathfrak{g}$
be a split simple subalgebra of $L$ of type type $B_{n}$, $C_{n}$
($n\geq2$) or $D_{n}$ ($n\geq3$). Then $L$ is $(BC_{n}\cup\{0\},\mathfrak{g})$-pregraded
if and only if the $\mathfrak{g}$-module $L$ is a direct sum of
copies of $V_{\mathfrak{g}}(2\omega_{1})$, $V_{\mathfrak{g}}(\omega_{2})$,
$V_{\mathfrak{g}}(\omega_{1})$ and $V_{\mathfrak{g}}(0)$. 
\end{prop}

By using Lemma \ref{Le- pregra} and looking into the dominant weights
appearing in $\Theta_{n}$ we immediately get a similar decomposition
for the $\Theta_{n}$-pregraded Lie algebras.
\begin{prop}
\label{Th-stru of gamma deco} Let $L$ be a Lie algebra and let $\mathfrak{g}$
be a subalgebra of $L$ isomorphic to $sl_{n}$. Then $L$ is $(\Theta_{n},\mathfrak{g})$-pregraded
if and only if the $\mathfrak{g}$-module $L$ is a direct sum of
copies of $\mathfrak{g}$, $V$, $V'$, $S$, $S'$, $\Lambda,$ $\Lambda'$
and $T$.
\end{prop}

Suppose $L$ is $(\Theta_{n},\mathfrak{g})$-graded. By collecting
isomorphic summands of $L$ into isotypic components, we may assume
that there are vector spaces $A,B,B',C,C',E,E'$ such that 
\begin{equation}
L\cong(\mathfrak{g}\otimes A)\oplus(V\otimes B)\oplus(V'\otimes B')\oplus(S\otimes C)\oplus(S'\otimes C')\oplus(\Lambda\otimes E)\oplus(\Lambda'\otimes E')\oplus D\label{eq:drezh}
\end{equation}
where $D$ is the sum of the trivial $\mathfrak{g}$-modules (and
also the centralizer of $\mathfrak{g}$ in $L$). 

Recall that $\mathcal{W}(M)$ denotes the set of weights of a $\mathfrak{g}$-module
$M$ and $M'$ denotes the dual of $M$. If $\mathfrak{k}$ is a subalgebra
of $\mathfrak{g}$ we denote by $M\downarrow\mathfrak{k}$ the restriction
of the $\mathfrak{g}$-module $M$ to $\mathfrak{k}$. We will need
the following trivial observation. 
\begin{lem}
\label{restrection} Let $\mathfrak{k}$ be a simple Lie algebra of
type type $B_{r}$, $C_{r}$ or $D_{r}$ and let $\mathfrak{g}$ be
a simple Lie algebra of type $A_{n-1}$. Denote $\Gamma_{\mathfrak{k}}:=\mathcal{W}((T\oplus V_{\mathfrak{k}})\otimes(T\oplus V_{\mathfrak{k}}))$
and $\Gamma_{\mathfrak{g}}:=\mathcal{W}((T\oplus V_{\mathfrak{g}}\oplus V_{\mathfrak{g}}')\otimes(T\oplus V_{\mathfrak{g}}\oplus V_{\mathfrak{g}}')).$
Then $\Gamma_{\mathfrak{k}}=BC_{r}\cup\{0\}$ and $\Gamma_{\mathfrak{g}}=\Theta_{n}$.
Moreover, 
\end{lem}

\begin{itemize}
\item[(1)]  if $\mathfrak{k}\cong so_{n}$ is a naturally embedded subalgebra
of $\mathfrak{g}\cong sl_{n}$ then $V_{\mathfrak{g}}\downarrow\mathfrak{k}\cong V_{\mathfrak{k}}$,
$V_{\mathfrak{g}}'\downarrow\mathfrak{k}\cong V_{\mathfrak{k}}$ and
$\Gamma_{\mathfrak{g}}\downarrow\mathfrak{k}=\Gamma_{\mathfrak{k}}$;
\item[(2)]  if $\mathfrak{g}\cong sl_{n}$ is a naturally embedded subalgebra
of $\mathfrak{k}\cong so_{2n+1}$, $so_{2n}$ or $sp_{2n}$ then $V_{\mathfrak{k}}\downarrow\mathfrak{g}\cong V_{\mathfrak{g}}\oplus V_{\mathfrak{g}}'$
(or $V_{\mathfrak{g}}\oplus V_{\mathfrak{g}}'\oplus T$ if $\mathfrak{k}\cong so_{2n+1}$)
and $\Gamma_{\mathfrak{k}}\downarrow\mathfrak{g}=\Gamma_{\mathfrak{g}}$.
\end{itemize}
\begin{thm}
\label{bc} Let $n\ge2$ and $r=\lfloor\frac{n}{2}\rfloor$. Then
every $\Theta_{n}$-graded Lie algebra is $BC_{r}$-graded with grading
subalgebra of type $B_{r}$ (if $n$ is odd) or $D_{r}$ (if $n$
is even). 
\end{thm}

\begin{proof}
Suppose $L$ is $(\Theta_{n},\mathfrak{g})$-graded. Let $\mathfrak{k}\cong so_{n}$
be a naturally embedded subalgebra of $\mathfrak{g}\cong sl_{n}$.
Then $\mathfrak{k}$ is of type $B_{r}$ (if $n$ is odd) or $D_{r}$
(if $n$ is even) for $r=\lfloor\frac{n}{2}\rfloor$. Note that $sl_{n}$
is $(BC_{r}\cup\{0\},\mathfrak{k})$-graded. By Proposition \ref{Th-transtivity of gradua},
we only need to show that the set of all weights of the $\mathfrak{k}$-module
$L$ is a subset of $BC_{r}\cup\{0\}$. Using Lemma \ref{restrection},
we get, as required,
\[
\mathcal{W}(L\downarrow\mathfrak{k})=\mathcal{W}(L\downarrow\mathfrak{g})\downarrow\mathfrak{k}\subseteq\Theta_{n}\downarrow\mathfrak{k}=\Gamma_{\mathfrak{g}}\downarrow\mathfrak{k}=\Gamma_{\mathfrak{k}}=BC_{r}\cup\{0\}.
\]
 
\end{proof}
\begin{rem}
\label{rem:finer} Suppose $L$ is $(\Theta_{n},\mathfrak{g})$-graded
($n\ge5$). Let $\mathfrak{k}\cong so_{n}$ be a naturally embedded
subalgebra of $\mathfrak{g}\cong sl_{n}$. As shown in the proof of
Theorem \ref{bc}, the algebra $L$ is $BC_{r}$-graded with respect
to the grading subalgebra $\mathfrak{k}$ with $r=\lfloor\frac{n}{2}\rfloor.$
The general theory of $BC_{r}$-graded Lie algebras gives multiplication
structure of $L$ in terms of $\mathfrak{k}$-decomposition components.
We are going to show that the multiplication structure of $L$ as
an $(\Theta_{n},\mathfrak{g})$-graded algebra is ``finer'' and
more specific. Let $V_{\mathfrak{k}}(\lambda)$ denote the simple
$\mathfrak{k}$-module with highest weight $\lambda$. We have 
\begin{eqnarray}
 &  & V_{\mathfrak{g}}(\omega_{1})\downarrow_{\mathfrak{k}}\cong V_{\mathfrak{g}}(\omega_{n})\downarrow_{\mathfrak{k}}\cong V_{\mathfrak{k}},\;V_{\mathfrak{g}}(2\omega_{1})\downarrow_{\mathfrak{k}}\cong V_{\mathfrak{g}}(2\omega_{n})\downarrow_{\mathfrak{k}}\cong\mathfrak{s}+T,\nonumber \\
 &  & V_{\mathfrak{g}}(\omega_{2})\downarrow_{\mathfrak{k}}\cong V_{\mathfrak{g}}(\omega_{n-1})\downarrow_{\mathfrak{k}}\cong\mathfrak{k},\;V_{\mathfrak{g}}(\omega_{1}+\omega_{n}))\downarrow_{\mathfrak{k}}\cong\mathfrak{k}+\mathfrak{s}\label{eq:rem}
\end{eqnarray}
where $T=V_{\mathfrak{k}}(0)$, $\mathfrak{k}=V_{\mathfrak{k}}(\omega_{2})$,
$\mathfrak{s}=V_{\mathfrak{k}}(2\omega_{1})$ and $V_{\mathfrak{k}}=V_{\mathfrak{k}}(\omega_{1}).$
By combining ($\ref{eq:drezh}$) and (\ref{eq:rem}), we can rewrite
$L$ as a $\mathfrak{k}$-module as follows: 
\[
L=(\mathfrak{k}\otimes(A\oplus E\oplus E'))\oplus(\mathfrak{s}\otimes(A\oplus C\oplus C'))\oplus(V_{\mathfrak{k}}\otimes(B\oplus B'))\oplus D'
\]
where $D'=(T\otimes(C\oplus C'))\oplus D$. If we wish to calculate
the product $[\mathfrak{s}\otimes C,\mathfrak{s}\otimes C]$ in $L$
using $BC_{r}$-grading structure then we can only say that
\[
[\mathfrak{s}\otimes C,\mathfrak{s}\otimes C]\subseteq(\mathfrak{k}\otimes(A\oplus E\oplus E'))\oplus(\mathfrak{s}\otimes(A\oplus C\oplus C'))\oplus D'.
\]
On the other hand, $\Theta_{n}$-grading structure (see Table \ref{t1}
) implies that $[\mathfrak{s}\otimes C,\mathfrak{s}\otimes C]=0$. 
\end{rem}

\begin{thm}
\label{bc-converse} Let $L$ be $BC_{r}$-graded for some integer
$r\ge2$. Then $L$ is $\Theta_{r}$-graded.
\end{thm}

\begin{proof}
Suppose $L$ is $BC_{r}$-graded with grading subalgebra $\mathfrak{k}$
of type $B_{r}$, $C_{r}$, or $D_{r}$. Let $\mathfrak{g}\cong sl_{r}$
be a naturally embedded subalgebra of $\mathfrak{k}$. It is easy
to see that $\mathfrak{k}$ is $(\Theta_{r},\mathfrak{g})$-graded.
By Proposition \ref{Th-transtivity of gradua}, we only need to show
that the set of all weights of the $\mathfrak{g}$-module $L$ is
a subset of $\Theta_{r}$. Using Lemma \ref{restrection}, we get,
as required,

\[
\mathcal{W}(L\downarrow\mathfrak{g})=\mathcal{W}(L\downarrow\mathfrak{k})\downarrow\mathfrak{g}\subseteq BC_{r}\cup\{0\}\downarrow\mathfrak{g}=\Gamma_{\mathfrak{k}}\downarrow\mathfrak{g}=\Gamma_{\mathfrak{g}}=\Theta_{r}.
\]
\end{proof}
\begin{rem}
\label{BCgraded_is_Agraded} Let $L$ be as in Theorem \ref{bc-converse}
and $r=5,6$. Then one can easily check that the conditions (\ref{eq:MainAssumptions})
hold, see \cite[Proposition 3.2.7]{yaseen2018generalized}.
\end{rem}

\begin{rem}
Theorems \ref{bc} and \ref{bc-converse} describe the relationship
between $\Theta_{n}$-graded and $BC_{n}$-graded Lie algebras. Even
though there are some similarities between the two theories, we consider
our approach as more natural and universal, which sheds more light
on the structure of weight-graded algebras. Our grading subalgebra
is $sl_{n}$. It is the most basic and natural simple Lie algebra
and appears often as a subalgebra (see for example Theorem \ref{Th-per is sl2}).
The $\Theta_{n}$-grading involves more irreducible modules and yields
a ``finer'' multiplicative structure on a $\Theta_{n}$-graded Lie
algebra $L$ because of the larger number of components in the decomposition
of $L$ (see Remark \ref{rem:finer}). As a result, the coordinate
algebra $\mathfrak{b}$ of $L$ has more components and finer structure
itself, see Theorem \ref{structure}. 
\end{rem}

\section{\label{sec:Multiplication-in--graded}Multiplication in $\Theta_{n}$-graded
Lie algebras}

In this section we describe the multiplicative structure of $(\Theta_{n},sl_{n})$-graded
Lie algebras ($n\ge5$). Recall that $\Theta_{n}=\{0,\pm\varepsilon_{i}\pm\varepsilon_{j},\pm\varepsilon_{i},\pm2\varepsilon_{i}\mid1\leq i\neq j\leq n\}$
were $\{\varepsilon_{1},\dots,\varepsilon_{n}\}$ is the set of the
weights of the natural $sl_{n}$-module. We denote by $\Theta_{n}^{+}$
the set of the dominant weights in $\Theta_{n}$. Thus,
\begin{align*}
\Theta_{n}^{+} & =\{\omega_{1}+\omega_{n-1}=\varepsilon_{1}-\varepsilon_{n},\ \omega_{1}=\varepsilon_{1},\ \omega_{n-1}=-\varepsilon_{n},\\
 & \ \ \ \ \ \ 2\omega_{1}=2\varepsilon_{1},\ 2\omega_{n-1}=-2\varepsilon_{n},\ \omega_{2}=\varepsilon_{1}+\varepsilon_{2},\ \omega_{n-2}=-\varepsilon_{n-1}-\varepsilon_{n},\ 0\}.
\end{align*}
These are the highest weights of the modules $\mathfrak{g}$, $V$,
$V'$, $S$, $S'$, $\Lambda,$ $\Lambda'$ and $T$, respectively.
We will use the same symbol $\Theta_{n}^{+}$ to denote the set of
these modules. We fix a base $\Pi=\{\alpha_{i}=\varepsilon_{i}-\varepsilon_{i+1}\mid i=1,\cdots,n-1\}$
of simple roots for the root system $A_{n-1}=\{\varepsilon_{i}-\varepsilon_{j}\mid1\leq i\neq j\leq n\}.$
Let $L$ be a $\Theta_{n}$-graded Lie algebra and let $\mathfrak{g}$
be the grading subalgebra of $L$ of type $\Delta=A_{n-1}$ with $n\geqq5$.
We identify $\mathfrak{g}$ with the matrix algebra $sl_{n}$. By
(\ref{eq:drezh}), the $\mathfrak{g}$-module $L$ is decomposed as
\[
L\cong(\mathfrak{g}\otimes A)\oplus(V\otimes B)\oplus(V'\otimes B')\oplus(S\otimes C)\oplus(S'\otimes C')\oplus(\Lambda\otimes E)\oplus(\Lambda'\otimes E')\oplus D.
\]
for some vector spaces $A,B,B',C,C',E,E'$ and the centralizer $D$
of $\mathfrak{g}$ in $L$. Alternatively, these spaces can also be
viewed as the corresponding $\mathfrak{g}$-mod Hom-spaces: $A=\hom(\mathfrak{g},L)$,
$B=\hom(V,L)$, etc, so for each simple $\mathfrak{g}$-module $M$,
the space $M\otimes\hom(M,L)$ is canonically identified with the
$M$-isotypic component of $L$ via the evaluation map 
\begin{equation}
M\otimes\hom(M,L)\rightharpoondown L,\,\;m\otimes\varphi\mapsto\varphi(m).\label{eq:MHom}
\end{equation}

\begin{defn}
\label{Lem-Vdual} (1) We identify the $\mathfrak{g}$-modules $V$
and $V'$ with the space $\mathbb{F}^{n}$ of column vectors with
the following actions: 
\begin{align*}
 & x.v=xv\text{ \text{ for }}x\in sl_{n},\,v\in V,\\
 & x.v'=-x^{t}v'\text{ for }x\in sl_{n},\,v'\in V'.
\end{align*}
(2) We identify $S$ and $S'$ (resp. $\Lambda$ and $\Lambda'$)
with symmetric (resp. skew-symmetric) $n\times n$ matrices. Then
$S$, $S'$, $\Lambda$ and $\Lambda'$ are $\mathfrak{g}$-modules
under the actions: 
\begin{align*}
 & x.s=xs+sx^{t}\text{ for }x\in sl_{n},\,s\in S,\\
 & x.\lambda=x\lambda+\lambda x^{t}\text{ for }x\in sl_{n},\,\lambda\in\Lambda,\\
 & x.s'=-s'x-x^{t}s'\text{ for }x\in sl_{n},\,s'\in S',\\
 & x.\lambda'=-\lambda'x-x^{t}\lambda'\text{ for }x\in sl_{n},\,\lambda'\in\Lambda'.
\end{align*}
\end{defn}

Since the subalgebra $\mathfrak{g}$ of $L$ is a $\mathfrak{g}$-submodule,
there exists a distinguished element $1$ of $A$ such that $\mathfrak{g}=\mathfrak{g}\otimes1$.
In particular, 
\begin{equation}
[x\otimes1,y\otimes b]=x.y\otimes b\label{eq:iden 1}
\end{equation}
 where $x\otimes1$ is in $\mathfrak{g}\otimes1$, $y\otimes b$ belongs
to one of the components in (\ref{eq:drezh}) except $D$, and $x.y$
is as in Definition \ref{Lem-Vdual}.

Let $\Theta(M)$ be the $\Theta$-component of $M$, i.e. the sum
of all simple submodules of $M$ with highest weights in $\Theta_{n}^{+}$.
In order to describe multiplication in $L$ we need to calculate first
the $\Theta$-components of the tensor products of the modules in
$\Theta_{n}^{+}$. For the larger ranks, the decompositions are easily
derived from \cite[Cor.3.5]{kumar2010tensor}, \cite[Proposition 3.2]{kumar2010tensor},
\cite[A-2]{seligman1976rational} and the stability results in \cite[Cor. 6.22 and 7.2]{benkart1990stability},
see \cite[Section 3.4]{yaseen2018generalized} for details. The remaining
small cases are easily verified with a computer program (such as LiE).
In Table \ref{t1} below we describe $\Theta$-components of all tensor
product decompositions for the modules in $\Theta_{n}^{+}$ $(n\geq5)$.
If the cell in row $X$ and column $Y$ contains $Z$ this means that
$\Theta(X\otimes Y)=\Theta(Y\otimes X)\cong Z$. 

\begin{table}[H]
\begin{tabular}{|c|r@{\extracolsep{0pt}.}l|c|c|c|c|c|c|}
\hline 
$\otimes$ & \multicolumn{2}{c|}{$\mathfrak{g}$} & $S$ & $\Lambda$ & $S'$ & $\Lambda'$ & $V$ & $V'$\tabularnewline
\hline 
\hline 
$\mathfrak{g}$ & \multicolumn{2}{c|}{$\mathfrak{g+\mathfrak{g}}+T$} & $S+\Lambda$ & $\begin{array}{c}
S+\Lambda\end{array}$ & $S'+\Lambda'$ & $S'+\Lambda'$ & $V$ & $V'$\tabularnewline
\hline 
$S$ & \multicolumn{2}{c|}{} & $0$ & $0$ & $\mathfrak{g}+T$ & $\mathfrak{g}$ & $0$ & $V$\tabularnewline
\hline 
$\Lambda$ & \multicolumn{2}{c|}{} &  & $\begin{array}{c}
0\;(n\ge7)\\
\Lambda'\;(n=6)\\
V'\;(n=5)
\end{array}$ & $\mathfrak{g}$ & $\mathfrak{g}+T$ & $\begin{array}{c}
0\;(n\ge6)\\
\Lambda'\;(n=5)
\end{array}$ & $V$\tabularnewline
\hline 
$S'$ & \multicolumn{2}{c|}{} &  &  & $0$ & $0$ & $V'$ & $0$\tabularnewline
\hline 
$\Lambda'$ & \multicolumn{2}{c|}{} &  &  &  & $\begin{array}{c}
0\;(n\ge7)\\
\Lambda\;(n=6)\\
V\;(n=5)
\end{array}$ & $V'$ & $\begin{array}{c}
0\;(n\ge6)\\
\Lambda\;(n=5)
\end{array}$\tabularnewline
\hline 
$V$ & \multicolumn{2}{c|}{} &  &  &  &  & $S+\Lambda$ & $\mathfrak{g}+T$\tabularnewline
\hline 
$V'$ & \multicolumn{2}{c|}{} &  &  &  &  &  & $S'+\Lambda'$\tabularnewline
\hline 
\end{tabular}

\caption{$\Theta$-component of tensor product decompositions for $sl_{n}$
$(n\protect\geq5)$}

\label{t1}
\end{table}

Let $L$ be an $\Theta_{n}$-graded Lie algebra and let $\mathfrak{g}$
be the grading subalgebra of $L$. Suppose that $n\ge7$ or $n=5,6$
and the conditions (\ref{eq:MainAssumptions}) hold. In (\ref{t2})
we list bases for all non-zero $\mathfrak{g}$-module homomorphism
spaces $\hom(X\otimes Y,Z)$ (we simply write $(X\otimes Y,Z)$) where
$X,Y,Z\in\{\mathfrak{g},V,V',S,\Lambda,S',\Lambda',T\}$ and $X$
and $Y$ are both non-trivial. Note that all of them are $1$-dimensional
except the first one (which is 2-dimensional). 

\begin{flalign}
 & (\mathfrak{g}\otimes\mathfrak{g},\mathfrak{g})=\langle x\otimes y\mapsto xy-yx,\ x\otimes y\mapsto xy+yx-\frac{2}{n}\tra(xy)I\rangle,\label{t2}\\
 & (V\otimes V',\mathfrak{g})=\langle u\otimes v'\mapsto uv'^{t}-\frac{\tra(uv'^{t})}{n}I\rangle,\nonumber \\
 & (S\otimes S',\mathfrak{g})=\langle s\otimes s'\mapsto ss'-\frac{\tra(ss')}{n}I\rangle,\quad(\Lambda\otimes\Lambda',\mathfrak{g})=\langle\lambda\otimes\lambda'\mapsto\lambda\lambda'-\frac{\tra(\lambda\lambda')}{n}I\rangle,\nonumber \\
 & (S\otimes\Lambda',\mathfrak{g})=\langle s\otimes\lambda'\mapsto s\lambda'\rangle,\quad(S'\otimes\Lambda,\mathfrak{g})=\langle s'\otimes\lambda\mapsto s'\lambda\rangle,\nonumber \\
 & (\mathfrak{g}\otimes V,V)=\langle x\otimes v\mapsto xv\rangle,\quad(\mathfrak{g}\otimes V',V')=\langle x\otimes v'\mapsto xv'\rangle,\nonumber \\
 & (\Lambda\otimes V',V)=\langle\lambda\otimes v'\mapsto\lambda v'\rangle,\quad(\Lambda'\otimes V',V')=\langle\lambda'\otimes v'\mapsto\lambda'v'\rangle,\nonumber \\
 & (S\otimes V',V)=\langle s\otimes v'\mapsto sv'\rangle,\quad(S'\otimes V,V')=\langle s'\otimes v\mapsto s'v\rangle,\nonumber \\
 & (\mathfrak{g}\otimes S,S)=\langle x\otimes s\mapsto xs+sx^{t}\rangle,\quad(\mathfrak{g}\otimes\Lambda,\Lambda)=\langle x\otimes\lambda\mapsto x\lambda+\lambda x^{t}\rangle,\nonumber \\
 & (V\otimes V,S)=\langle u\otimes v\mapsto uv^{t}+vu^{t}\rangle,\quad(V'\otimes V',S')=\langle u'\otimes v'\mapsto u'v'^{t}+v'u'^{t}\rangle,\nonumber \\
 & (V\otimes V,\Lambda)=\langle u\otimes v\mapsto uv^{t}-vu^{t}\rangle,\quad(V'\otimes V',\Lambda')=\langle u'\otimes v'\mapsto u'v'^{t}-v'u'^{t}\rangle,\nonumber \\
 & (\mathfrak{g}\otimes\Lambda,S)=\langle x\otimes\lambda\mapsto x\lambda-\lambda x^{t}\rangle,\quad(\Lambda'\mathfrak{\otimes g},S')=\langle\lambda'\otimes x\mapsto\lambda'x-x^{t}\lambda'\rangle,\nonumber \\
 & (\mathfrak{g}\otimes S,\Lambda)=\langle x\otimes s\mapsto xs-sx^{t}\rangle,\quad(S'\mathfrak{\otimes g},S')=\langle s'\otimes x\mapsto s'x+x^{t}s'\rangle,\nonumber \\
 & (\Lambda'\mathfrak{\otimes g},\Lambda')=\langle\lambda'\otimes x\mapsto\lambda'x+x^{t}\lambda'\rangle,\quad(S'\mathfrak{\otimes g},\Lambda')=\langle s'\otimes x\mapsto s'x-x^{t}s'\rangle,\nonumber \\
 & (\mathfrak{g}\otimes\mathfrak{g},T)=\langle x_{1}\otimes x_{2}\mapsto\frac{1}{n}\tra(x_{1}x_{2})\rangle,\quad(V'\otimes V,T)=\langle v'\otimes u\mapsto\frac{1}{n}\tra(uv'^{t})\rangle,\nonumber \\
 & (S\otimes S',T)=\langle s\otimes s'\mapsto\frac{1}{n}\tra(ss')\rangle,\quad(\Lambda\otimes\Lambda',T)=\langle\lambda\otimes\lambda'\mapsto\frac{1}{n}\tra(\lambda\lambda')\rangle.\nonumber 
\end{flalign}

The Lie algebra structure on the decomposition (\ref{eq:drezh}) induces
certain bilinear maps among the spaces $A,B,B',C,C',E,E',D$. Indeed,
denote the irreducible modules and the corresponding spaces by $M_{1},\dots,M_{8}$
and $H_{1},\dots,H_{8}$, respectively. Then $L=\bigoplus_{i=1}^{8}M_{i}\otimes H_{i}$
and $H_{i}=\hom(M_{i},L)$, see (\ref{eq:MHom}). The Lie product
on $L$ can be identified with an element of $\hom(L\otimes L,L)$.
Since any homomorphisms between non-isomorphic irreducible $\mathfrak{g}$-modules
are zero, the product is actually an element of $\hom(\Theta(L\otimes L),L)$
where $\Theta(L\otimes L)$ is the sum of all irreducible $\mathfrak{g}$-submodules
of $L\otimes L$ isomorphic to one of $M_{1},\dots,M_{8}$. The $\mathfrak{g}$-module
$L\otimes L$ is decomposed as $L\otimes L=\bigoplus_{i,j=1}^{8}M_{i}\otimes M_{j}\otimes H_{i}\otimes H_{j}$
and the $\Theta$-component of $L\otimes L$ can be found as 
\[
\Theta(L\otimes L)=\bigoplus_{k=1}^{8}M_{k}\otimes\Hom_{\ffg}(L\otimes L,M_{k})=\bigoplus_{k=1}^{8}M_{k}\otimes\left(\oplus_{i,j=1}^{8}M_{ij}^{k}\otimes H_{i}\otimes H_{j}\right)
\]
 where $M_{ij}^{k}=\Hom_{\ffg}(M_{i}\otimes M_{j},M_{k})$. Then the
Lie bracket on $L$ is an element $\mu$ of the space
\begin{eqnarray*}
\hom(\Theta(L\otimes L),L) & = & \bigoplus_{k=1}^{8}\Hom_{\bbF}\left(\oplus_{i,j=1}^{8}M_{ij}^{k}\otimes H_{i}\otimes H_{j},H_{k}\right)\\
 & = & \bigoplus_{i,j,k=1}^{8}\Hom_{\bbF}\left(M_{ij}^{k}\otimes H_{i}\otimes H_{j},H_{k}\right)\\
 & = & \bigoplus_{i,j,k=1}^{8}\Hom_{\bbF}\left(M_{ij}^{k},\Hom_{\bbF}(H_{i}\otimes H_{j},H_{k})\right).
\end{eqnarray*}
Denote by $\{b_{1}^{kij},b_{2}^{kij}\dots\}$ the basis of the space
$\hom(M_{i}\otimes M_{j},M_{k})$ as in (\ref{t2}). Then there exist
unique elements $\chi_{1}^{kij},\chi_{2}^{kij},\dots$ in $\Hom_{\bbF}(H_{i}\otimes H_{j},H_{k})$
(the images of $b_{1}^{kij},b_{2}^{kij}\dots$) which correspond to
multiplication $\mu$ on $L$. These elements $\chi_{s}^{kij}\in\Hom_{\bbF}(H_{i}\otimes H_{j},H_{k})$
are the claimed bilinear maps $H_{i}\times H_{j}\rightarrow H_{k}$.

In Table \ref{t3}, if the cell in row $X$ and column $Y$ contains
$Z$, this means that there is a bilinear map $X\otimes Y\rightarrow Z$
given by $x\otimes y\mapsto(x,y)_{Z}$. For simplicity of notation,
we will write $dy$ instead of $(d,y)_{D}$ if $X=Z=D$ and we will
write $\langle x,y\rangle$ instead of $(x,y)_{D}$ if $X,Y\neq D$
and $Z=D.$ In the case $X=Y=Z=A$, we have two bilinear products
$a_{1}\otimes a_{2}\mapsto a_{1}\circ a_{2}$ and $a_{1}\otimes a_{2}\mapsto[a_{1},a_{2}]$
for $a_{1},a_{2}\in A$. Note that some of the cells are empty. The
corresponding products $X\otimes Y\rightarrow Z$ will be defined
later by extending the existing maps $Y\otimes X\rightarrow Z$. This
will make the table symmetric.

\begin{table}[H]
\begin{tabular}{|c|r@{\extracolsep{0pt}.}l|c|c|c|c|c|c|c|}
\hline 
$.$ & \multicolumn{2}{c|}{$A$} & $B$ & $B'$ & $C$ & $C'$ & $E$ & $E'$ & $D$\tabularnewline
\hline 
\hline 
$A$ & \multicolumn{2}{c|}{$\begin{array}{c}
(A,\circ,[\ ]),\,D\end{array}$} & $B$ &  & $C,E$ &  & $C,E$ &  & \tabularnewline
\hline 
$B$ & \multicolumn{2}{c|}{} & $C,E$ & $A,D$ & $0$ &  & $0$ &  & \tabularnewline
\hline 
$B'$ & \multicolumn{2}{c|}{$A$} &  & $C',E'$ & $B$ & $0$ & $B$ & $0$ & \tabularnewline
\hline 
$C$ & \multicolumn{2}{c|}{} & $0$ &  & $0$ & $A,D$ & $0$ & $A$ & \tabularnewline
\hline 
$C'$ & \multicolumn{2}{c|}{$C',E'$} & $B'$ & $0$ &  & $0$ & $A$ & $0$ & \tabularnewline
\hline 
$E$ & \multicolumn{2}{c|}{} & $0$ &  & $0$ &  & $0$ & $A,D$ & \tabularnewline
\hline 
$E'$ & \multicolumn{2}{c|}{$C',E'$} & $B'$ & $0$ &  & $0$ &  & $0$ & \tabularnewline
\hline 
$D$ & \multicolumn{2}{c|}{$A$} & $B$ & $B'$ & $C$ & $C'$ & $E$ & $E'$ & $D$\tabularnewline
\hline 
\end{tabular}

\caption{Bilinear products}

\label{t3}
\end{table}

Let $x$ and $y$ be $n\times n$ matrices. We will use the following
products: $[x,y]:=xy-yx$, $x\circ y:=xy+yx-\frac{2}{n}\tra(xy)I$,
$x\diamond y:=xy+yx$ and $(x\mid y):=\frac{1}{n}\tra(xy)$. 

Following the methods in \cite{seligman1976rational,berman1992lie,benkart1996lie,allison2002lie}
and using the results of (\ref{t2}), Tables \ref{t1} and \ref{t3},
we may suppose that the multiplication in $L$ is given as follows
(see \cite[Section 3.4]{yaseen2018generalized} for a sample argument).
For all $x,y\in sl_{n}$, $u,v\in V$, $u',v'\in V'$, $s\in S$,
$\lambda\in\Lambda$, $s'\in S'$, $\lambda'\in\Lambda'$ and for
all $a,a_{1},a_{2}\in A$, $b,b_{1},b_{2}\in B$, $b',b_{1}',b_{2}'\in B'$,
$c\in C$, $c'\in C'$, $e\in E$, $e'\in E'$ and $d,d_{1},d_{2}\in D$,

\begin{flalign}
 & [x\otimes a_{1},y\otimes a_{2}]=(x\circ y)\otimes\frac{[a_{1},a_{2}]}{2}+[x,y]\otimes\frac{a_{1}\circ a_{2}}{2}+(x\mid y)\langle a_{1},a_{2}\rangle,\label{main for}\\
 & [u\otimes b,v'\otimes b']=(uv'^{t}-\frac{\tra(uv'^{t})}{n}I)\otimes(b,b')_{A}+\frac{2}{n}\tra(uv'^{t})\langle b,b'\rangle=-[v'\otimes b',u\otimes b],\nonumber \\
 & [s\otimes c,s'\otimes c']=(ss'-(s\mid s')I)\otimes(c,c')_{A}+(s\mid s')\langle c,c'\rangle=-[s'\otimes c',s\otimes c],\nonumber \\
 & [\lambda\otimes e,\lambda'\otimes e']=(\lambda\lambda'-(\lambda\mid\lambda')I)\otimes(e,e')_{A}+(\lambda\mid\lambda')\langle e,e'\rangle=-[\lambda'\otimes e',\lambda\otimes e],\nonumber \\
 & [u\otimes b_{1},v\otimes b_{2}]=(uv^{t}+vu^{t})\otimes\frac{(b_{1},b_{2})_{C}}{2}+(uv^{t}-vu^{t})\otimes\frac{(b_{1},b_{2})_{E}}{2},\nonumber \\
 & [u'\otimes b'_{1},v'\otimes b'_{2}]=(u'v'^{t}+v'u'^{t})\otimes\frac{(b'_{1},b'_{2})_{C'}}{2}+(u'v'^{t}-v'u'^{t})\otimes\frac{(b'_{1},b'_{2})_{E'}}{2},\nonumber \\
 & [x\otimes a,s\otimes c]=(xs+sx^{t})\otimes\frac{(a,c)_{C}}{2}+(xs-sx^{t})\otimes\frac{(a,c)_{E}}{2}=-[s\otimes c,x\otimes a],\nonumber \\
 & [x\otimes a,\lambda\otimes e]=(x\lambda+\lambda x^{t})\otimes\frac{(a,e)_{E}}{2}+(x\lambda-\lambda x^{t})\otimes\frac{(a,e)_{C}}{2}=-[\lambda\otimes e,x\otimes a],\nonumber \\
 & [s'\otimes c',x\otimes a]=(s'x+x^{t}s')\otimes\frac{(c',a)_{C'}}{2}+(s'x-x^{t}s')\otimes\frac{(c',a)_{E'}}{2}=-[x\otimes a,s'\otimes c'],\nonumber \\
 & [\lambda'\otimes e',x\otimes a]=(\lambda'x+x^{t}\lambda')\otimes\frac{(e',a)_{E'}}{2}+(\lambda'x-x^{t}\lambda')\otimes\frac{(e',a)_{C'}}{2}=-[x\otimes a,\lambda'\otimes e'],\nonumber \\
 & [s\otimes c,\lambda'\otimes e']=s\lambda'\otimes(c,e')_{A}=-[\lambda'\otimes e',s\otimes c],\nonumber \\
 & [s'\otimes c',\lambda\otimes e]=s'\lambda\otimes(c',e)_{A}=-[\lambda\otimes e,s'\otimes c'],\nonumber \\
 & [x\otimes a,u\otimes b]=xu\otimes(a,b)_{B}=-[u\otimes b,x\otimes a],\nonumber \\
 & [s'\otimes c',u\otimes b]=s'u\otimes(c',b)_{B'}=-[u\otimes b,s'\otimes c'],\nonumber \\
 & [\lambda'\otimes e',u\otimes b]=\lambda'u\otimes(e',b)_{B'}=-[u\otimes b,\lambda'\otimes e'],\nonumber \\
 & [u'\otimes b',x\otimes a]=x^{t}u'\otimes(b',a)_{B'}=-[x\otimes a,u'\otimes b'],\nonumber \\
 & [u'\otimes b',s\otimes c]=su'\otimes(b',c)_{B}=-[s\otimes c,u'\otimes b'],\nonumber \\
 & [u'\otimes b',\lambda\otimes e]=-\lambda u'\otimes(b',e)_{B}=-[\lambda\otimes e,u'\otimes b'],\nonumber \\
 & [d,x\otimes a]=x\otimes da=-[x\otimes a,d],\quad\quad\ [d,u\otimes b]=u\otimes db=-[u\otimes b,d],\nonumber \\
 & [d,s\otimes c]=s\otimes dc=-[s\otimes c,d],\quad\quad\quad[d,\lambda\otimes e]=\lambda\otimes de=-[\lambda\otimes e,d],\nonumber \\
 & [d,s'\otimes c']=s'\otimes dc'=-[s'\otimes c',d],\quad\ [d,u'\otimes b']=u'\otimes db'=-[u'\otimes b',d],\nonumber \\
 & [d,\lambda'\otimes e']=\lambda'\otimes de'=-[\lambda'\otimes e',d],\quad[d_{1},d_{2}]\in D,\nonumber 
\end{flalign}
All other products of the homogeneous components of the decomposition
(\ref{eq:drezh}) are zero. 

\section{\label{chap:4} The coordinate algebra of a $\Theta_{n}$-graded
Lie algebra}

Let $L$ be an $\Theta_{n}$-graded Lie algebra and let $\mathfrak{g}\cong sl_{n}$
be the grading subalgebra of $L$. Assume that $n\ge7$ or $n=5,6$
and the conditions (\ref{eq:MainAssumptions}) hold. Let $\mathfrak{g}^{\pm}=\{x\in sl_{n}\mid x^{t}=\pm x\}$.
Then $\mathfrak{g}\otimes A=(\mathfrak{g^{+}}\oplus\mathfrak{g^{-}})\otimes A=(\mathfrak{g^{+}}\otimes A)\oplus(\mathfrak{g^{-}}\otimes A)=(\mathfrak{g^{+}}\otimes A^{-})\oplus(\mathfrak{g}^{-}\otimes A^{+})$
where $A^{\pm}$ is a copy of the vector space $A$. We denote by
$a^{\pm}$ the image of $a\in A$ in the space $A^{\pm}$. Recall
that we identify $\mathfrak{g}$ with $\mathfrak{g}\otimes1$ where
$1$ is a distinguished element of $A$ and we denote $\mathfrak{a}:=A^{+}\oplus A^{-}\oplus C\oplus E\oplus C'\oplus E'$
and $\mathfrak{b}:=\mathfrak{a}\oplus B\oplus B'$. We show that the
product in $L$ induces a unital algebra structure on both $\mathfrak{a}$
and $\mathfrak{b}$. We prove that $\mathfrak{a}$ is an associative
subalgebra of $\mathfrak{b}$ and $\mathfrak{b}$ (which is not associative
in general) has an involution $\eta$ whose symmetric and skew-symmetric
elements are $A^{+}\oplus E\oplus E'\oplus B\oplus B'$ and $A^{-}\oplus C\oplus C'$.
Let $x$ and $y$ be $n\times n$ matrices. Recall the products $[x,y]:=xy-yx$,
$x\circ y:=xy+yx-\frac{2}{n}\tra(xy)I$, $x\diamond y:=xy+yx$ and
$(x\mid y):=\frac{1}{n}\tra(xy)$. 

\subsection{Unital associative algebra $\mathfrak{a}$}

We are going to define Lie and Jordan multiplication on $\mathfrak{a}$
by extending the bilinear products given in Table \ref{t4} in a natural
way. It can be shown that all products $(\alpha_{1},\alpha_{2})_{Z}$
with $\alpha_{1},\alpha_{2}\in\mathfrak{a}$ are either symmetric
or skew-symmetric. This is why we will write $(\alpha_{1}\circ\alpha_{2})_{Z}$
or $[\alpha_{1},\alpha_{2}]_{Z}$, respectively, instead of $(\alpha_{1},\alpha_{2})_{Z}$.
The aim of this subsection is to show that $\mathfrak{a}$ is an associative
algebra with respect to the new multiplication given by $\alpha_{1}\alpha_{2}:=\frac{[\alpha_{1},\alpha_{2}]}{2}+\frac{\alpha_{1}\circ\alpha_{2}}{2}.$
\begin{rem}
\label{FF} In this remark we rewrite some of the products in (\ref{main for})
in terms of symmetric and skew-symmetric elements. Note that every
$x\in\mathfrak{g}$ is uniquely decomposed as $x=x^{+}+x^{-}$ where
$x^{+}=\frac{x+x^{t}}{2}\in\mathfrak{g^{+}}$ and $x^{-}=\frac{x-x^{t}}{2}\in\mathfrak{g^{-}}$.

(a) Let $x_{1}^{+}\otimes a_{1}^{-},\,x_{2}^{+}\otimes a_{2}^{-}\in\mathfrak{g^{+}}\otimes A^{-}$
and $x_{1}^{-}\otimes a_{1}^{+},\,x_{2}^{-}\otimes a_{2}^{+}\in\mathfrak{g^{-}}\otimes A^{+}$.
Since
\begin{align*}
[x\otimes a_{1},y\otimes a_{2}] & =x\circ y\otimes\frac{[a_{1},a_{2}]}{2}+[x,y]\otimes\frac{a_{1}\circ a_{2}}{2}+(x\mid y)\langle a_{1},a_{2}\rangle
\end{align*}
and $(x_{1}^{+}\mid x_{1}^{-})=\frac{1}{n}\tra(x_{1}^{+}x_{1}^{-})=0$,
we have
\begin{align*}
[x_{1}^{+}\otimes a_{1}^{-},x_{2}^{+}\otimes a_{2}^{-}] & =x_{1}^{+}\circ x_{2}^{+}\otimes\frac{[a_{1}^{-},a_{2}^{-}]_{A^{-}}}{2}+[x_{1}^{+},x_{2}^{+}]\otimes\frac{(a_{1}^{-}\circ a_{2}^{-})_{A^{+}}}{2}+(x_{1}^{+}\mid x_{2}^{+})\langle a_{1}^{-},a_{2}^{-}\rangle\text{,}\\{}
[x_{1}^{-}\otimes a_{1}^{+},x_{2}^{-}\otimes a_{2}^{+}] & =x_{1}^{-}\circ x_{2}^{-}\otimes\frac{[a_{1}^{+},a_{2}^{+}]_{A^{-}}}{2}+[x_{1}^{-},x_{2}^{-}]\otimes\frac{(a_{1}^{+}\circ a_{2}^{+})_{A^{+}}}{2}+(x_{1}^{-}\mid x_{2}^{-})\langle a_{1}^{+},a_{2}^{+}\rangle\text{,}\\{}
[x_{1}^{+}\otimes a_{1}^{-},x_{1}^{-}\otimes a_{1}^{+}] & =x_{1}^{+}\diamond x_{1}^{-}\otimes\frac{[a_{1}^{-},a_{1}^{+}]_{A^{+}}}{2}+[x_{1}^{+},x_{1}^{-}]\otimes\frac{(a_{1}^{-}\circ a_{1}^{+})_{A^{-}}}{2}\text{.}
\end{align*}

(b) Let $s\otimes c\in S\otimes C$ and $\lambda\otimes e\in\Lambda\otimes E$.
Since 
\begin{alignat*}{1}
 & [x\otimes a,s\otimes c]=(xs+sx^{t})\otimes\frac{(a,c)_{C}}{2}+(xs-sx^{t})\otimes\frac{(a,c)_{E}}{2},\\
 & x^{+}s+s(x^{+})^{t}=x^{+}s+sx^{+}=x^{+}\circ s,\,\;x^{+}s-s(x^{+})^{t}=x^{+}s-sx^{+}=[x^{+},s],\\
 & x^{-}s+s(x^{-})^{t}=x^{-}s-sx^{-}=[x^{-},s],\,\;x^{-}s-s(x^{-})^{t}=x^{-}s+sx^{-}=x^{-}\circ s,
\end{alignat*}
we obtain
\begin{align*}
[x^{+}\otimes a^{-},s\otimes c] & =x^{+}\diamond s\otimes\frac{[a^{-},c]_{C}}{2}+[x^{+},s]\otimes\frac{(a^{-}\circ c)_{E}}{2},\\{}
[x^{-}\otimes a^{+},s\otimes c] & =x^{-}\diamond s\otimes\frac{[a^{+},c]_{E}}{2}+[x^{-},s]\otimes\frac{(a^{+}\circ c)_{C}}{2}.
\end{align*}
Similarly, we get
\begin{alignat*}{1}
[x^{+}\otimes a^{-},\lambda\otimes e] & =x^{+}\diamond\lambda\otimes\frac{[a^{-},e]_{E}}{2}+[x^{+},\lambda]\otimes\frac{(a^{-}\circ e)_{C}}{2},\\{}
[x^{-}\otimes a^{+},\lambda\otimes e] & =x^{-}\diamond\lambda\otimes\frac{[a^{+},e]_{C}}{2}+[x^{-},\lambda]\otimes\frac{(a^{+}\circ e)_{E}}{2}.
\end{alignat*}

(c) Let $s'\otimes c'\in S'\otimes C'$ and $\lambda'\otimes e'\in\Lambda'\otimes E'$.
Since
\begin{alignat*}{1}
[s'\otimes c',x\otimes a] & =(s'x+x^{t}s')\otimes\frac{(c',a)_{C'}}{2}+(s'x-x^{t}s')\otimes\frac{(c',a)_{E'}}{2},\\
s'x^{+}+(x^{+})^{t}s' & =s'\circ x^{+},\;s'x^{+}-(x^{+}){}^{t}s'=[s',x^{+}],\\
s'x^{-}+(x^{-})^{t}s' & =[s',x^{-}],\;s'x^{-}-(x^{-})^{t}s'=s'\circ x^{-},
\end{alignat*}
we get
\begin{alignat*}{1}
[s'\otimes c',x^{+}\otimes a^{-}] & =s'\diamond x^{+}\otimes\frac{[c',a^{-}]_{C'}}{2}+[s',x^{+}]\otimes\frac{(c'\circ a^{-})_{E'}}{2},\\{}
[s'\otimes c',x^{-}\otimes a^{+}] & =s'\diamond x^{-}\otimes\frac{[c',a^{+}]_{E'}}{2}+[s',x^{-}]\otimes\frac{(c'\circ a^{+})_{C'}}{2}.
\end{alignat*}
 Similarly, we get
\begin{align*}
[\lambda'\otimes e',x^{+}\otimes a^{-}] & =\lambda'\diamond x^{+}\otimes\frac{[e',a^{-}]_{E'}}{2}+[\lambda',x^{+}]\otimes\frac{(e'\circ a^{-})_{C'}}{2},\\{}
[\lambda'\otimes e',x^{-}\otimes a^{+}] & =\lambda'\diamond x^{-}\otimes\frac{[e',a^{+}]_{C'}}{2}+[\lambda',x^{-}]\otimes\frac{(e'\circ a^{+})_{E'}}{2}.
\end{align*}

(d) For any $x\otimes a\in\mathfrak{g}\otimes A$, we have $x\otimes a=\frac{(x+x^{t})}{2}\otimes a+\frac{(x-x^{t})}{2}\otimes a\in\mathfrak{g^{+}}\otimes A+\mathfrak{g^{-}}\otimes A.$
Since $[s\otimes c,s'\otimes c']=(ss'-(s\mid s')I)\otimes(c,c')_{A}+(s\mid s')\langle c,c'\rangle$,
$ss'-(s\mid s')I+(ss'-(s\mid s')I)^{t}=s\circ s'$ and $ss'-(s\mid s')I-(ss'-(s\mid s')I)^{t}=[s,s']$,
we get 
\[
[s\otimes c,s'\otimes c']=s\circ s'\otimes\frac{[c,c']_{A^{-}}}{2}+[s,s']\otimes\frac{(c\circ c')_{A^{+}}}{2}+(s\mid s')\langle c,c'\rangle.
\]
 Similary, we get
\[
[\lambda\otimes e,\lambda'\otimes e']=\lambda\circ\lambda'\otimes\frac{[e,e']_{A^{-}}}{2}+[\lambda,\lambda']\otimes\frac{(e\circ e')_{A^{+}}}{2}+(\lambda\mid\lambda')\langle e,e'\rangle.
\]
 Since $[s\otimes c,\lambda'\otimes e']=s\lambda'\otimes(c,e')_{A}$
and $s\lambda'+(s\lambda')^{t}=[s,\lambda'],$ $s\lambda'-(s\lambda')^{t}=s\diamond\lambda'$,
we get 
\[
[s\otimes c,\lambda'\otimes e']=s\diamond\lambda'\otimes\frac{[c,e']_{A^{+}}}{2}+[s,\lambda']\otimes\frac{(c\circ e')_{A^{-}}}{2}.
\]
 Similary, we get
\[
[s'\otimes c',\lambda\otimes e]=s'\diamond\lambda\otimes\frac{[c',e]_{A^{+}}}{2}+[s',\lambda]\otimes\frac{(c'\circ e)_{A^{-}}}{2}.
\]
\end{rem}

The mappings $\alpha\otimes\beta\mapsto(\alpha\circ\beta)_{Z_{1}}$
and $\alpha\otimes\beta\mapsto[\alpha,\beta]_{Z_{2}}$ can be extended
to $Y\otimes X$ in a consistent way by defining $(\beta\circ\alpha)_{Z_{1}}=(\alpha\circ\beta)_{Z_{1}}$
and $[\beta,\alpha]_{Z_{2}}=-[\alpha,\beta]_{Z_{2}}$. In Table \ref{t4}
below, if the cell in row $X$ and column $Y$ contains $(Z_{1},\circ)$,
and $(Z_{2},[\ ])$ this means that there is a symmetric bilinear
map $X\times Y\rightarrow Z_{1},$ given by $\alpha\otimes\beta\mapsto(\alpha\circ\beta)_{Z_{1}}$
and a skew symmetric bilinear map $X\times Y\rightarrow Z_{2},$ given
by $\alpha\otimes\beta\mapsto[\alpha,\beta]_{Z_{2}}$ $(\alpha\in X,\beta\in Y)$.

\begin{table}[H]
\begin{tabular}{|c|r@{\extracolsep{0pt}.}l|c|c|c|c|c|}
\hline 
$.$ & \multicolumn{2}{c|}{$A^{+}$} & $A^{-}$ & $C$ & $E$ & $C'$ & $E'$\tabularnewline
\hline 
\hline 
$A^{+}$ & \multicolumn{2}{c|}{$\begin{array}{c}
(A^{+},\circ)\\
(A^{-},[\ ])
\end{array}$} & $\begin{array}{c}
(A^{-},\circ)\\
(A^{+},[\ ])
\end{array}$ & $\begin{array}{c}
(C,\circ)\\
(E,[\ ])
\end{array}$ & $\begin{array}{c}
(E,\circ)\\
(C,[\ ])
\end{array}$ & $\begin{array}{c}
(C',\circ)\\
(E,[\ ])
\end{array}$ & $\begin{array}{c}
(E',\circ)\\
(C',[\ ])
\end{array}$\tabularnewline
\hline 
$A^{-}$ & \multicolumn{2}{c|}{$\begin{array}{c}
(A^{-},\circ)\\
(A^{+},[\ ])
\end{array}$} & $\begin{array}{c}
(A^{+},\circ)\\
(A^{-},[\ ])\text{ }
\end{array}$ & $\begin{array}{c}
(E,\circ)\\
(C,[\ ])
\end{array}$ & $\begin{array}{c}
(C,\circ)\\
(E,[\ ])
\end{array}$ & $\begin{array}{c}
(E',\circ)\\
(C',[\ ])
\end{array}$ & $\begin{array}{c}
(C',\circ)\\
(E',[\ ])
\end{array}$\tabularnewline
\hline 
$C$ & \multicolumn{2}{c|}{$\begin{array}{c}
(C,\circ)\\
(E,[\ ])
\end{array}$} & $\begin{array}{c}
(E,\circ)\mbox{ }\\
(C,[\ ])\text{ }
\end{array}$ & $0$ & $0$ & $\begin{array}{c}
(A^{+},\circ)\\
(A^{-},[\ ])
\end{array}$ & $\begin{array}{c}
(A^{-},\circ)\\
(A^{+},[\ ])
\end{array}$\tabularnewline
\hline 
$E$ & \multicolumn{2}{c|}{$\begin{array}{c}
(E,\circ)\\
(C,[\ ])
\end{array}$} & $\begin{array}{c}
(C,\circ)\\
(E,[\ ])
\end{array}$ & $0$ & $0$ & $\begin{array}{c}
(A^{-},\circ)\\
(A^{+},[\ ])
\end{array}$ & $\begin{array}{c}
(A^{+},\circ)\\
(A^{-},[\ ])\text{ }
\end{array}$\tabularnewline
\hline 
$C'$ & \multicolumn{2}{c|}{$\begin{array}{c}
(C',\circ)\\
(E,[\ ])
\end{array}$} & $\begin{array}{c}
(E',\circ)\\
(C',[\ ])
\end{array}$ & $\begin{array}{c}
(A^{+},\circ)\\
(A^{-},[\ ])
\end{array}$ & $\begin{array}{c}
(A^{-},\circ)\\
(A^{+},[\ ])
\end{array}$ & $0$ & $0$\tabularnewline
\hline 
$E'$ & \multicolumn{2}{c|}{$\begin{array}{c}
(E',\circ)\\
(C',[\ ])
\end{array}$} & $\begin{array}{c}
(C',\circ)\\
(E',[\ ])
\end{array}$ & $\begin{array}{c}
(A^{-},\circ)\\
(A^{+},[\ ])
\end{array}$ & $\begin{array}{c}
(A^{+},\circ)\\
(A^{-},[\ ])
\end{array}$ & $0$ & $0$\tabularnewline
\hline 
\end{tabular}

\caption{Products of homogeneous components of $\mathfrak{a}$}

\label{t4}
\end{table}

We are going to show that $\mathfrak{a}=A^{+}\oplus A^{-}\oplus C\oplus E\oplus C'\oplus E'$
is an associative algebra with respect to multiplication defined as
follows: 
\begin{equation}
\alpha_{1}\alpha_{2}:=\frac{[\alpha_{1},\alpha_{2}]}{2}+\frac{\alpha_{1}\circ\alpha_{2}}{2}\label{pro on a}
\end{equation}
 for all homogeneous $\alpha_{1},\alpha_{2}\in\mathfrak{a}$ with
the products $[\ ]$ and $\circ$ given by Table \ref{t4}. Note that
$[\alpha_{1},\alpha_{2}]=\alpha_{1}\alpha_{2}-\alpha_{2}\alpha_{1}$
and $\alpha_{1}\circ\alpha_{2}=\alpha_{1}\alpha_{2}+\alpha_{2}\alpha_{1}$.

From Table \ref{t4} and the formulas in Remark \ref{FF}, we deduce
the following.
\begin{lem}
\label{total} Let $\alpha_{1}$ and $\alpha_{2}$ be homogeneous
elements of $\mathfrak{a}$. Then
\begin{eqnarray*}
[z_{1}\otimes\alpha_{1},z_{2}\otimes\alpha_{2}] & = & z_{1}\circ z_{2}\otimes\frac{[\alpha_{1},\alpha_{2}]}{2}+[z_{1},z_{2}]\otimes\frac{\alpha_{1}\circ\alpha_{2}}{2}+(z_{1}\mid z_{2})\langle\alpha_{1},\alpha_{2}\rangle
\end{eqnarray*}
if $\alpha_{1},\alpha_{2}\in X$ with $X=A^{\pm}$ or $\alpha_{1}\in X$
and $\alpha_{2}\in X'$ with $X=C,E$. In all other cases we have
\begin{eqnarray*}
[z_{1}\otimes\alpha_{1},z_{2}\otimes\alpha_{2}] & = & z_{1}\diamond z_{2}\otimes\frac{[\alpha_{1},\alpha_{2}]}{2}+[z_{1},z_{2}]\otimes\frac{\alpha_{1}\circ\alpha_{2}}{2}.
\end{eqnarray*}
\end{lem}

\begin{thm}
\label{associati p} $\mathfrak{a}=A^{+}\oplus A^{-}\oplus C\oplus E\oplus C'\oplus E'$
is an associative algebra with identity element $1^{+}$.
\end{thm}

\begin{proof}
It will be shown in Proposition \ref{iden} that $1^{+}$ is the identity
element of a larger algebra $\mathfrak{b}$ containing $\mathfrak{a}$
as a subalgebra. Therefore we only need to prove the associativity.
Let $\alpha_{1},\alpha_{2},\alpha_{3}\mathfrak{\in a}.$ We need to
show that $\alpha_{1}(\alpha_{2}\alpha_{3})=(\alpha_{1}\alpha_{2})\alpha_{3}.$
By linearity, we can assume that $\alpha_{1},\alpha_{2}$ and $\alpha_{3}$
are homogeneous. Set $z_{1}=E_{1,2}+\varepsilon_{1}E_{2,1}$, $z_{2}=E_{2,3}+\varepsilon_{2}E_{3,2}$
and $z_{3}=E_{3,4}+\varepsilon_{3}E_{4,3}$ where $\varepsilon_{i}=\pm1$.
The signs of each $\varepsilon_{i}$ can be chosen in such a way that
$z_{i}\otimes\alpha_{i}$ belongs to the corresponding homogeneous
component of $L$. Note that $\tra(z_{i}z_{j})=0$, for all $i\neq j$.
Hence by Lemma \ref{total}, we have 
\[
[z_{i}\otimes\alpha_{i},z_{j}\otimes\alpha_{j}]=z_{i}\diamond z_{j}\otimes\frac{[\alpha_{i},\alpha_{j}]}{2}+[z_{i},z_{j}]\otimes\frac{\alpha_{i}\circ\alpha_{j}}{2}.
\]
Consider the Jacoby identity for $z_{1}\otimes\alpha_{1},z_{2}\otimes\alpha_{2},z_{3}\otimes\alpha_{3}$:
\[
[z_{1}\otimes\alpha_{1},[z_{2}\otimes\alpha_{2},z_{3}\otimes\alpha_{3}]]=[[z_{1}\otimes\alpha_{1},z_{2}\otimes\alpha_{2}],z_{3}\otimes\alpha_{3}]+[z_{2}\otimes\alpha_{2},[z_{1}\otimes\alpha_{1},z_{3}\otimes\alpha_{3}]].
\]
Using Lemma \ref{total} yields 
\begin{align}
 & 2[z_{1},[z_{2},z_{3}]]\otimes\alpha_{1}\circ(\alpha_{2}\circ\alpha_{3})+z_{1}\diamond[z_{2},z_{3}]\otimes[\alpha_{1},\alpha_{2}\circ\alpha_{3}]+[z_{1},(z_{2}\diamond z_{3})]\otimes\alpha_{1}\circ[\alpha_{2},\alpha_{3}]\label{eq:66}\\
 & +z_{1}\diamond(z_{2}\diamond z_{3})\otimes[\alpha_{1},[\alpha_{2},\alpha_{3}]]=[[z_{1},z_{2}],z_{3}]\otimes(\alpha_{1}\circ\alpha_{2})\circ\alpha_{3}+([z_{1},z_{2}]\diamond z_{3})\otimes[\alpha_{1}\circ\alpha_{2},\alpha_{3}]\nonumber \\
 & +[z_{1}\diamond z_{2},z_{3}]\otimes[\alpha_{1},\alpha_{2}]\circ\alpha_{3}+(z_{1}\circ z_{2})\circ z_{3}\otimes[[\alpha_{1},\alpha_{2}],\alpha_{3}]+[z_{2},[z_{1},z_{3}]]\otimes\alpha_{2}\circ(\alpha_{1}\circ\alpha_{3})\nonumber \\
 & +z_{2}\diamond[z_{1},z_{3}]\otimes[\alpha_{2},\alpha_{1}\circ\alpha_{3}]+[z_{2},(z_{1}\diamond z_{3})]\otimes\alpha_{2}\circ[\alpha_{1},\alpha_{3}]+z_{2}\diamond(z_{1}\diamond z_{3})\otimes[\alpha_{2},[\alpha_{1},\alpha_{3}]].\nonumber 
\end{align}
Note that $z_{1}\diamond(z_{2}\diamond z_{3})=E_{1,4}+\varepsilon_{1}\varepsilon_{2}\varepsilon_{3}E_{4,1}$,
$[z_{1},(z_{2}\diamond z_{3})]=E_{1,4}-\varepsilon_{1}\varepsilon_{2}\varepsilon_{3}E_{4,1}$,
$z_{1}\diamond[z_{2},z_{3}]=E_{1,4}-\varepsilon_{1}\varepsilon_{2}\varepsilon_{3}E_{4,1}$,
$[[z_{1},z_{2}],z_{3}]=E_{1,4}+\varepsilon_{1}\varepsilon_{2}\varepsilon_{3}E_{4,1}$,
$(z_{1}\diamond z_{2})\diamond z_{3}=E_{1,4}+\varepsilon_{1}\varepsilon_{2}\varepsilon_{3}E_{4,1}$,
$[z_{1}\diamond z_{2},z_{3}]=E_{1,4}-\varepsilon_{1}\varepsilon_{2}\varepsilon_{3}E_{4,1}$,
$[z_{1},z_{2}]\diamond z_{3}=E_{1,4}-\varepsilon_{1}\varepsilon_{2}\varepsilon_{3}E_{4,1}$
and $[z_{2},[z_{1},z_{3}]]=z_{2}\diamond(z_{1}\diamond z_{3})=[z_{2},(z_{1}\diamond z_{3})]=z_{2}\diamond[z_{1},z_{3}]=0.$
Now (\ref{eq:66}) becomes 
\begin{align*}
 & (E_{1,4}+\varepsilon_{1}\varepsilon_{2}\varepsilon_{3}E_{4,1})\otimes\alpha_{1}\circ(\alpha_{2}\circ\alpha_{3})+(E_{1,4}-\varepsilon_{1}\varepsilon_{2}\varepsilon_{3}E_{4,1})\otimes[\alpha_{1},\alpha_{2}\circ\alpha_{3}]\\
 & +(E_{1,4}-\varepsilon_{1}\varepsilon_{2}\varepsilon_{3}E_{4,1})\otimes\alpha_{1}\circ[\alpha_{2},\alpha_{3}]+(E_{1,4}+\varepsilon_{1}\varepsilon_{2}\varepsilon_{3}E_{4,1})\otimes[\alpha_{1},[\alpha_{2},\alpha_{3}]]\\
 & =(E_{1,4}+\varepsilon_{1}\varepsilon_{2}\varepsilon_{3}E_{4,1})\otimes(\alpha_{1}\circ\alpha_{2})\circ\alpha_{3}+(E_{1,4}-\varepsilon_{1}\varepsilon_{2}\varepsilon_{3}E_{4,1})\otimes[\alpha_{1}\circ\alpha_{2},\alpha_{3}]\\
 & +(E_{1,4}-\varepsilon_{1}\varepsilon_{2}\varepsilon_{3}E_{4,1})\otimes[\alpha_{1},\alpha_{2}]\circ\alpha_{3}+(E_{1,4}+\varepsilon_{1}\varepsilon_{2}\varepsilon_{3}E_{4,1})\otimes[[\alpha_{1},\alpha_{2}],\alpha_{3}].
\end{align*}
By collecting the coefficients of $E_{1,4}$ we get
\begin{eqnarray*}
 & \alpha_{1}\circ(\alpha_{2}\circ\alpha_{3})+[\alpha_{1},\alpha_{2}\circ\alpha_{3}]+\alpha_{1}\circ[\alpha_{2},\alpha_{3}]+[\alpha_{1},[\alpha_{2},\alpha_{3}]]\\
 & =(\alpha_{1}\circ\alpha_{2})\circ\alpha_{3}+[\alpha_{1}\circ\alpha_{2},\alpha_{3}]+[\alpha_{1},\alpha_{2}]\circ\alpha_{3}+[[\alpha_{1},\alpha_{2}],\alpha_{3}],
\end{eqnarray*}
or equivalently $\alpha_{1}(\alpha_{2}\alpha_{3})=(\alpha_{1}\alpha_{2})\alpha_{3},$
as required.
\end{proof}
From Theorem \ref{associati p} and tensor product decompositions
for $sl_{n}$ ($n\geq5$), we deduce the following
\begin{cor}
\label{cor A is sub} (1) $\mathcal{A}=A^{-}\oplus A^{+}$ is an associative
subalgebra of $\mathfrak{a}$ with identity element $1^{+}$.

(2) $C\oplus E$ and $C'\oplus E'$ are $\mathcal{A}$-bimodules.
\end{cor}

\begin{thm}
\label{a invol} The linear transformation $\gamma:\mathfrak{a}\rightarrow\mathfrak{a}$
defined by
\[
\gamma(a^{-})=-a^{-},\gamma(a^{+})=a^{+},\gamma(c)=-c,\gamma(e)=e,\gamma(c')=-c',\gamma(e')=e',
\]
is an antiautomorphism of order 2 of the algebra $\mathfrak{a}.$
\end{thm}

\begin{proof}
One can easily check that $\gamma(xy)=\gamma(y)\gamma(x)$ for all
homogeneous $x$ and $y$ in $\mathfrak{a}$, see \cite[Theorem 4.16]{yaseen2018generalized}.
\end{proof}

\subsection{Coordinate algebra $\mathfrak{b}$}

The aim of this subsection is to show that $\mathfrak{b}=\mathfrak{a}\oplus B\oplus B'$
is a (non-associative) algebra with identity $1^{+}$ with respect
to the multiplication extending that on $\mathfrak{a}$ given in Table
\ref{t5}. It can be shown that all products $(\beta_{1},\beta_{2})_{Z}$
with $\beta_{1},\beta_{2}\in B\oplus B'$ are either symmetric or
skew-symmetric. This is why we will write $(\beta_{1}\circ\beta_{2})_{Z}$
or $[\beta_{1},\beta_{2}]_{Z}$, respectively, instead of $(\beta_{1},\beta_{2})_{Z}$.
For $\alpha\in\mathfrak{a}$ and $\beta\in B\oplus B'$ we will write
$\alpha\beta$ (resp. $\beta\alpha$) instead of $(\alpha,\beta)_{Z}$
(resp. $(\beta,\alpha)_{Z}$). Let $b\in B$ and $b'\in B$. We define
$b\alpha:=\gamma(\alpha)b$ and $\alpha b':=b'\gamma(\alpha)$. We
will show that $B\oplus B'$ is an $\mathfrak{a}$-bimodule. Let $u\otimes b\in V\otimes B$
and $v'\otimes b'\in V'\otimes B'$. We need the following formula
from (\ref{main for}): 
\[
[u\otimes b,v'\otimes b']=(uv'^{t}-\frac{\tra(uv'^{t})}{n}I)\otimes(b,b')_{A}+\frac{2\tra(uv'^{t})}{n}\langle b,b'\rangle.
\]
By splitting $(uv'^{t}-\frac{\tra(uv'^{t})}{n}I)\otimes(b,b')_{A}$
into symmetric and skew-symmetric parts ($x\otimes a=\frac{(x+x^{t})}{2}\otimes a+\frac{(x-x^{t})}{2}\otimes a\in\mathfrak{g^{+}}\otimes A+\mathfrak{g^{-}}\otimes A$),
we get $[u\otimes b,v'\otimes b']=(uv'^{t}+v'u^{t}-\frac{2\tra(uv'^{t})}{n}I)\otimes\frac{(b,b')_{A^{-}}}{2}+(uv'^{t}-v'u^{t})\otimes\frac{(b,b')_{A^{+}}}{2}+\frac{2\tra(uv'^{t})}{n}\langle b,b'\rangle$.
Let $b,b_{1},b_{2}\in B$ and $b',b_{1}',b_{2}'\in B'$. By rewriting
some of the products in (\ref{main for}) and the above equation in
terms of symmetric and skew-symmetric elements, we get 
\begin{align}
 & [u\otimes b_{1},v\otimes b_{2}]=(uv^{t}+vu^{t})\otimes\frac{[b_{1},b_{2}]_{C}}{2}+(uv^{t}-vu^{t})\otimes\frac{(b_{1}\circ b_{2})_{E}}{2},\label{eq:formulla for natural elements}\\
 & [u'\otimes b'_{1},v'\otimes b'_{2}]=(u'v'^{t}+v'u'^{t})\otimes\frac{[b'_{1},b'_{2}]_{C'}}{2}+(u'v'^{t}-v'u'^{t})\otimes\frac{(b'_{1}\circ b'_{2})_{E'}}{2},\nonumber \\
 & [u\otimes b,v'\otimes b']=(uv'^{t}+v'u^{t}-\frac{2\tra(uv'^{t})}{n}I)\otimes\frac{[b,b']_{A^{-}}}{2}+(uv'^{t}-v'u^{t})\otimes\frac{(b\circ b')_{A^{+}}}{2}+\frac{2\tra(uv'^{t})}{n}\langle b,b'\rangle.\nonumber 
\end{align}
We define 
\begin{eqnarray}
b_{1}b_{2}:=\frac{[b_{1},b_{2}]_{C}}{2}+\frac{(b_{1}\circ b_{2})_{E}}{2}, & \  & b'_{1}b'_{2}:=\frac{[b'_{1},b'_{2}]_{C'}}{2}+\frac{(b'_{1}\circ b'_{2})_{E'}}{2},\label{pro on b}\\
bb':=\frac{[b,b']_{A^{-}}}{2}+\frac{(b\circ b')_{A^{+}}}{2}, & \  & b'b:=-\frac{[b,b']_{A^{-}}}{2}+\frac{(b\circ b')_{A^{+}}}{2}.\nonumber 
\end{eqnarray}
 Then $\mathfrak{b}=\mathfrak{a}\oplus B\oplus B'$ is an algebra
with multiplication extending that on $\mathfrak{a.}$ The following
table describes the products of homogeneous elements of $\mathfrak{b}$
(use Table \ref{t4} for the products on $\mathfrak{a}$).

\begin{table}[H]
\begin{tabular}{|c|c|c|c|c|c|}
\hline 
. & $A^{+}+A^{-}$ & $C+E$ & $C'+E'$ & $B$ & $B'$\tabularnewline
\hline 
\hline 
$A^{+}+A^{-}$ & $A^{+}+A^{-}$ & $C+E$ & $C'+E'$ & $B$ & $B'$\tabularnewline
\hline 
$C+E$ & $C+E$ & $0$ & $A^{+}+A^{-}$ & $0$ & $B$\tabularnewline
\hline 
$C'+E'$ & $C'+E'$ & $A^{+}+A^{-}$ & $0$ & $B'$ & 0\tabularnewline
\hline 
$B$ & $B$ & $0$ & $B'$ & $\begin{array}{c}
(E,\circ)\\
(C,[\ ])
\end{array}$ & $\begin{array}{c}
(A^{+},\circ)\mbox{ }\\
(A^{-},[\ ])
\end{array}$\tabularnewline
\hline 
$B'$ & $B'$ & $B$ & $0$ & $\begin{array}{c}
(A^{+},\circ)\\
(A^{-},[\ ])
\end{array}$ & $\begin{array}{c}
(E',\circ)\mbox{ }\\
(C',[\ ])
\end{array}$\tabularnewline
\hline 
\end{tabular}

\caption{Products in $\mathfrak{b}$}

\label{t5}
\end{table}

\begin{thm}
\label{b involution}The linear transformation $\eta:\mathfrak{b}\rightarrow\mathfrak{b}$
defined by $\eta(\alpha)=\gamma(\alpha)$, $\eta(b)=b$ and $\eta(b')=b'$
for all $\alpha\in\mathfrak{a}$, $b\in B$ and $b'\in B'$ is an
antiautomorphism of order $2$ of the algebra $\mathfrak{b}$.
\end{thm}

\begin{proof}
In Theorem \ref{a invol}, we showed that $\eta(xy)=\eta(y)\eta(x)$
for all $x$ and $y$ in $\mathfrak{a}$. It remains to consider the
components $B$ and $B'$. Let $b,b_{1},b_{2}\in B,b',b_{1}',b_{2}'\in B'$
and $\alpha\in\mathfrak{a}$. Then, as required,
\begin{align*}
\eta(b_{1}b_{2}) & =\eta(\frac{[b_{1},b_{2}]_{C}+(b_{1}\circ b_{2})_{E}}{2})=\frac{-[b_{1},b_{2}]_{C}+(b_{1}\circ b_{2})_{E}}{2}=b_{2}b_{1}=\eta(b_{2})\eta(b_{1}),\\
\eta(b_{1}'b_{2}') & =\eta(\frac{[b'_{1},b'_{2}]_{C'}+(b'_{1}\circ b'_{2})_{E'}}{2})=\frac{-[b'_{1},b'_{2}]_{C'}+(b'_{1}\circ b'_{2})_{E'}}{2}=b_{2}'b_{1}'=\eta(b_{2}')\eta(b_{1}'),\\
\eta(bb') & =\eta(\frac{[b,b']_{A^{-}}+(b\circ b')_{A^{+}}}{2})=\frac{-[b,b']_{A^{-}}+(b\circ b')_{A^{+}}}{2}=b'b=\eta(b')\eta(b),\\
\eta(\alpha b) & =\alpha b=b\eta(\alpha)=\eta(b)\eta(\alpha),\quad\eta(b'\alpha)=b'\alpha=\eta(\alpha)b'=\eta(\alpha)\eta(b').
\end{align*}
\end{proof}
\begin{prop}
\label{iden} $1^{+}$ is the identity element of $\mathfrak{b}$.
\end{prop}

\begin{proof}
Using (\ref{main for}) and (\ref{eq:iden 1}) and the fact that $\circ$
is symmetric, $[\,,\,]$ is skew symmetric and $\eta(1^{+})=1^{+}$
one can check that $1^{+}$ is the identity element of $\mathfrak{b}$,
see \cite[Proposition 4.2.2]{yaseen2018generalized}
\end{proof}
Using (\ref{main for}) and Table \ref{t5}, we get the following.
\begin{lem}
\label{bb'} Let $b\in B$, $b'\in B'$ and $\alpha\in\mathfrak{a}$.
Then 
\[
[z\otimes\alpha,u\otimes b]=zu\otimes\alpha b\text{ \ and \ }[u'\otimes b',z\otimes\alpha]=z^{t}u'\otimes b'\alpha.
\]
\end{lem}

\begin{prop}
\label{b is A module}$B\oplus B'$ is an $\mathfrak{a}$-bimodule.
\end{prop}

\begin{proof}
Let $b\in B,b'\in B'$ and let $\alpha_{1},\alpha_{2}$ be homogeneous
elements in $\mathfrak{a}$. Set 
\[
z_{1}=E_{1,2}+\varepsilon_{1}E_{2,1},\ z_{2}=E_{2,3}+\varepsilon_{2}E_{3,2}\text{ and }u=u'=e_{3}\text{ where }\varepsilon_{i}=\pm1.
\]
 Then $[z_{1},z_{2}]=E_{1,3}-\varepsilon_{1}\varepsilon_{2}E_{3,1}$,
$z_{1}\circ z_{2}=E_{1,3}+\varepsilon_{1}\varepsilon_{2}E_{3,1}$,
$z_{1}z_{2}=E_{1,3}$ and $(z_{1}|z_{2})=0$.

First we are going to show that $(\alpha_{1}\alpha_{2})b=\alpha_{1}(\alpha_{2}b)$.
Consider the Jacoby identity for $z_{1}\otimes\alpha_{1},z_{2}\otimes\alpha_{2},u\otimes b$:
\[
[z_{1}\otimes\alpha_{1},[z_{2}\otimes\alpha_{2},u\otimes b]]=[[z_{1}\otimes\alpha_{1},z_{2}\otimes\alpha_{2}],u\otimes b]+[z_{2}\otimes\alpha_{2},[z_{1}\otimes\alpha_{1},u\otimes b]].
\]
Using Lemmas \ref{bb'} and \ref{total}, we get 
\[
z_{1}(z_{2}u)\otimes\alpha_{1}(\alpha_{2}b)-(z_{1}\circ z_{2})u\otimes\frac{[\alpha_{1},\alpha_{2}]}{2}b-[z_{1},z_{2}]u\otimes\frac{\alpha_{1}\circ\alpha_{2}}{2}b=0.
\]
Substituting matrix units, we get that $e_{1}\otimes(\alpha_{1}(\alpha_{2}b)-\frac{[\alpha_{1},\alpha_{2}]}{2}b-\frac{\alpha_{1}\circ\alpha_{2}}{2}b)=0,$
so $\alpha_{1}(\alpha_{2}b)=\frac{[\alpha_{1},\alpha_{2}]}{2}b+\frac{\alpha_{1}\circ\alpha_{2}}{2}b=(\alpha_{1}\alpha_{2})b$,
as required. Now we are going to show that $(b'\alpha_{2})\alpha_{1}=b'(\alpha_{2}\alpha_{1})$.
Consider the Jacoby identity for $z_{1}\otimes\alpha_{1},z_{2}\otimes\alpha_{2},u'\otimes b'$:
\[
[z_{1}\otimes\alpha_{1},[z_{2}\otimes\alpha_{2},u'\otimes b']]=[[z_{1}\otimes\alpha_{1},z_{2}\otimes\alpha_{2}],u'\otimes b']+[z_{2}\otimes\alpha_{2},[z_{1}\otimes\alpha_{1},u'\otimes b']].
\]
Using Lemmas \ref{total} and \ref{bb'}, we get 
\[
(z_{2}z_{1})^{t}u'\otimes(b'\alpha_{2})\alpha_{1}=-(z_{1}\circ z_{2})^{t}u'\otimes b'\frac{[\alpha_{1},\alpha_{2}]}{2}-[z_{1},z_{2}]^{t}u'\otimes b'\frac{\alpha_{1}\circ\alpha_{2}}{2}.
\]
Substituting matrix units, we get that $\varepsilon_{1}\varepsilon_{2}e_{1}\otimes(b'\alpha_{2})\alpha_{1}=-\varepsilon_{1}\varepsilon_{2}e_{1}\otimes b'\frac{[\alpha_{1},\alpha_{2}]}{2}+b'\frac{\alpha_{1}\circ\alpha_{2}}{2},$
so $(b'\alpha_{2})\alpha_{1}=b'(\alpha_{2}\alpha_{1}),$ as required.
It remains to show $b(\alpha_{1}\alpha_{2})=(b\alpha_{1})\alpha_{2}$
and $(\alpha_{1}\alpha_{2})b'=\alpha_{1}(\alpha_{2}b')$. We have
\[
b(\alpha_{1}\alpha_{2})=\eta((\eta(\alpha_{2})\eta(\alpha_{1}))\eta(b))=\eta(\eta(\alpha_{2})(\eta(\alpha_{1})\eta(b)))=\eta(\eta(\alpha_{2})\eta((b\alpha_{1})))=(b\alpha_{1})\alpha_{2}.
\]
Similarly, we get $(\alpha_{1}\alpha_{2})b'=\alpha_{1}(\alpha_{2}b')$,
as required.
\end{proof}
Note that both $B$ and $B'$ are invariant under multiplication by
$\mathcal{\mathcal{A}}=A^{+}\oplus A^{-}$, see Table \ref{t5}, so
we get the following.
\begin{cor}
\label{cor B is A (sub) }$B$ and $B'$ are $\mathcal{\mathcal{A}}$-bimodules.
\end{cor}

\begin{prop}
\label{mod B} Let $\chi(\beta_{1},\beta_{2}):=\beta_{1}\beta_{2}$
for all $\beta_{1},\beta_{2}\in B\oplus B'$. Then $\chi$ is a hermitian
form on the $\mathfrak{a}$-bimodule $B\oplus B'$ with values in
$\mathfrak{a}$. More exactly, for all $\alpha\in\mathfrak{a}$ and
$\beta_{1},\beta_{2}\in B\oplus B'$ we have

$(i)\ \chi(\alpha\beta_{1},\beta_{2})=\alpha\chi(\beta_{1},\beta_{2})$,

$(ii)\ \eta(\chi(\beta_{1},\beta_{2}))=\chi(\beta_{2},\beta_{1})$,

$(iii)\ \chi(\beta_{1},\alpha\beta_{2})=\chi(\beta_{1},\beta_{2})\eta(\alpha)$.
\end{prop}

\begin{proof}
(i) We need to show that $(\alpha\beta_{1})\beta_{2}=\alpha(\beta_{1}\beta_{2})$
for all homogeneous $\beta_{1}$, $\beta_{2}$ in $B\oplus B'$ and
$\alpha\in\mathfrak{a}$. Set $z=E_{1,2}+\varepsilon E_{2,1}$, $u_{1}=u_{1}'=e_{1}$
and $u_{2}=u_{2}'=e_{3}$ where $\varepsilon=\pm1$. Let $b_{1},b_{2}\in B$
and $b_{1}',b_{2}'\in B'$. First we are going to show that $\alpha(b_{1}b_{2})=(\alpha b_{1})b_{2}$.
Consider the Jacoby identity for $z\otimes\alpha$, $u_{1}\otimes b_{1}$,
$u_{2}\otimes b_{2}$:
\[
[z\otimes\alpha,[u_{1}\otimes b_{1},u_{2}\otimes b_{2}]]=[[z\otimes\alpha,u_{1}\otimes b_{1}],u_{2}\otimes b_{2}]+[u_{1}\otimes b_{1},[z\otimes\alpha,u_{2}\otimes b_{2}]].
\]
Using (\ref{eq:formulla for natural elements}) and Lemma \ref{bb'}
we get
\[
[z\otimes\alpha,(E_{1,3}+E_{3,1})\otimes\frac{1}{2}[b_{1},b_{2}]_{C}+[z\otimes\alpha,(E_{1,3}-E_{3,1})\otimes\frac{1}{2}(b_{1}\circ b_{2})_{E}]=[\varepsilon e_{2}\otimes\alpha b_{1},u_{2}\otimes b_{2}].
\]
By using Lemma \ref{total} and (\ref{eq:formulla for natural elements}),
we get
\begin{align*}
 & (\varepsilon E_{2,3}+\varepsilon E_{3,2})\otimes[\alpha,[b_{1},b_{2}]_{C}]+(\varepsilon E_{2,3}-\varepsilon E_{3,2})\otimes\alpha\circ[b_{1},b_{2}]_{C}+(\varepsilon E_{2,3}+\varepsilon E_{3,2})\otimes[\alpha,(b_{1}\circ b_{2})_{E}]\\
 & +(\varepsilon E_{2,3}-\varepsilon E_{3,2})\otimes\alpha\circ(b_{1}\circ b_{2})_{E}=(\varepsilon E_{2,3}+\varepsilon E_{3,2})\otimes[\alpha b_{1},b_{2}]+(\varepsilon E_{2,3}-\varepsilon E_{3,2})\otimes\alpha b_{1}\circ b_{2}.
\end{align*}
By collecting the coefficients of $E_{2,3}$, we get:
\[
[\alpha,[b_{1},b_{2}]_{C}]+\alpha\circ[b_{1},b_{2}]_{C}+[\alpha,(b_{1}\circ b_{2})_{E}]+\alpha\circ(b_{1}\circ b_{2})_{E}=[\alpha b_{1},b_{2}]+\alpha b_{1}\circ b_{2},
\]
 or equivalently $\alpha(b_{1}b_{2})=(\alpha b_{1})b_{2}$, as required.
Similarly, we get $\alpha(b_{1}b_{2}')=(\alpha b_{1})b_{2}'$ and
$b_{2}'(b_{1}'\alpha)=(b_{2}'b_{1}')\alpha$ (by using the Jacoby
identity for $z\otimes\alpha,u_{1}\otimes b_{1},u_{2}'\otimes b_{2}'$
and $z\otimes\alpha$, $u_{1}'\otimes b_{1}'$, $u_{2}'\otimes b_{2}'$,
respectively). By applying $\eta$ to both sides of the last equation
and using the fact that $\eta$ is identity on both $B$ and $B'$,
we get $(\eta(\alpha)b_{1}')b_{2}'=\eta(\alpha)(b_{1}'b_{2}')$, or
equivalently $(\alpha b_{1}')b_{2}'=\alpha(b_{1}'b_{2}')$, as required.
By using the Jacoby identity for $z\otimes\alpha,u_{1}\otimes b_{1},u_{2}'\otimes b_{2}'$
we get $(b_{2}b_{1}')\alpha=b_{2}(b_{1}'\alpha)$. By applying $\eta$
we get $\eta(\alpha)(b_{1}'b_{2})=(\eta(\alpha)b_{1}')b_{2}$, or
equivalently $\alpha(b_{1}'b_{2})=(\alpha b_{1}')b_{2},$ as required.

$(ii)$ We only need to check this for homogeneous elements. We have,
as required,
\begin{align*}
\eta(\chi(b_{1},b_{2})) & =\eta(\frac{[b_{1},b_{2}]_{C}+(b_{1}\circ b_{2})_{E}}{2})=\frac{-[b_{1},b_{2}]_{C}+(b_{1}\circ b_{2})_{E}}{2}=\chi(b_{2},b_{1}),\\
\eta(\chi(b_{1}',b_{2}')) & =\eta(\frac{[b_{1}',b_{2}']_{C'}+(b_{1}'\circ b_{2}')_{E'}}{2})=\frac{-[b_{1}',b_{2}']_{C'}+(b_{1}'\circ b_{2}')_{E'}}{2}=\chi(b_{2}',b_{1}'),\\
\eta(\chi(b_{1},b_{1}')) & =\eta(\frac{[b_{1},b_{1}']_{A^{-}}+(b_{1}\circ b_{1}')_{A^{+}}}{2})=\frac{-[b_{1},b_{1}']_{A^{-}}+(b_{1}\circ b_{1}')_{A^{+}}}{2}=\chi(b_{1}',b_{1}),\\
\eta(\chi(b_{1}',b_{1})) & =\eta(\frac{[b_{1}',b_{1}]_{A^{-}}+(b_{1}'\circ b_{1})_{A^{+}}}{2})=\frac{-[b_{1}',b_{1}]_{A^{-}}+(b_{1}'\circ b_{1})_{A^{+}}}{2}=\chi(b_{1},b_{1}'),
\end{align*}

$(iii)$ Using $(i)$ and $(ii)$, we get, as required, 
\[
\chi(\beta_{1},\alpha\beta_{2})=\eta(\chi(\alpha\beta_{2},\beta_{1}))=\eta(\alpha\chi(\beta_{2},\beta_{1}))=\eta(\chi(\beta_{2},\beta_{1}))\eta(\alpha)=\chi(\beta_{1},\beta_{2})\eta(\alpha).
\]
\end{proof}
The mapping $\langle,\rangle:X\otimes X'\rightarrow D$ with $X=B,C,E$
can be extended to $X'\otimes X$ in a consistent way by defining
$\langle x',x\rangle:=-\langle x,x'\rangle$. Let $X,Y\in\{A^{+},A^{-},B,B',C,C',E,E'\}$.
Recall also the maps $\langle,\rangle:A^{\pm}\otimes A^{\pm}\rightarrow D$
described previously (see Remark \ref{FF}(a)). For the convenience,
we extend the mappings to the whole space $\mathfrak{b}$ by defining
the remaining $\langle X,Y\rangle$ to be zero. Hence $\langle\mathfrak{b},\mathfrak{b}\rangle=\langle A^{+},A^{+}\rangle+\langle A^{-},A^{-}\rangle+\langle B,B'\rangle+\langle C,C'\rangle+\langle E,E'\rangle.$
It follows from condition $(\Gamma3)$ in Definition \ref{def of gamma}
that 
\begin{equation}
D=\langle\mathfrak{b},\mathfrak{b}\rangle=\langle A^{+},A^{+}\rangle+\langle A^{-},A^{-}\rangle+\langle B,B'\rangle+\langle C,C'\rangle+\langle E,E'\rangle.\label{eq:kk}
\end{equation}

\begin{prop}
\label{derivation rule} Let $\alpha_{1},\alpha_{2}$ and $\alpha_{3}$
be homogeneous elements in $\mathfrak{b}$ with $\langle\alpha_{1},\alpha_{2}\rangle\neq0$.
Then
\begin{equation}
\langle\alpha_{1},\alpha_{2}\rangle\alpha_{3}=\begin{cases}
[[\alpha_{1},\alpha_{2}]_{A^{-}},\alpha_{3}] & \text{ if }\alpha_{1},\alpha_{2},\alpha_{3}\in\mathfrak{a},\\{}
[\alpha_{1},\alpha_{2}]_{A^{-}}\alpha_{3} & \text{ if }\alpha_{1},\alpha_{2}\in\mathfrak{a},\;\alpha_{3}\in B\oplus B',\\
\frac{1}{2}[[\alpha_{1},\alpha_{2}]_{A^{-}},\alpha_{3}] & \text{ if }\alpha_{1}\in B,\;\alpha_{2}\in B',\;\alpha_{3}\in\mathfrak{a}.\\
\frac{1}{2}([\alpha_{1},\alpha_{2}]_{A^{-}}\alpha_{3}+n((\alpha_{3}\alpha_{2})\alpha_{1}-(\alpha_{3}\alpha_{1})\alpha_{2})) & \text{ if }\alpha_{1},\alpha_{2},\alpha_{3}\in B\oplus B',
\end{cases}\label{eq:<a,a>a}
\end{equation}
\end{prop}

\begin{proof}
Since $\langle\alpha_{1},\alpha_{2}\rangle\neq0$, we need to consider
only the following cases:

\emph{Case 1}: $\alpha_{1},\alpha_{2},\alpha_{3}\in\mathfrak{a}$.
Consider the Jacoby identity for $(E_{1,2}+\varepsilon_{1}E_{2,1})\otimes\alpha_{1}$,
$(E_{1,2}+\varepsilon_{1}E_{2,1})\otimes\alpha_{2}$, $(E_{2,3}+\varepsilon_{2}E_{3,2})\otimes\alpha_{3}$
where $\varepsilon_{i}=\pm1$, then use Lemma \ref{total} to get
\begin{eqnarray*}
 &  & (\varepsilon_{1}E_{2,3}+\varepsilon_{1}\varepsilon_{2}E_{3,2})\otimes\alpha_{1}\circ(\alpha_{2}\circ\alpha_{3})+(\varepsilon_{1}E_{2,3}-\varepsilon_{1}\varepsilon_{2}E_{3,2})\otimes[\alpha_{1},\alpha_{2}\circ\alpha_{3}]\\
 &  & +(\varepsilon_{1}E_{2,3}-\varepsilon_{1}\varepsilon_{2}E_{3,2})\otimes\alpha_{1}\circ[\alpha_{2},\alpha_{3}]+(\varepsilon_{1}E_{2,3}+\varepsilon_{1}\varepsilon_{2}E_{3,2})\otimes[\alpha_{1},[\alpha_{2},\alpha_{3}]]\\
 &  & =2(\varepsilon_{1}E_{2,3}-\varepsilon_{1}\varepsilon_{2}E_{3,2})\otimes[\alpha_{1},\alpha_{2}]_{A^{-}}\circ\alpha_{3}+2\frac{(n-4)}{n}(\varepsilon_{1}E_{2,3}+\varepsilon_{1}\varepsilon_{2}E_{3,2})\otimes[[\alpha_{1},\alpha_{2}]_{A^{-}},\alpha_{3}]\\
 &  & +\frac{8\varepsilon_{1}}{n}(E_{2,3}+\varepsilon_{2}E_{3,2})\otimes\langle\alpha_{1},\alpha_{2}\rangle\alpha_{3}+(\varepsilon_{1}E_{2,3}+\varepsilon_{1}\varepsilon_{2}E_{3,2})\otimes\alpha_{2}\circ(\alpha_{1}\circ\alpha_{3})+(\varepsilon_{1}E_{2,3}-\varepsilon_{1}\varepsilon_{2}E_{3,2})\\
 &  & \otimes[\alpha_{2},\alpha_{1}\circ\alpha_{3}]+(\varepsilon_{1}E_{2,3}-\varepsilon_{1}\varepsilon_{2}E_{3,2})\otimes\alpha_{2}\circ[\alpha_{1},\alpha_{3}]+(\varepsilon_{1}E_{2,3}+\varepsilon_{1}\varepsilon_{2}E_{3,2})\otimes[\alpha_{2},[\alpha_{1},\alpha_{3}]].
\end{eqnarray*}
By collecting the coefficients of $E_{2,3}$ and using associativity
of $\mathfrak{a}$, we get $\langle\alpha_{1},\alpha_{2}\rangle\alpha_{3}=[[\alpha_{1},\alpha_{2}]_{A^{-}},\alpha_{3}]$.

\emph{Case 2}: $\alpha_{1},\alpha_{2}\in\mathfrak{a}$ and $\alpha_{3}\in B\oplus B'$.
First assume that $\alpha_{3}\in B$. Consider the Jacoby identity
for $(E_{1,2}+\varepsilon_{1}E_{2,1})\otimes\alpha_{1}$, $(E_{1,2}+\varepsilon_{1}E_{2,1})\otimes\alpha_{2}$,
$e_{1}\otimes\alpha_{3}$ where $\varepsilon_{1}=\pm1$, then use
Lemmas \ref{bb'} and \ref{total} to get $\varepsilon_{1}e_{1}\otimes(\alpha_{1}(\alpha_{2}\alpha_{3})+\frac{1}{2}(-2+\frac{4}{n})[\alpha_{1},\alpha_{2}]_{A^{-}}\alpha_{3}-\frac{2}{n}\langle\alpha_{1},\alpha_{2}\rangle\alpha_{3}-\alpha_{2}(\alpha_{1}\alpha_{3}))=0,$
and so 
\[
\alpha_{1}(\alpha_{2}\alpha_{3})-\frac{1}{2}(2-\frac{4}{n})[\alpha_{1},\alpha_{2}]_{A^{-}}\alpha_{3}-\frac{2}{n}\langle\alpha_{1},\alpha_{2}\rangle\alpha_{3}-\alpha_{2}(\alpha_{1}\alpha_{3})=0.
\]
 Since $[\alpha_{1},\alpha_{2}]_{A^{-}}\alpha_{3}=\alpha_{1}(\alpha_{2}\alpha_{3})-\alpha_{2}(\alpha_{1}\alpha_{3})$,
we get $\langle\alpha_{1},\alpha_{2}\rangle\alpha_{3}=[\alpha_{1},\alpha_{2}]_{A^{-}}\alpha_{3},$
as required. Similarly, one can show that $\langle\alpha_{1},\alpha_{2}\rangle\alpha_{3}=[\alpha_{1},\alpha_{2}]_{A^{-}}\alpha_{3}$
for $\alpha_{1},\alpha_{2}\in\mathfrak{a}$ and $\alpha_{3}\in B'$.

\emph{Case 3}: $\alpha_{1}\in B$, $\alpha_{2}\in B'$ and $\alpha_{3}\in\mathfrak{a}$.
Consider the Jacoby identity for $e_{1}\otimes\alpha_{1}$, $e_{1}\otimes\alpha_{2}$,
$(E_{1,2}+\varepsilon E_{2,1})\otimes\alpha_{3}$ where $\varepsilon=\pm1$,
then use (\ref{eq:formulla for natural elements}), Lemmas \ref{bb'}
and \ref{total} to get 
\begin{eqnarray*}
 &  & (E_{2,1}+E_{1,2})\otimes[\alpha_{1},\alpha_{2}\alpha_{3}]+(E_{1,2}-E_{2,1})\otimes\alpha_{1}\circ(\alpha_{2}\alpha_{3})=((E_{1,2}+\varepsilon E_{2,1})-\frac{2}{n}(E_{1,2}+\varepsilon E_{2,1}))\\
 &  & \otimes[[\alpha_{1},\alpha_{2}]_{A^{-}},\alpha_{3}]+(E_{1,2}-\varepsilon E_{2,1})\otimes[\alpha_{1},\alpha_{2}]_{A^{-}}\circ\alpha_{3}+\frac{4}{n}(E_{1,2}+\varepsilon E_{2,1})\otimes\langle\alpha_{1},\alpha_{2}\rangle\alpha_{3}\\
 &  & +\varepsilon(E_{2,1}+E_{1,2})\otimes[\alpha_{3}\alpha_{1},\alpha_{2}]+\varepsilon(E_{2,1}-E_{1,2})\otimes(\alpha_{3}\alpha_{1})\circ\alpha_{2}.
\end{eqnarray*}
By collecting the coefficients of $E_{1,2}$ we get
\[
\alpha_{1}(\alpha_{2}\alpha_{3})=[\alpha_{1},\alpha_{2}]_{A^{-}}\alpha_{3}-\varepsilon\alpha_{2}(\alpha_{3}\alpha_{1})-\frac{1}{n}[[\alpha_{1},\alpha_{2}]_{A^{-}},\alpha_{3}]+\frac{2}{n}\langle\alpha_{1},\alpha_{2}\rangle\alpha_{3}.
\]
 Since $[\alpha_{1},\alpha_{2}]_{A^{-}}\alpha_{3}=(\alpha_{1}\alpha_{2}-\alpha_{2}\alpha_{1})\alpha_{3}=(\alpha_{1}\alpha_{2})\alpha_{3}-(\alpha_{2}\alpha_{1})\alpha_{3}$,
$(\alpha_{1}\alpha_{2})\alpha_{3}=\alpha_{1}(\alpha_{2}\alpha_{3})$
and $(\alpha_{2}\alpha_{1})\alpha_{3}=\alpha_{2}(\eta(\alpha_{3})\alpha_{1})=-\varepsilon\alpha_{2}(\alpha_{3}\alpha_{1})$,
by using Proposition \ref{mod B} we get $\langle\alpha_{1},\alpha_{2}\rangle\alpha_{3}=\frac{1}{2}[[\alpha_{1},\alpha_{2}]_{A^{-}},\alpha_{3}]$,
as required.

\emph{Case 4:} $\alpha_{1}\in B,\alpha_{2}\in B'$ and $\alpha_{3}\in B$
(the case with $\alpha_{3}\in B'$ being similar). Consider the Jacoby
identity for $e_{2}\otimes\alpha_{3}$, $e_{1}\otimes\alpha_{2}$,
$e_{1}\otimes\alpha_{1}$ then use (\ref{eq:formulla for natural elements})
we get
\begin{eqnarray*}
 &  & [e_{2}\otimes\alpha_{3},\frac{1}{2}(2E_{11}-\frac{2}{n}I)\otimes[\alpha_{2},\alpha_{1}]_{A^{-}}+\frac{2}{n}\langle\alpha_{2},\alpha_{1}\rangle]=[\frac{1}{2}(E_{2,1}+E_{1,2})\otimes[\alpha_{3},\alpha_{2}]_{A^{-}}+\frac{1}{2}(E_{2,1}-E_{1,2})\\
 &  & \otimes(\alpha_{3}\circ\alpha_{2})_{A^{+}},e_{1}\otimes\alpha_{1}]+[e_{1}\otimes\alpha_{2},\frac{1}{2}(E_{2,1}+E_{1,2})\otimes[\alpha_{3},\alpha_{1}]_{C}+\frac{1}{2}(E_{2,1}-E_{1,2})\otimes(\alpha_{3}\circ\alpha_{1})_{E}].
\end{eqnarray*}
Using (\ref{main for}) and Lemma \ref{bb'} we get
\[
\frac{1}{n}e_{2}\otimes([\alpha_{2},\alpha_{1}]_{A^{-}}\alpha_{3}-\langle\alpha_{2},\alpha_{1}\rangle\alpha_{3})=\frac{1}{2}e_{2}\otimes([\alpha_{3},\alpha_{2}]_{A^{-}}\alpha_{1}+(\alpha_{3}\circ\alpha_{2})_{A^{+}}\alpha_{1}-[\alpha_{3},\alpha_{1}]_{C}\alpha_{2}-(\alpha_{3}\circ\alpha_{1})_{E}\alpha_{2}),
\]
so, $\langle\alpha_{2},\alpha_{1}\rangle\alpha_{3}=\frac{1}{2}([\alpha_{2},\alpha_{1}]_{A^{-}}\alpha_{3}+n((\alpha_{3}\alpha_{1})\alpha_{2}-(\alpha_{3}\alpha_{2})\alpha_{1})),$or
equivalently, $\langle\alpha_{1},\alpha_{2}\rangle\alpha_{3}=\frac{1}{2}([\alpha_{1},\alpha_{2}]_{A^{-}}\alpha_{3}+n((\alpha_{3}\alpha_{2})\alpha_{1}-(\alpha_{3}\alpha_{1})\alpha_{2}))$,
as required. 
\end{proof}
\begin{prop}
\label{derivation 2}(1) $[d,\langle\alpha,\beta\rangle]=\langle d\alpha,\beta\rangle+\langle\alpha,d\beta\rangle$
for all $\alpha,\beta\in\mathfrak{b}$ and $d\in D.$

(2) $\langle A^{+},A^{+}\rangle$, $\langle A^{-},A^{-}\rangle$,
$\langle B,B'\rangle$, $\langle C,C'\rangle$ and $\langle E,E'\rangle$
are ideals of the Lie algebra $D$.

(3) $D$ acts by derivations on $\mathfrak{b}$ and leaves all subspaces
$A^{+},A^{-},B,B',\dots,E,E'$ invariant.
\end{prop}

\begin{proof}
Let $\alpha=a_{1}^{+}+a_{1}^{-}+b_{1}+b_{1}'+c_{1}+c_{1}'+e_{1}+e_{1}'$
and $\beta=a_{2}^{+}+a_{2}^{-}+b_{2}+b_{2}'+c_{2}+c_{2}'+e_{2}+e_{2}'$
be the decompositions of $\alpha$ and $\beta$ into homogeneous parts.
By considering Jacobi identities for the following 5 triples,
\begin{align*}
 & (i)\ d,\,x_{1}^{+}\otimes a_{1}^{-},\,x_{2}^{+}\otimes a_{2}^{-};\quad(ii)\ d,\,x_{1}^{-}\otimes a_{1}^{+},\,x_{2}^{-}\otimes a_{2}^{+};\\
 & (iii)\ d,\,u\otimes b_{i},v'\otimes b_{j}';\quad(iv)\ d,\,s\otimes c,s'\otimes c';\quad(v)\ d,\,\lambda\otimes e,\lambda'\otimes e';
\end{align*}
we get the following equations, respectively,
\begin{multline}
[d,\langle a_{1}^{-},a_{2}^{-}\rangle]=\langle da_{1}^{-},a_{2}^{-}\rangle+\langle a_{1}^{-},da_{2}^{-}\rangle;\quad[d,\langle a_{1}^{+},a_{2}^{+}\rangle]=\langle da_{1}^{+},a_{2}^{+}\rangle+\langle a_{1}^{+},da_{2}^{+}\rangle;\\{}
[d,\langle b_{i},b_{j}'\rangle]=\langle db_{i},b_{j}'\rangle+\langle b_{i},db_{j}'\rangle;\;[d,\langle c_{i},c_{j}'\rangle]=\langle dc_{i},c_{j}'\rangle+\langle c_{i},dc_{j}'\rangle;\;[d,\langle e_{i},e_{j}'\rangle]=\langle de_{i},e_{j}'\rangle+\langle e_{i},de_{j}'\rangle.\label{eq:part1}
\end{multline}
and
\begin{multline}
d(a_{1}^{-}a_{2}^{-})=(da_{1}^{-})a_{2}^{-}+a_{1}^{-}(da_{2}^{-});\quad d(a_{1}^{+}a_{2}^{+})=(da_{1}^{+})a_{2}^{+}+a_{1}^{+}(da_{2}^{+});\\
d(b_{i}b_{j}')=(db_{i})b_{j}'+b(db_{j}');\quad d(c_{i}c_{j}')=(dc_{i})c_{j}'+c_{i}(dc_{j}');\quad d(e_{i}e_{j}')=(de_{i})e_{j}'+e_{i}(de_{j}'),\label{eq:part2}
\end{multline}
where $i,j=1,2$. We illustrate this by considering the case $(i)$.
By applying Jacobi identity to $d$, $x_{1}^{+}\otimes a_{1}^{-}$,
$x_{2}^{+}\otimes a_{2}^{-}$, we get 
\[
[d,[x_{1}^{+}\otimes a_{1}^{-},x_{2}^{+}\otimes a_{2}^{-}]]=[[d,x_{1}^{+}\otimes a_{1}^{-}],x\otimes a_{2}^{-}]+[x_{1}^{+}\otimes a_{1}^{-},[d,x_{2}^{+}\otimes a_{2}^{-}]]
\]
Using (\ref{main for}) and Lemma \ref{total} we get
\begin{multline*}
x_{1}^{+}\circ x_{2}^{+}\otimes d\frac{[a_{1}^{-},a_{2}^{-}]_{A^{-}}}{2}+[x_{1}^{+},x_{2}^{+}]\otimes d\frac{(a_{1}^{-}\circ a_{2}^{-})_{A^{+}}}{2}+(x_{1}^{+}\mid x_{2}^{+})[d,\langle a_{1}^{-},a_{2}^{-}\rangle]\\
=x_{1}^{+}\circ x_{2}^{+}\otimes\frac{[da_{1}^{-},a_{2}^{-}]_{A^{-}}}{2}+[x_{1}^{+},x_{2}^{+}]\otimes\frac{(da_{1}^{-}\circ a_{2}^{-})_{A^{+}}}{2}+(x_{1}^{+}\mid x_{2}^{+})\langle da_{1}^{-},a_{2}^{-}\rangle\\
+x_{1}^{+}\circ x_{2}^{+}\otimes\frac{[a_{1}^{-},da_{2}^{-}]_{A^{-}}}{2}+[x_{1}^{+},x_{2}^{+}]\otimes\frac{(a_{1}^{-}\circ da_{2}^{-})_{A^{+}}}{2}+(x_{1}^{+}\mid x_{2}^{+})\langle a_{1}^{-},da_{2}^{-}\rangle.
\end{multline*}
Then 
\begin{multline}
x_{1}^{+}\circ x_{2}^{+}\otimes d\frac{[a_{1}^{-},a_{2}^{-}]_{A^{-}}}{2}+[x_{1}^{+},x_{2}^{+}]\otimes d\frac{(a_{1}^{-}\circ a_{2}^{-})_{A^{+}}}{2}=x_{1}^{+}\circ x_{2}^{+}\otimes\frac{[da_{1}^{-},a_{2}^{-}]_{A^{-}}}{2}+\\{}
[x_{1}^{+},x_{2}^{+}]\otimes\frac{(da_{1}^{-}\circ a_{2}^{-})_{A^{+}}}{2}+x_{1}^{+}\circ x_{2}^{+}\otimes\frac{[a_{1}^{-},da_{2}^{-}]_{A^{-}}}{2}+[x_{1}^{+},x_{2}^{+}]\otimes\frac{(a_{1}^{-}\circ da_{2}^{-})_{A^{+}}}{2}\label{ideal 1}
\end{multline}
and 
\begin{equation}
(x_{1}^{+}\mid x_{2}^{+})[d,\langle a_{1}^{-},a_{2}^{-}\rangle]=(x_{1}^{+}\mid x_{2}^{+})(\langle da_{1}^{-},a_{2}^{-}\rangle+\langle a_{1}^{-},da_{2}^{-}\rangle).\label{ideal 2}
\end{equation}
When $x_{1}^{+}=x_{2}^{+}=E_{1,2}+E_{2,1}$, we have $\tra(x_{1}^{+}x_{2}^{+})=1$.
Hence (\ref{ideal 2}) is equivalent to $[d,\langle a_{1}^{-},a_{2}^{-}\rangle]=\langle da_{1}^{-},a_{2}^{-}\rangle+\langle a_{1}^{-},da_{2}^{-}\rangle.$
When $x_{1}^{+}=E_{1,2}+E_{2,1}$ and $x_{2}^{+}=E_{2,3}+E_{3,2}$,
we have $[x_{1}^{+},x_{2}^{+}]=E_{1,3}+E_{3,1}$ and $x_{1}^{+}\circ x_{2}^{+}=E_{1,3}+E_{3,1}$.
Hence (\ref{ideal 1}) is equivalent to:

\begin{align*}
d(\frac{[a_{1}^{-},a_{2}^{-}]_{A^{-}}}{2}+\frac{(a_{1}^{-}\circ a_{2}^{-})_{A^{+}}}{2}) & =\frac{[da_{1}^{-},a_{2}^{-}]_{A^{-}}}{2}+\frac{(da_{1}^{-}\circ a_{2}^{-})_{A^{+}}}{2}+\frac{[a_{1}^{-},da_{2}^{-}]_{A^{-}}}{2}+\frac{(a_{1}^{-}\circ da_{2}^{-})_{A^{+}}}{2},
\end{align*}
or equivalently, $d(a_{1}^{-}a_{2}^{-})=(da_{1}^{-})a_{2}^{-}+a_{1}^{-}(da_{2}^{-})$,
as in equation (\ref{eq:part1}).

By combining the equations (\ref{eq:part1}) we get $[d,\langle\alpha,\beta\rangle]=\langle d\alpha,\beta\rangle+\langle\alpha,d\beta\rangle,$
for all $d\in D$ and $\alpha,\beta\in\mathfrak{b}$. This implies
that the subspaces $\langle A^{+},A^{+}\rangle$, $\langle A^{-},A^{-}\rangle$,
$\langle B,B'\rangle$, $\langle C,C'\rangle$ and $\langle E,E'\rangle$
are ideals in $D.$ The equations (\ref{eq:part2}) show that $d$
acts by derivation. Similarly, one can show that $D$ acts by derivations
on $\mathfrak{b}$. Using Proposition \ref{derivation rule} and Tables
\ref{t4} and \ref{t5}, we see that the action of $D$ leaves all
subspaces $A^{+},A^{-},B,\dots,E'$ invariant as required. 
\end{proof}
The above results can be summarized as follows.
\begin{thm}[\textcolor{black}{The structure theorem for $\Theta_{n}$-graded }Lie
algebra\textcolor{black}{s}]
 \label{structure} Let $L$ be an $\Theta_{n}$-graded Lie algebra
and let $\mathfrak{g}\cong sl_{n}$ be the grading subalgebra of $L$.
Suppose that $n\ge7$ or $n=5,6$ and the conditions (\ref{eq:MainAssumptions})
hold. Then
\[
L=(\mathfrak{g}\otimes A)\oplus(V\otimes B)\oplus(V'\otimes B')\oplus(S\otimes C)\oplus(S'\otimes C')\oplus(\Lambda\otimes E)\oplus(\Lambda'\otimes E')\oplus D
\]
with multiplication given by (\ref{main for}) where $A,B,B',C,C',E,E'$
are vector spaces and $D$ is the sum of the trivial $\mathfrak{g}$-modules.
Define by $\mathfrak{g^{+}}:=\{x\in\mathfrak{g}\mid x^{t}=x\}$ and
$\mathfrak{g^{-}}:=\{x\in\mathfrak{g}\mid x^{t}=-x\}$ the subspaces
of symmetric and skew-symmetric matrices in $\mathfrak{g}$, respectively.
Then the component $\mathfrak{g}\otimes A$ can be decomposed further
as $\mathfrak{g}\otimes A=(\mathfrak{g^{+}}\oplus\mathfrak{g^{-}})\otimes A=(\mathfrak{g^{+}}\otimes A^{-})\oplus(\mathfrak{g}^{-}\otimes A^{+})$
where $A^{-}$ and $A^{+}$ are two copies of the vector space $A$.
Denote$\mathfrak{a}:=A^{+}\oplus A^{-}\oplus C\oplus E\oplus C'\oplus E'$
and $\mathfrak{b}:=\mathfrak{a}\oplus B\oplus B'$. Then the product
in $L$ induces an algebra structure on both $\mathfrak{a}$ and $\mathfrak{b}$
satisfying the following properties.

(i) $\mathfrak{a}$ is a unital associative subalgebra of $\mathfrak{b}$
with involution whose symmetric and skew-symmetric elements are $A^{+}\oplus E\oplus E'$
and $A^{-}\oplus C\oplus C'$, respectively, see Theorems \ref{associati p}
and \ref{a invol}.

(ii) $\mathfrak{b}$ is a unital algebra with an involution $\eta$
whose symmetric and skew-symmetric elements are $A^{+}\oplus E\oplus E'\oplus B\oplus B'$
and $A^{-}\oplus C\oplus C'$, respectively, see Theorem \ref{b involution}
and Proposition \ref{iden}.

(iii) $B\oplus B'$ is an associative $\mathfrak{a}$-bimodule with
a hermitian form $\chi$ with values in $\mathfrak{a}$. More exactly,
for all $\beta_{1},\beta_{2}\in B\oplus B'$ and $\alpha\in\mathfrak{a}$
we have $\chi(\beta_{1},\beta_{2})=\beta_{1}\beta_{2}$, $\chi(\alpha\beta_{1},\beta_{2})=\alpha\chi(\beta_{1},\beta_{2})$,
$\eta(\chi(\beta_{1},\beta_{2}))=\chi(\beta_{2},\beta_{1})$ and $\chi(\beta_{1},\alpha\beta_{2})=\chi(\beta_{1},\beta_{2})\eta(\alpha)$,
see Propositions \ref{b is A module} and \ref{mod B}.

(iv) $\mathcal{A}:=A^{-}\oplus A^{+}$ is a unital associative subalgebra
of $\mathfrak{a}$ and $C\oplus E$, $C'\oplus E'$, $B$ and $B'$
are $\mathcal{A}$-bimodules, see Corollaries \ref{cor A is sub}
and \ref{cor B is A (sub) }.

(v) $D$ acts by derivations on $\mathfrak{b}$, see Propositions
\ref{derivation rule} and \ref{derivation 2}.
\end{thm}

\subsection{Matrix realization of the algebra $\mathfrak{a}$}

Recall that $\mathfrak{g}\otimes A=\mathfrak{g^{+}}\otimes A^{-}\oplus\mathfrak{g}^{-}\otimes A^{+}$
where $\mathfrak{g}^{\pm}=\{x\in sl_{n}\mid x^{t}=\pm x\}$ and $A^{\pm}$
is a copy of the vector space $A$. We identify $\mathfrak{g}$ with
$\mathfrak{g}\otimes1$ where $1$ is a distinguished element of $A$.
We denote by $a^{\pm}$ the image of $a\in A$ in the space $A^{\pm}$.
Recall that $\mathcal{A}=A^{+}\oplus A^{-}$ is an associative algebra
with identity element $1^{+}$. Consider the subspaces $A_{1}=span\{a^{+}+a^{-}\mid a\in A\}$
and $A_{2}=span\{a^{+}-a^{-}\mid a\in A\}$. Then $\mathcal{A}=A_{1}\oplus A_{2}$
as a vector space. In this subsection we show that $A_{1}$ and $A_{2}$
are $2$-sided ideals of the algebra $\mathcal{A}$ and that the associative
algebra $\mathfrak{a}$ has a realization by $2\times2$ matrices
with entries in the components of $\mathfrak{a}$. We start with the
following observation.
\begin{lem}
\label{element and -} For all $a^{\pm}\in A^{\pm}$, $c\in C$, $c'\in C'$,
$e\in E$, $e'\in E'$, $b\in B$, $b'\in B'$ we have

$(1)\;a^{+}=1^{-}.a^{-}=a^{-}.1^{-}$ and $a^{-}=1^{-}.a^{+}=a^{+}.1^{-}$;

$(2)\;c=1^{-}.c=-c.1^{-}$ and $e=1^{-}.e=-e.1^{-}$;

$(3)\;c'=c'.1^{-}=-1^{-}.c'$ and $e'=e'.1^{-}=-1^{-}.e'$;

$(4)\;b=1^{-}b$ and $b'=b'.1^{-}$.
\end{lem}

\begin{proof}
We will only prove (1), the other statements being similar. Let $x^{+}\in\mathfrak{g}^{+}$
and $y^{\pm}\in\mathfrak{g}^{\pm}$. Using (\ref{eq:iden 1}) we get
\[
[x^{+}\otimes1^{-},y^{+}\otimes a^{-}]=[x^{+},y^{+}]\otimes a^{+}\text{ and }[x^{+}\otimes1^{-},y^{-}\otimes a^{+}]=[x^{+},y^{-}]\otimes a^{-}.
\]
Using these relations and the formulas in Remark \ref{FF}, we get
\begin{align*}
[x^{+},y^{+}]\otimes a^{+} & =x^{+}\circ y^{+}\otimes\frac{[1^{-},a^{-}]_{A^{-}}}{2}+[x^{+},y^{+}]\otimes\frac{(1^{-}\circ a^{-})_{A^{+}}}{2}+(x^{+}\mid y^{+})\langle1^{-},a^{-}\rangle\\{}
[x^{+},y^{-}]\otimes a^{-} & =.x^{+}\diamond y^{-}\otimes\frac{[1^{-},a^{+}]_{A^{+}}}{2}+[x^{+},y^{-}]\otimes\frac{(1^{-}\circ a^{+})_{A^{-}}}{2}
\end{align*}
so $a^{+}=\frac{(1^{-}\circ a^{-})_{A^{+}}}{2}$, $\frac{[1^{-},a^{-}]_{A^{-}}}{2}=0$,
$a^{-}=\frac{(1^{-}\circ a^{+})_{A^{-}}}{2}$ and $\frac{[1^{-},a^{+}]_{A^{+}}}{2}=0$.
This implies (1), as required.
\end{proof}
\begin{prop}
\label{pdec} Let $e_{1}=\frac{1^{+}+1^{-}}{2}$ and $e_{2}=\frac{1^{+}-1^{-}}{2}$.
Then the following hold.

(1) $e_{1}$ and $e_{2}$ are orthogonal idempotents with $e_{1}+e_{2}=1^{+}$
and $\eta(e_{1})=e_{2}$.

(2) Let $\mathfrak{a}=e_{1}\mathfrak{a}e_{1}\oplus e_{1}\mathfrak{a}e_{2}\oplus e_{2}\mathfrak{a}e_{1}\oplus e_{2}\mathfrak{a}e_{2}$
be the Peirce decomposition of $\mathfrak{a}$. Then $e_{1}\mathfrak{a}e_{1}=A_{1}$,
$e_{1}\mathfrak{a}e_{2}=C\oplus E$, $e_{2}\mathfrak{a}e_{1}=C'\oplus E'$,
and $e_{2}\mathfrak{a}e_{2}=A_{2}$.

(3) $A_{1}$ and $A_{2}$ are $2$-sided ideals of\textup{ $\mathcal{A}=A_{1}\oplus A_{2}$.}

(4) $e_{i}$ is the identity of $A_{i}$.

(5) $\eta(A_{1})=A_{2}$.

(6) $B=\mathcal{B}e_{2}$ and $B'=\mathcal{B}e_{1}$.

(7) $A_{1}\cong A$ and $A_{2}\cong A^{op}$ (the opposite algebra
of $A$) as algebras.
\end{prop}

\begin{proof}
(1)-(6) This is easy to check using Lemma \ref{element and -} and
properties of the Peirce decomposition.

(7) Define the map $\mathfrak{\varphi:}A\rightarrow A_{1}$ by $\varphi(a)=\frac{a^{+}+a^{-}}{2}$
where $a\in A$. Note that this map is well defined and bijective.
It remains only to check that $\varphi$ is an algebra homomorphism.
Let $a,b\in A$. Then 
\begin{align*}
\varphi(ab) & =\varphi(\frac{a\circ b}{2}+\frac{[a,b]}{2})=\varphi(\frac{a\circ b}{2})+\varphi(\frac{[a,b]}{2})=(\frac{a\circ b}{4})^{+}+(\frac{a\circ b}{4})^{-}+(\frac{[a,b]}{4})^{+}+(\frac{[a,b]}{4})^{-}\\
 & =\frac{1}{4}(a^{+}a^{+}+a^{+}a^{-}+a^{-}a^{+}+a^{-}a^{-})=(\frac{a^{+}+a^{-}}{2})(\frac{a^{+}+a^{-}}{2})=\varphi(a)\varphi(b),
\end{align*}
so $\varphi$ is a homomorphism. Thus, $A_{1}\cong A$ and $A_{2}=\eta(A_{1})\cong A^{op}$,
as required.
\end{proof}
Using Peirce decomposition of $\mathfrak{a}$ as in Proposition \ref{pdec}
we immediately get the following.
\begin{prop}
The associative algebra $\mathfrak{a}$ has the following realization
by $2\times2$ matrices with entries in the components of $\mathfrak{a}$:
$\mathfrak{a}\cong\left[\begin{array}{cc}
A_{1} & C\oplus E\\
C'\oplus E' & A_{2}
\end{array}\right].$ In particular,
\begin{align*}
A^{+} & \cong\left\{ \left[\begin{array}{cc}
a_{1} & 0\\
0 & \eta(a_{1})
\end{array}\right]\mid a_{1}\in A_{1}\right\} \quad(a^{+}\mapsto\frac{1}{2}\left[\begin{array}{cc}
a^{+}+a^{-} & 0\\
0 & a^{+}-a^{-}
\end{array}\right]),\\
A^{-} & \cong\left\{ \left[\begin{array}{cc}
a_{1} & 0\\
0 & -\eta(a_{1})
\end{array}\right]\mid a_{1}\in A_{1}\right\} \quad(a^{-}\mapsto\frac{1}{2}\left[\begin{array}{cc}
a^{+}+a^{-} & 0\\
0 & -a^{+}+a^{-}
\end{array}\right]).
\end{align*}
\end{prop}

Let $A$ be an associative algebra with involution $\sigma$ (of the
first kind) over $F$. Recall that $A$ becomes a Lie algebra $A^{(-)}$
under the Lie bracket $[x,y]=xy-yx$. Let $sym(A)$ (resp. $skew(A)$)
denotes the set of symmetric elements (resp. skew-symmetric elements)
of $A$ with respect to $\sigma$. Then, $skew(A)$ is a Lie subalgebra
of $A^{(-)}$. The following is well known.
\begin{lem}
\label{involution} Let $A_{1}$ and $A_{2}$ be two associative algebras
with involutions $\sigma_{1}$ and $\sigma_{2}$, respectively. Then
$A=A_{1}\otimes A_{2}$ is an associative algebra with involution
$\sigma=\sigma_{1}\otimes\sigma_{2}$. Moreover, we have

$(1)\,sym(A)=sym(A_{1})\otimes sym(A_{2})\oplus skew(A_{1})\otimes skew(A_{2}).$

$(2)\,skew(A)=skew(A_{1})\otimes sym(A_{2})\oplus sym(A_{1})\otimes skew(A_{2}).$
\end{lem}

\begin{proof}
It is easy to see that the right-hand side of the equation (1) (resp.
(2)) is a subspace of $sym(A)$ (resp. $skew(A)$). It remains to
note that
\begin{align*}
A_{1}\otimes A_{2} & =(sym(A_{1})\oplus skew(A_{1}))\otimes(sym(A_{2})\oplus skew(A_{2}))=sym(A_{1})\otimes sym(A_{2})\\
 & \oplus skew(A_{1})\otimes skew(A_{2})\oplus skew(A_{1})\otimes sym(A_{2})\oplus sym(A_{1})\otimes skew(A_{2}).
\end{align*}
\end{proof}

\section{\label{chap:5} Central extensions and classification of $\Theta_{n}$-graded
Lie algebras}

Recall that a \emph{central extension }of a Lie algebra $L$ is a
pair $(\tilde{L},\pi)$ consisting of a Lie algebra $\tilde{L}$ and
a surjective Lie algebra homomorphism $\pi:\tilde{L}\rightarrow L$
whose kernel lies in the center of $\tilde{L}$. A \emph{cover} or
\emph{covering} of $L$ is a central extension $(\tilde{L},\pi)$
of $L$ with $\tilde{L}$ perfect, i.e., $\tilde{L}=[\tilde{L},\tilde{L}]$.
A homomorphism of central extensions from the central extension $f:K\rightarrow L$
to the central extension $f':K'\rightarrow L$ is a Lie algebra homomorphism
$g:K\rightarrow K'$ satisfying $f=f'\circ g$. A central extension
$U:K\rightarrow L$ is \emph{universal}, if there exists a unique
homomorphism from $K$ to any other central extension $\tilde{K}$
of $L$. Any perfect Lie algebra $L$ has a unique universal central
extension, which is also perfect, called a \emph{universal covering
algebra} of $L$. Two perfect Lie algebras $L_{1}$ and $L_{2}$ are
said to be \emph{centrally isogenous} if they have the same universal
covering algebra (up to isomorphism). The aim of this section is to
classify $\Theta_{n}$-graded Lie algebras up to isomorphism and describe
their central extensions. 

This section is organized as follows. First we study basic properties
of central extensions of $(\Gamma,\mathfrak{g})$-graded Lie algebras.
Then we focus our attention to $(\Theta_{n},sl_{n})$-graded Lie algebras.
We define a centerless algebra $\mathcal{L}(\mathfrak{b})$ and show
that it is $\Theta_{n}$-graded with coordinate algebra $\mathfrak{b}$.
We also show that any $\Theta_{n}$-graded Lie algebra $L$ with coordinate
algebra $\mathfrak{b}$ is a cover of the centerless Lie algebra $\mathcal{L}(\mathfrak{b})$
and $L$ is uniquely determined (up to central isogeny) by its ``coordinate''
algebra $\mathfrak{b}$. Moreover, $L$ is centrally isogenous to
the explicitly constructed $\Theta_{n}$-graded unitary Lie algebra
$\mathfrak{u}$ of the hermitian form $\xi=w\bot-\chi$ on the $\mathcal{\mathfrak{a}}$-module
$\mathfrak{a}^{n}\oplus\mathcal{B}$. This completes the classification
of $\Theta_{n}$-graded Lie algebras up to central extensions.  At
the end we classify the $\Theta_{n}$-graded Lie algebras up to isomorphism. 

\subsection{Central extensions of $(\Gamma,\mathfrak{g})$-graded Lie algebras}

Central extensions of root graded and $BC_{r}$-graded Lie algebras
in terms of the homology of its coordinate algebra were determined
and described up to isomorphism by Allison, Benkart and Y. Gao in
\cite{allison2000central} and \cite{allison2002lie}. We mostly adopt
their approach here and follow their notations whenever possible.
Theorems \ref{universal covering is graded} and \ref{bilinea Gama}
and Lemma \ref{lifting 1 } below are natural generalizations of \cite[Proposition 1.24]{berman1992lie}
and \cite[3.1-3.4, 3.7  ]{allison2000central}, respectively, and
are proved exactly in the same way (see also \cite[5.1.2, 5.1.4, 5.1.5]{yaseen2018generalized}). 

First we note that every $\Gamma$-graded Lie algebra is perfect,
so it has a unique universal covering algebra \cite[7.9.2]{weibel1995introduction}.
\begin{thm}
\label{gama perfect} Let $L$ be a $(\Gamma,\mathfrak{g})$-graded
Lie algebra. Then $L$ is perfect. 
\end{thm}

\begin{proof}
We need to show $L\subseteq[L,L]$, i.e. $L_{\alpha}\subseteq\left[L,L\right]$
for all $\alpha\in\Gamma$. By condition $(\Gamma3)$ in Definition
\ref{def of gamma}, $L_{0}\subseteq\left[L,L\right]$. Suppose now
that $\alpha\in\Gamma\setminus\{0\}$. Then there exists $h\in H$
such that $\alpha(h)\neq0$ so for all $x\in L_{\alpha}$, $[h,x]=\alpha(h)x$
and $x=[\alpha(h)^{-1}h,x]\in\left[L_{0},L_{\alpha}\right]$. Thus,
$L_{\alpha}\subseteq\left[L_{0},L_{\alpha}\right]$, as required. 
\end{proof}
\begin{thm}
\label{universal covering is graded} Let $L$ be a $(\Gamma,\mathfrak{g})$-graded
Lie algebra and let $(U,\psi)$ be the universal covering algebra
of $L$. Then $U$ is graded by $\Gamma$ and $\psi\mid_{U_{\alpha}}U_{\alpha}\rightarrow L_{\alpha}$
is an isomorphism for all $\alpha\in\Gamma\setminus\{0\}$. In particular
$\Ker\psi\subset U_{0}$.
\end{thm}

\begin{cor}
\label{cor:iclass}(1) Let $(U,\psi)$ be the universal covering algebra
of $L$. Then $U$ is $(\Gamma,\mathfrak{g})$-graded if and only
if $L$ is $(\Gamma,\mathfrak{g})$-graded.

(2) All Lie algebras in a given isogeny class are $\Gamma$-graded
if one of them is, and all have isomorphic weight spaces for non-zero
weights.
\end{cor}

Let $\pi:$ $\tilde{L}\rightarrow L$ be a central extension with
kernel $\mathbb{E}$. Then we can lift $L$ to a subspace of $\tilde{L}$
which is mapped isomorphically to $L$ by $\pi$. We identify this
subspace with $L$. Then $\tilde{L}=L\oplus\mathbb{E}$ and the multiplication
on $\tilde{L}$ is given by $[f,\mathfrak{g}\tilde{]}=[f,\mathfrak{g}]+\zeta(f,\mathfrak{g}),f,\mathfrak{g}\in L,$
where $[f,g]$ denotes the product in $L$ and $\zeta$$:L\times L\rightarrow\mathbb{E}$
is a \emph{2-cocycle} on $L$, i.e. a bilinear map satisfying, for
all $f,g,h\in L,$
\begin{equation}
(i)\;\zeta(f,g)=-\zeta(\mathfrak{g},f),\quad(ii)\ \zeta([f,g],h)+\zeta([g,h],f)+\zeta([h,f],g)=0.\label{eq:cocyc}
\end{equation}

\begin{lem}
\label{lifting 1 } Suppose that $\pi:$ $\tilde{L}\rightarrow L$
is a central extension of a $(\Gamma,\mathfrak{g})$-graded Lie algebra
$L$ with kernel $\mathbb{E}$. Then there is lifting of the grading
subalgebra $\mathfrak{g}$ of $L$ to a subalgebra of $\tilde{L}$.
Moreover, $L$ can be lifted to a subspace $L$ of $\tilde{L}$ which
contains the given $\mathfrak{g}$ so that the corresponding $2$-cocycle
satisfies $\zeta(\mathfrak{g},L)=0$.
\end{lem}

Let $M$ be an irreducible $\ffg$-module and let $M'$ be its dual.
Let $\pi:M\times M'\to\bbF$ be any non-degenerate $\ffg$-invariant
bilinear form. Note that $\pi$ is unique up to a scalar multiple
as $\hom(M\otimes M',\bbF)\cong\hom(M,M)\cong\bbF$. Set $\pi(M,N)=0$
if $M$ and $N$ are irreducible and $N\not\cong M'$. 
\begin{thm}
\label{bilinea Gama} Let $L$ be a $(\Gamma,\mathfrak{g})$-graded
Lie algebra and let the $\mathfrak{g}$-module $L=\bigoplus_{\mu\in Q}V(\mu)\otimes W_{\mu}$
for some vector spaces $W_{\mu}$. Let $\tilde{L}=L\oplus\mathbb{E}$
be a central extension of $L$ determined by the $2$-cocycle $\zeta(\,,\,)$$:L\times L\rightarrow\mathbb{E}$
with $\zeta(\mathfrak{g},L)=0$. Then,

(1) $V(\mu)$ and $V(\nu)$ ($\mu,\nu\in Q$) are orthogonal relative
to $\zeta(\,,\,)$ whenever $V(\mu)\ncong V(\nu)'$ as $\mathfrak{g}$-modules;

(2) there exists an $\bbF$-bilinear map $\epsilon:W\times W\rightarrow\mathbb{\mathbb{E}}$
on the space $W:=\bigoplus_{\mu\in Q\setminus\{0\}}W_{\mu}$ with
$\epsilon(W_{\mu},W_{\nu})=0$ whenever $V(\mu)\ncong V(\nu)'$, such
that $\zeta(u_{\mu}\otimes w_{\mu},v_{\nu}\otimes w_{\nu})=\pi(u_{\mu},u_{\nu})\epsilon(w_{\mu},w_{\nu})$
for all $u_{\mu}\otimes w_{\mu}\in V(\mu)\otimes W_{\mu}$ and $u_{\nu}\otimes w_{\nu}\in V(\nu)\otimes W_{\nu}$.
\end{thm}

By a $2$-\emph{cocycle} on the algebra $\ffb$ we mean an $\bbF$-bilinear
map $\epsilon:\ffb\times\ffb\rightarrow\mathbb{E}$ into the $\bbF$-vector
space $\mathbb{E}$ satisfying for all $\beta_{1},\beta_{2},\beta_{3}\in\ffb,$
\begin{equation}
(i)\;\epsilon(\beta_{1},\beta_{2})=-\epsilon(\beta_{2},\beta_{1}),\quad(ii)\ \epsilon(\beta_{1}\beta_{2},\beta_{3})+\epsilon(\beta_{2}\beta_{3},\beta_{1})+\epsilon(\beta_{3}\beta_{1},\beta_{2})=0.\label{eq:cocycb}
\end{equation}

\begin{thm}
\label{lift 2} Assume that $\tilde{L}=L\oplus\mathbb{E}$ is a central
extension of the $\Theta_{n}$-graded Lie algebra $L=(\mathfrak{g}\otimes A)\oplus(V\otimes B)\oplus\dots\oplus(\Lambda'\otimes E')\oplus D$
determined by the $2$-cocycle $\zeta(\,,\,)$$:L\times L\rightarrow\mathbb{E}$
with $\zeta(\mathfrak{g},L)=0$. Then,

(1) $V(\mu)$ and $V(\nu)$ ($\mu,\nu\in\Theta_{n}^{+}$) are orthogonal
relative to $\zeta(\,,\,)$ whenever $V(\mu)\ncong V(\nu)'$ as $\mathfrak{g}$-modules;

(2) there exists a $2$-cocycle $\epsilon:\ffb\times\ffb\rightarrow\mathbb{E}$
on the algebra $\ffb$ with $\epsilon(W_{\mu},W_{\nu})=0$ whenever
$V(\mu)\ncong V(\nu)'$, such that 
\begin{eqnarray}
(a)\  & \zeta(x^{\pm}\otimes a_{1}^{\mp},y^{\pm}\otimes a_{2}^{\mp}) & =(x^{\pm}\mid y^{\pm})\epsilon(a_{1}^{\mp},a_{2}^{\mp}),\label{eq:lift2 equation}\\
(b)\  & \zeta(s\otimes c,s'\otimes c') & =(s\mid s')\epsilon(c,c'),\nonumber \\
(c)\  & \zeta(\lambda\otimes e,\lambda'\otimes e') & =(\lambda\mid\lambda')\epsilon(e,e'),\nonumber \\
(d)\  & \zeta(v\otimes b,v'\otimes b') & =\frac{2}{n}\tra(uv'^{t})\epsilon(b,b'),\nonumber \\
(e)\  & \zeta(d,\langle\beta,\beta'\rangle) & =\epsilon(d\beta,\beta')+\epsilon(\beta,d\beta')=-\zeta(\langle\beta,\beta'\rangle,d),\nonumber 
\end{eqnarray}
for all $x,y\in\ffg$, $v\in V$, $v'\in V'$, $s\in S$, $\lambda\in\Lambda$,
$s'\in S'$, $\lambda'\in\Lambda'$ and for all $a_{1}^{\mp},a_{2}^{\mp}\in A^{\mp}$,
$b\in B$, $b'\in B'$, $c\in C$, $c'\in C'$, $e\in E$, $e'\in E'$,
$\beta,\beta'\in\mathfrak{b}$ and $d\in D$.
\end{thm}

\begin{proof}
Let $W:=A\oplus C\oplus E\oplus C'\oplus E'\oplus B\oplus B'$. By
Theorem \ref{bilinea Gama}, (1) holds and  there exists an $\bbF$-bilinear
map $\epsilon:W\times W\rightarrow\mathbb{\mathbb{E}}$ such that
\begin{eqnarray*}
(a)\  & \zeta(x\otimes a_{1},y\otimes a_{2}) & =(x\mid y)\epsilon(a_{1},a_{2}),\quad(b)\ \zeta(s\otimes c,s'\otimes c')=(s\mid s')\epsilon(c,c'),\\
(c)\  & \zeta(\lambda\otimes e,\lambda'\otimes e') & =(\lambda\mid\lambda')\epsilon(e,e'),\quad(d)\ \zeta(v\otimes b,v'\otimes b')=\frac{2}{n}\tra(uv'^{t})\epsilon(b,b'),
\end{eqnarray*}
for all $x,y\in\mathfrak{g}$, $v\in V$, $v'\in V'$, $s\in S$,
$\lambda\in\Lambda$, $s'\in S'$, $\lambda'\in\Lambda'$ and for
all $a_{1},a_{2}\in A$, $b\in B$, $b'\in B'$, $c\in C$, $c'\in C'$,
$e\in E$, $e'\in E'$, $\beta,\beta'\in\mathfrak{b}$ and $d\in D$.
Since $(x^{+}\mid x^{-})=0$ for all $x^{\pm}\in\mathfrak{g}^{\pm}$,
we can extend the mapping $\epsilon$ to the algebra $\ffb=A^{+}\oplus A^{-}\oplus C\oplus E\oplus C'\oplus E'\oplus B\oplus B'$
by defining $\epsilon(a_{1}^{+},a_{2}^{-})=\epsilon(a_{1}^{-},a_{2}^{+})=0$
, $\epsilon(a_{1}^{\pm},a_{2}^{\pm})=\epsilon(a_{1},a_{2})$ and $\epsilon(a^{\pm},\alpha)=\epsilon(\alpha,a^{\pm})=0$
for all $a_{1},a_{2}\in A$, $a^{\pm}\in A^{\pm}$ and $\alpha\in C\oplus E\oplus C'\oplus E'\oplus B\oplus B'$.
Thus, we obtain an $\bbF$-bilinear map taking $\mathfrak{b}\times\mathfrak{b}$
to $\mathbb{E}$. Applying the 2-cocycle relation $\zeta([f,g],h)+\zeta([g,h],f)+\zeta([h,f],\mathfrak{g})=0$
and using the orthogonality of some of the components, we determine
that $\epsilon(\,,\,)$ is a 2- cocycle of $\mathfrak{b}$. We illustrate
these calculations by considering homogeneous elements $\alpha_{1},$$\alpha_{2}$
and $\alpha_{3}$ in $\mathfrak{a}$. Set 
\[
z_{1}=E_{1,2}+\varepsilon_{1}E_{2,1},\ z_{2}=E_{2,3}+\varepsilon_{2}E_{3,2}\text{ and }z_{3}=E_{3,1}+\varepsilon_{3}E_{1,3}\text{ where }\varepsilon_{i}=\pm1.
\]
The sign of each $\varepsilon_{i}$ can be chosen in such a way that
$z_{i}\otimes\alpha_{i}$ belongs to the corresponding homogeneous
component of $L$. Note that $\tra(z_{i}z_{j})=0$ for all $i\neq j$.
Hence by Lemma \ref{total}, we have $[z_{i}\otimes\alpha_{i},z_{j}\otimes\alpha_{j}]=z_{i}\diamond z_{j}\otimes\frac{[\alpha_{i},\alpha_{j}]}{2}+[z_{i},z_{j}]\otimes\frac{\alpha_{i}\circ\alpha_{j}}{2}.$
Then from (\ref{eq:cocyc}) with $z_{1}\otimes\alpha_{1}$, $z_{2}\otimes\alpha_{2}$,
$z_{1}\otimes\alpha_{3}$, we obtain
\begin{multline*}
([z_{1},z_{2}]\mid z_{3})\epsilon(\alpha_{1}\circ\alpha_{2},\alpha_{3})+((z_{1}\diamond z_{2})\mid z_{3})\epsilon([\alpha_{1},\alpha_{2}],\alpha_{3})+([z_{2},z_{3}]\mid z_{2})\epsilon([\alpha_{2},\alpha_{3}],\alpha_{1})\\
+(z_{2}\diamond z_{3}\mid z_{2})\epsilon([\alpha_{2},\alpha_{3}],\alpha_{1})+([z_{3},z_{1}]\mid z_{2})\epsilon(\alpha_{3}\circ\alpha_{1},\alpha_{2})+(z_{3}\diamond z_{1}\mid z_{2})\epsilon([\alpha_{3},\alpha_{1}],\alpha_{2})=0.
\end{multline*}
Using the fact that $(z\mid y)=\frac{1}{n}\tra(zy)$, it is easy to
verify that the form is associative relative to the \textquotedblleft $\diamond$\textquotedblright{}
product, (i.e. $(z\diamond y\mid z)=(z\mid y\diamond z)$ holds for
all $x,y,z\in\mathfrak{g}\cup S\cup S'\cup\Lambda\cup\Lambda'$),
and also relative to the commutator product. Thus,
\begin{multline*}
([z_{1},z_{2}]\mid z_{3})(\epsilon(\alpha_{1}\circ\alpha_{2},\alpha_{3})+\epsilon(\alpha_{2}\circ\alpha_{3},\alpha_{1})+\epsilon(\alpha_{3}\circ\alpha_{1},\alpha_{2}))\\
+(z_{1}\diamond z_{2}\mid z_{3})(\epsilon([\alpha_{1},\alpha_{2}],\alpha_{3})+\epsilon([\alpha_{2},\alpha_{3}],\alpha_{1})+\epsilon([\alpha_{3},\alpha_{1}],\alpha_{2}))=0.
\end{multline*}
Note that $\varepsilon_{1}\varepsilon_{2}\varepsilon_{3}=\pm1$, $[z_{1},z_{2}]z_{3}=E_{11}-\varepsilon_{1}\varepsilon_{2}\varepsilon_{3}E_{33}$
and $(z_{1}\diamond z_{2})z_{3}=E_{11}+\varepsilon_{1}\varepsilon_{2}\varepsilon_{3}E_{33}$.
If $\varepsilon_{1}\varepsilon_{2}\varepsilon_{3}=1$, then
\begin{equation}
\epsilon([\alpha_{1},\alpha_{2}],\alpha_{3})+\epsilon([\alpha_{2},\alpha_{3}],\alpha_{1})+\epsilon([\alpha_{3},\alpha_{1}],\alpha_{2})=0\label{eq:co1}
\end{equation}
 and we have four cases: $\varepsilon_{1}=\varepsilon_{2}=\varepsilon_{3}=1$;
$\varepsilon_{1}=1$ and $\varepsilon_{2}=\varepsilon_{3}=-1$; $\varepsilon_{1}=\varepsilon_{2}=-1$
and $\varepsilon_{3}=1$; $\varepsilon_{1}=\varepsilon_{3}=-1$ and
$\varepsilon_{2}=1$. In each of these cases $\epsilon(\alpha_{1}\circ\alpha_{2},\alpha_{3})=\epsilon(\alpha_{2}\circ\alpha_{3},\alpha_{1})=\epsilon(\alpha_{3}\circ\alpha_{1},\alpha_{2})=0$
(see Table \ref{t4}), so 
\begin{equation}
\epsilon(\alpha_{1}\circ\alpha_{2},\alpha_{3})+\epsilon(\alpha_{2}\circ\alpha_{3},\alpha_{1})+\epsilon(\alpha_{3}\circ\alpha_{1},\alpha_{2})=0\label{eq:co 2}
\end{equation}
as well. Adding equations (\ref{eq:co1}) and (\ref{eq:co 2}) gives
the desired 2-cocycle condition.

If $\varepsilon_{1}\varepsilon_{2}\varepsilon_{3}=-1$ , then
\begin{equation}
\epsilon(\alpha_{1}\circ\alpha_{2},\alpha_{3})+\epsilon(\alpha_{2}\circ\alpha_{3},\alpha_{1})+\epsilon(\alpha_{3}\circ\alpha_{1},\alpha_{2})=0\label{eq:co1-1}
\end{equation}
 and we have four cases: $\varepsilon_{1}=\varepsilon_{2}=\varepsilon_{3}=-1$;
$\varepsilon_{1}=-1$ and $\varepsilon_{2}=\varepsilon_{3}=1$; $\varepsilon_{1}=\varepsilon_{2}=1$
and $\varepsilon_{3}=-1$; $\varepsilon_{1}=\varepsilon_{3}=1$ and
$\varepsilon_{2}=-1$. In each of these cases $\epsilon([\alpha_{1},\alpha_{2}],\alpha_{3})=\epsilon([\alpha_{2},\alpha_{3}],\alpha_{1})=\epsilon([\alpha_{3},\alpha_{1}],\alpha_{2})=0$
(see Table \ref{t4}), so 
\begin{equation}
\epsilon([\alpha_{1},\alpha_{2}],\alpha_{3})+\epsilon([\alpha_{2},\alpha_{3}],\alpha_{1})+\epsilon([\alpha_{3},\alpha_{1}],\alpha_{2})=0\label{eq:co 2-1}
\end{equation}
 as well. Adding equations (\ref{eq:co1-1}) and (\ref{eq:co 2-1})
gives the desired 2-cocycle condition.

To prove (e), consider the 2-cocycle relation (\ref{eq:cocyc}) for
the elements $E_{1,2}+\varepsilon E_{2,1}\otimes\alpha_{1}$, $E_{1,2}+\varepsilon E_{2,1}\otimes\alpha_{2}$,
$d$ where $\varepsilon=\pm1$ and use Lemma \ref{total} (resp.$e_{1}\otimes b$,
$e_{1}\otimes b'\otimes\alpha_{2}$, $d$ and use (\ref{eq:formulla for natural elements})).
\end{proof}

\subsection{\label{sec:Classification-of--graded}Classification of $\Theta_{n}$-graded
Lie algebras, $n\ge5$}

We construct a centerless algebra $\mathcal{L}(\mathfrak{b})$ and
show that it is a $\Theta_{n}$-graded Lie algebra with coordinate
algebra $\mathfrak{b}$. Instead of proving directly that $\mathcal{L}(\mathfrak{b})$
satisfies the Jacoby identity (which is quite tedious and lengthy),
we construct an explicit example of a $\Theta_{n}$-graded Lie algebra
$\mathfrak{u}$ such that $\mathfrak{u}$ modulo its center is isomorphic
to $\mathcal{L}(\mathfrak{b})$, see Example \ref{exa:Model x}. We
prove that any $\Theta_{n}$-graded Lie algebra $L$ with coordinate
algebra $\mathfrak{b}$ is a cover of the centerless Lie algebra $\mathcal{L}(\mathfrak{b})$.
We show that every $\Theta_{n}$-graded Lie algebra $L$ is uniquely
determined (up to central isogeny) by its coordinate algebra $\mathfrak{b}$
and $L$ is centrally isogenous to the $\Theta_{n}$-graded unitary
Lie algebra $\mathfrak{u}$ of the hermitian form $\xi=w\bot-\chi$
on the $\mathcal{\mathfrak{a}}$-module $\mathfrak{a}^{n}\oplus\mathcal{B}$
(Proposition \ref{model ex} and Theorem \ref{model theorem}).
\begin{defn}
\cite[2.2]{allison2008unitary} Let $A$ be an associative algebra
with involution $\eta$. A map $\xi:X\times X\rightarrow A$ is called
a \emph{hermitian form} over $A$ if $X$ is a right $A$-module and
$\xi:X\times X\rightarrow A$ is a bi-additive map such that $\xi(xa,y)=\eta(a)\xi(x,y)$,
$\xi(x,ya)=\xi(x,y)a$ and $\xi(y,x)=\eta(\xi(x,y)),$ for $a\in A$
and $x,y\in X$. If $Y$ is an $A$-submodule of $X$, then $Y^{\bot}:=\left\{ x\text{\ensuremath{\in}}X\mid\xi(x,y)=0\text{ for all }y\text{\ensuremath{\in}}Y\right\} $
is also an $A$-submodule of $X.$ The form $\xi$ is said to be \emph{nondegenerate}
if $X^{\bot}=0$.
\end{defn}

\begin{defn}
\cite[4.1.1]{allison2008unitary} Let $A$ be an associative algebra
with involution. Suppose that $\xi:X\times X\rightarrow A$ is a hermitian
form over $A$. Let $\mathfrak{U}(X,\xi)=\left\{ T\in\End_{A}(X)\mid\xi(T(u),v)+\xi(u,T(v))=0,\ \forall\,u,v\in X\right\} $.
Then $\mathfrak{U}(X,\xi)$ is a Lie subalgebra of $\End_{A}(X)$,
and we say that $\mathfrak{U}(X,\xi)$ is the \emph{unitary }Lie algebra
of $\xi$.
\end{defn}

\begin{example}[Models of $\Theta_{n}$-graded Lie algebras, $n\ge5$]
\label{exa:Model x}  Let $\mathfrak{a}$ be any associative algebra
with involution $\eta$, identity element $1^{+}$ and two orthogonal
idempotents $e_{1}$ and $e_{2}$ such that $1^{+}=e_{1}+e_{2}$ and
$e_{2}=\eta(e_{1})$. Let $\mathcal{B}$ be any unital associative
right $\mathfrak{a}$-module with a hermitian form $\chi$ with values
in $\mathfrak{a}$. Put $\eta_{\mathcal{B}}=I$. Define $\beta_{1}\cdot\beta_{2}=\chi(\beta_{1},\beta_{2})$
for all $\beta_{1},\beta_{2}\in\mathcal{B}$. Then $\mathfrak{b}=\mathfrak{a}\oplus\mathcal{B}$
is a (non-associative) algebra with multiplication extending that
on $\mathfrak{a.}$ . For each $n\ge5$, we are going to explicitly
construct a $\Theta_{n}$-graded Lie algebra with coordinate algebra
$\ffb=\mathfrak{a}\oplus\mathcal{B}$. We start with the Peirce decomposition
$\mathfrak{a}=e_{1}\mathfrak{a}e_{1}\oplus e_{1}\mathfrak{a}e_{2}\oplus e_{2}\mathfrak{a}e_{1}\oplus e_{2}\mathfrak{a}e_{2}$.
Note that $\eta(e_{1}\mathfrak{a}e_{1})=e_{2}\mathfrak{a}e_{2}$ and
both $e_{1}\mathfrak{a}e_{2}$ and $e_{2}\mathfrak{a}e_{1}$ are $\eta$-invariant.
Define
\begin{align}
A^{+} & =sym(e_{1}\mathfrak{a}e_{1}\oplus e_{2}\mathfrak{a}e_{2}),\quad A^{-}=skew(e_{1}\mathfrak{a}e_{1}\oplus e_{2}\mathfrak{a}e_{2}),\quad B=\mathcal{B}e_{2},\quad B'=\mathcal{B}e_{1},\label{eq:pd}\\
E & =sym(e_{1}\mathfrak{a}e_{2}),\quad C=skew(e_{1}\mathfrak{a}e_{2}),\quad E'=sym(e_{2}\mathfrak{a}e_{1}),\quad C'=skew(e_{2}\mathfrak{a}e_{1}),\nonumber 
\end{align}
Thus, we have $\mathfrak{a}=A^{+}\oplus A^{-}\oplus C\oplus E\oplus C'\oplus E'$
and $\mathcal{B}=B\oplus B'$. The right $\mathfrak{a}$-module $\mathcal{B}$
can be regarded as a left $\mathfrak{a}$-module by means of the action
$\alpha\cdot\beta=\beta\eta(\alpha)$ for $\alpha\in\mathfrak{a}$
and $\beta\in\mathcal{B}$. Since $\mathfrak{a}$ is a right $\mathfrak{a}$-module
under right multiplication, $\mathfrak{a}^{n}$ ($n\times1$ column
vectors with entries in $\mathfrak{a}$) is also a right $\mathcal{\mathfrak{a}}$-module.
Let $w:\mathfrak{a}^{n}\times\mathfrak{a}^{n}\rightarrow\mathcal{\mathfrak{a}}$
be a non-degenerate bilinear form on $\mathfrak{a}^{n}$ defined by
$w(\alpha_{1},\alpha_{2})=\eta(\alpha_{1})^{t}\alpha_{2}$where $\alpha_{1},\alpha_{2}\in\mathfrak{a}^{n}$.
Let $\xi:(\mathfrak{a}^{n}\oplus\mathcal{B})\times(\mathfrak{a}^{n}\oplus\mathcal{B})\rightarrow\mathfrak{a}^{n}\oplus\mathcal{B}$
be a bilinear form on $\mathfrak{a}^{n}\oplus\mathcal{B}$ defined
by $\xi(\alpha_{1}\oplus\beta_{1},\alpha_{2}\oplus\beta_{2})=w(\alpha_{1},\alpha_{2})-\chi(\beta_{1},\beta_{2})$
where $\beta_{1},\beta_{2}\in\mathcal{B}$ and $\alpha_{1},\alpha_{2}\in\mathfrak{a}^{n}$.
Then 
\[
\mathfrak{U}=\mathfrak{U}(X,\xi)=\left\{ T\in\End_{\mathfrak{a}}(\mathfrak{a}^{n}\oplus\mathcal{B})\mid\xi(T(u),v)+\xi(u,T(v))=0,\ \forall u,v\in\mathfrak{a}^{n}\oplus\mathcal{B}\right\} 
\]
 is a Lie subalgebra of $\End_{\mathfrak{a}}(\mathfrak{a}^{n}\oplus\mathcal{B})$
under the commutator $[T,T']=TT'-T'T$, called the\emph{ unitary Lie
algebra of the hermitian form $\xi=w\bot-\chi$. }We can identify
$\End_{\mathfrak{a}}(\mathfrak{a}^{n}\oplus\mathcal{B})$ in a natural
way with the algebra of $2\times2$ matrices: $\left[\begin{array}{cc}
\End_{\mathfrak{a}}(\mathfrak{a}^{n}) & \Hom_{\mathfrak{a}}(\mathcal{B},\mathfrak{a}^{n})\\
\Hom_{\mathfrak{a}}(\mathfrak{a}^{n},\mathcal{B}) & \End_{\mathfrak{a}}(\mathcal{B})
\end{array}\right]$ whose components have the following realizations: 

$M_{n}(\mathfrak{a})\cong\End_{\mathfrak{a}}(\mathfrak{a}^{n})$ via
map $M\mapsto([\alpha_{1},\dots,\alpha_{n}]^{t}\mapsto M[\alpha_{1},\dots,\alpha_{n}]^{t})$; 

$(\mathcal{B^{\ast}})^{n}\cong\Hom_{\mathfrak{a}}(\mathcal{B},\mathfrak{a}^{n})$
where $\mathcal{B^{\ast}}:=\End_{\mathfrak{a}}(\mathcal{B},\mathfrak{a})$
via map $[\lambda_{1},\dots,\lambda_{n}]^{t}\mapsto(\beta\mapsto[\lambda_{1}\beta,\dots,\lambda_{n}\beta]^{t})$; 

$(\mathcal{B}^{n})^{t}\cong\Hom_{\mathfrak{a}}(\mathfrak{a}^{n},\mathcal{B})$
via map $[\beta_{1},\dots,\beta_{n}]\mapsto([\alpha_{1},\dots,\alpha_{n}]^{t}\mapsto[\beta_{1},\dots,\beta_{n}][\alpha_{1},\dots,\alpha_{n}]^{t})$. 

\noindent Elements of $\mathfrak{a}^{n}\oplus\mathcal{B}$ can be
viewed as column vectors $[\alpha_{1},\dots,\alpha_{n},\beta]^{t}$,
and elements of $\End_{\mathfrak{a}}(\mathfrak{a}^{n}\oplus\mathcal{B})$
can be regarded as matrices $\left[\begin{array}{cc}
M & Y\\
X & N
\end{array}\right]$ where $M\in M_{n}(\mathfrak{a})$, $X=[\beta_{1},\cdots,\beta_{n}],\ (\beta_{i}\in\mathcal{B})$,
$Y=[\lambda_{1},\cdots,\lambda_{n}]^{t},\ (\lambda_{i}\in\mathcal{B^{\ast}})$
and $N\in\End_{\mathfrak{a}}(\mathcal{B})$. The action of $\End_{\mathfrak{a}}(\mathfrak{a}^{n}+\mathcal{B})$
on $\mathfrak{a}^{n}\oplus\mathcal{B}$ is by left multiplication,
and composition in $\End_{\mathfrak{a}}(\mathfrak{a}^{n}\oplus\mathcal{B})$
is matrix multiplication. For $c\in\mathcal{B}$, we define $\chi_{c}:\mathcal{B}\rightarrow\mathfrak{a}$
by $\chi_{c}(c')=\chi(c,c')$. For $\lambda=[\lambda_{1},\dots,\lambda_{n}]^{t}\in(\mathcal{B^{\ast}})^{n},\text{ set }\chi_{\underline{\lambda}}=[\chi_{\lambda_{1}},\dots,\chi_{\lambda_{n}}]^{t}.$
Let $\left[\begin{array}{cc}
M & Y\\
X & N
\end{array}\right]\in\mathfrak{U}$ and $\left[\begin{array}{c}
\alpha_{1}\\
\beta_{1}
\end{array}\right]$ $,\left[\begin{array}{c}
\alpha_{2}\\
\beta_{2}
\end{array}\right]\in\mathfrak{a}^{n}\oplus\mathcal{B}$. Then
\begin{align*}
0 & =\xi(\left[\begin{array}{cc}
M & Y\\
X & N
\end{array}\right]\left[\begin{array}{c}
\alpha_{1}\\
\beta_{1}
\end{array}\right],\left[\begin{array}{c}
\alpha_{2}\\
\beta_{2}
\end{array}\right])+\xi(\left[\begin{array}{c}
\alpha_{1}\\
\beta_{1}
\end{array}\right],\left[\begin{array}{cc}
M & Y\\
X & N
\end{array}\right]\left[\begin{array}{c}
\alpha_{2}\\
\beta_{2}
\end{array}\right])\\
 & =\xi(\left[\begin{array}{c}
M\alpha_{1}+Y\beta_{1}\\
X\alpha_{1}+N\beta_{1}
\end{array}\right],\left[\begin{array}{c}
\alpha_{2}\\
\beta_{2}
\end{array}\right])+\xi(\left[\begin{array}{c}
\alpha_{1}\\
\beta_{1}
\end{array}\right],\left[\begin{array}{c}
M\alpha_{2}+Y\beta_{2}\\
X\alpha_{2}+N\beta_{2}
\end{array}\right])\\
 & =w(M\alpha_{1}+Y\beta_{1},\alpha_{2})-\chi(X\alpha_{1}+N\beta_{1},\beta_{2})+w(\alpha_{1},M\alpha_{2}+Y\beta_{2})-\chi(\beta_{1},X\alpha_{2}+N\beta_{2})\\
 & =\eta(M\alpha_{1}+Y\beta_{1})^{t}\alpha_{2}+\eta(\alpha_{1})^{t}(M\alpha_{2}+Y\beta_{2})-\chi(X\alpha_{1}+N\beta_{1},\beta_{2})-\chi(\beta_{1},X\alpha_{2}+N\beta_{2})\\
 & =\eta(M\alpha_{1})^{t}\alpha_{2}+\eta(Y\beta_{1})^{t}\alpha_{2}+\eta(\alpha_{1})^{t}(M\alpha_{2})+\eta(\alpha_{1})^{t}(Y\beta_{2})\\
 & -\chi(X\alpha_{1},\beta_{2})-\chi(N\beta_{1},\beta_{2})-\chi(\beta_{1},X\alpha_{2})-\chi(\beta_{1},N\beta_{2}).
\end{align*}
We deduce that

(a) $\eta(M\alpha_{1})^{t}\alpha_{2}+\eta(\alpha_{1})^{t}(M\alpha_{2})=0$,
so $\eta(M)^{t}+M=0$; $\quad$(b) $\chi(N\beta_{1},\beta_{2})+\chi(\beta_{1},N\beta_{2})=0$;

(c) $\eta(Y\beta_{1})^{t}\alpha_{2}=\chi(\beta_{1},X\alpha_{2})$;
$\quad$(d) $\eta(\alpha_{1})^{t}(Y\beta_{2})-\chi(X\alpha_{1},\beta_{2})=w(\alpha_{1},Y\beta_{2})-\chi(X\alpha_{1},\beta_{2})=0$.

Fix $X=[\gamma_{1},\cdots,\gamma_{n}]$ and $Y=[\lambda_{1},\dots,\lambda_{n}]^{t}$.
By (c), we have $\eta(Y\beta_{1})^{t}\alpha_{2}=\beta_{1}(X\alpha_{2})$
where $\alpha_{2}\in\mathfrak{a}^{n}$ and $\beta_{1}\in\mathcal{B}$.
Hence $\eta([\lambda_{1}\beta_{1},\dots,\lambda_{n}\beta_{1}])\alpha_{2}=\beta_{1}([\gamma_{1},\cdots,\gamma_{n}]\alpha_{2})$,
so
\[
[\lambda_{1}\beta_{1},\cdots,\lambda_{n}\beta_{1}]=[\eta(\beta_{1}\gamma_{1}),\cdots,\eta(\beta_{1}\gamma_{n})]=[\gamma_{1}\beta_{1},\cdots,\gamma_{n}\beta_{1}].
\]
Therefore $\lambda_{i}\beta_{1}=\gamma_{i}\beta_{1}=\chi(\gamma_{i},\beta_{1})$.
It follows from the nondegeneracy of $w$ that for any $X\in(\mathcal{B}^{n})^{t}\cong\Hom_{\mathfrak{a}}(\mathfrak{a}^{n},\mathcal{B})$,
there is a unique $Y\in(\mathcal{B^{\ast}})^{n}\cong\Hom_{\mathfrak{a}}(\mathcal{B},\mathfrak{a}^{n})$
satisfying (c). Moreover, when $X=(\underline{\beta})^{t}$ in (c),
then $Y=\chi_{\underline{\beta}}$. With these convention, we have
\[
\mathfrak{U}=\left\{ \left[\begin{array}{cc}
M & \chi_{\underline{\beta}}\\
\underline{\beta}^{t} & N
\end{array}\right]\mid M\in M_{n}(\mathfrak{a}),\ (\eta M)^{t}+M=0,\ \underline{\beta}\in\mathcal{B}^{n},\ N\in\mathfrak{U}(\chi)\right\} ,
\]
where $\mathfrak{U}(\chi)=\{N\in\End_{\mathfrak{a}}(\mathcal{B})\mid\chi(N\beta,\beta')+\chi(\beta,N\beta')=0\ \forall\beta,\beta'\in\mathcal{B}\}$
is the unitary Lie algebra of $\chi$. Recall that $1^{+}=e_{1}+e_{2}$.
Put $1^{-}=e_{1}-e_{2}$. Let 
\begin{align*}
\overline{\mathfrak{g}} & =\left\{ \left[\begin{array}{cc}
M & 0\\
0 & 0
\end{array}\right]\mid M\in M_{n}(\bbF)\otimes span\{1^{+},1^{-}\}\text{ and }(\eta M)^{t}+M=0\right\} \\
 & =\left\{ \left[\begin{array}{cc}
M & 0\\
0 & 0
\end{array}\right]\mid M\in sym(M_{n}(\bbF))\otimes1^{-}\oplus skew(M_{n}(\bbF))\otimes1^{+}\right\} .
\end{align*}
 By Lemma \ref{involution}, the map $\eta:M_{n}(\bbF)\otimes\mathfrak{a}\rightarrow M_{n}(\bbF)\otimes\mathfrak{a}$,
given by $\sigma(x\otimes\alpha)=x^{t}\otimes\eta(\alpha)$, is an
involution of the algebra $M_{n}(\bbF)\otimes\mathfrak{a}\cong M_{n}(\mathfrak{a})$.
We have 
\[
skew(M_{n}(\bbF)\otimes\mathfrak{a})=sym(M_{n}(\bbF))\otimes skew(\mathfrak{a})\oplus skew(M_{n}(\bbF))\otimes sym(\mathfrak{a})
\]
where $skew(\mathfrak{a})=A^{-}\oplus C\oplus C'$ and $sym(\mathfrak{a})=A^{+}\oplus E\oplus E'$
with respect to $\eta$. Note that $sym(M_{n}(\bbF))\otimes1^{-}\oplus skew(M_{n}(\bbF))\otimes1^{+}$
is a Lie subalgebra of $skew(M_{n}(\bbF)\otimes\mathfrak{a})$ and
it is isomorphic to $gl_{n}$. (The corresponding isomorphism $\varphi:gl_{n}\rightarrow\overline{\mathfrak{g}}$
is given by $\varphi(x)=\left[\begin{array}{cc}
(x+x^{t})\otimes\frac{(e_{1}-e_{2})}{2}\oplus(x-x^{t})\otimes\frac{(e_{1}+e_{2})}{2} & 0\\
0 & 0
\end{array}\right]).$ Put $\mathfrak{g}=[\overline{\mathfrak{g}},\overline{\mathfrak{g}}]\cong sl_{n}.$
Let $\mathfrak{h}=\left[\begin{array}{cc}
H\otimes1^{-} & 0\\
0 & 0
\end{array}\right]$ where $H$ is the set of diagonal matrices of $sl_{n}$. Then $\mathfrak{h}$
is a Cartan subalgebra of $\mathfrak{g}$ and $\mathfrak{U}$ has
the following weight spaces with respect to the adjoint action of
$\mathfrak{h}$:
\begin{align*}
\mathfrak{U}_{\varepsilon_{i}-\varepsilon_{j}} & =\left\{ \left[\begin{array}{cc}
E_{i,j}\otimes e_{1}\alpha e_{1}+E_{j,i}\otimes e_{2}\alpha e_{2} & 0\\
0 & 0
\end{array}\right]\mid\alpha\in\mathfrak{a}\right\} ,\ 1\leq i\neq j\leq n;\\
\mathfrak{U}_{\varepsilon_{i}+\varepsilon_{j}} & =\left\{ \left[\begin{array}{cc}
E_{i,j}\otimes(c+e)-E_{j,i}\otimes\eta(c+e) & 0\\
0 & 0
\end{array}\right]\mid(c+e)\in C+E\right\} ,\ 1\leq i,j\leq n;\\
\mathfrak{U}_{-\varepsilon_{i}-\varepsilon_{j}} & =\left\{ \left[\begin{array}{cc}
E_{i,j}\otimes(c'+e')-E_{j,i}\otimes\eta(c'+e') & 0\\
0 & 0
\end{array}\right]\mid(c'+e')\in C'+E'\right\} ,\ 1\leq i,j\leq n;\\
\mathfrak{U}_{\varepsilon_{i}} & =\left\{ \left[\begin{array}{cc}
0 & v_{i}\otimes b\\
(v_{i})^{t}\otimes b & 0
\end{array}\right]\mid v\in V,\ b\in B\right\} ,\ 1\leq i\leq n;\\
\mathfrak{U}_{-\varepsilon_{i}} & =\left\{ \left[\begin{array}{cc}
0 & v_{i}'\otimes b'\\
(v_{i}')^{t}\otimes b' & 0
\end{array}\right]\mid v'\in V',\ b'\in B'\right\} ,1\leq i\leq n;\\
\mathfrak{U}_{0} & =\left\{ \left[\begin{array}{cc}
(E_{i,i}-E_{i+1,i+1})\otimes a^{-} & 0\\
0 & 0
\end{array}\right]\mid a^{-}\in A^{-},\ i=1,2,\cdots,n-1\right\} \cup\left\{ \left[\begin{array}{cc}
0 & 0\\
0 & N
\end{array}\right]\mid N\in\mathfrak{U}(\chi)\right\} .
\end{align*}

Note that $\mathfrak{U}$ is $\Theta_{n}$-pregraded but not necessarily
$\Theta_{n}$-graded. Let $\mathfrak{u}$ be the ideal of $\mathfrak{U}$
generated by $\mathfrak{g}$. Then by Proposition \ref{pre-4 equivalent},
$\mathfrak{u}=\bigoplus_{\alpha\in\Theta_{n}\setminus\{0\}}\mathfrak{U}_{\alpha}\bigoplus\sum_{\alpha,-\alpha\in\Theta_{n}\setminus\{0\}}\left[\mathfrak{U}_{\alpha},\mathfrak{U}_{-\alpha}\right]$
and $\mathfrak{u}$ is $\Theta_{n}$-graded.  We call $\mathfrak{u}$
the\emph{ $\Theta_{n}$-graded unitary Lie algebra of} $\xi=w\bot-\chi$. 

Identify $M\otimes\alpha\in M_{n}\otimes\mathfrak{a}$ with $\left[\begin{array}{cc}
M\otimes\alpha & 0\\
0 & 0
\end{array}\right]$, $P\in\End_{\mathfrak{a}}(\mathcal{B})$ with $\left[\begin{array}{cc}
0 & 0\\
0 & P
\end{array}\right]$ and $v\otimes\beta$ with $\left[\begin{array}{cc}
0 & v\otimes\beta\\
v^{t}\otimes\beta & 0
\end{array}\right]$ where $v\in V$ and $\beta\in\mathcal{B}$. As $\mathfrak{g}$-modules,
$\mathfrak{g}\otimes A$, $V\otimes B$, $V'\otimes B'$, $S\otimes C$,
$S'\otimes C'$, $\Lambda\otimes E$ and $\Lambda'\otimes E'$ are
generated by highest weight vectors corresponding to non-zero weights.
Hence, these modules are contained in $\mathfrak{u}$. Then, with
the above identifications, we have 
\[
\mathfrak{u}=(\mathfrak{g}\otimes A)\oplus(V\otimes B)\oplus(V'\otimes B')\oplus\ldots\oplus(\Lambda'\otimes E')\oplus D
\]
 where $D=\left[\begin{array}{cc}
I\otimes A^{-} & 0\\
0 & U(\chi)
\end{array}\right]\cap\mathfrak{u}$ is the centralizer of $\mathfrak{g}$ in $\mathfrak{u}$. We have
a standard Lie bracket on $\mathfrak{u}$: 
\[
[x\otimes\alpha,y\otimes\beta]=(x\otimes\alpha)(y\otimes\beta)-(y\otimes\beta)(x\otimes\alpha)=xy\otimes\alpha\beta-yx\otimes\beta\alpha.
\]
Define $[\alpha_{1},\alpha_{2}]=\alpha_{1}\alpha_{2}-\alpha_{2}\alpha_{1}$
and $\alpha_{1}\circ\alpha_{2}=\alpha_{1}\alpha_{2}+\alpha_{2}\alpha_{1}$
for $\alpha_{1},\alpha_{2}\in\mathfrak{b}$. We claim that the coordinate
algebra of $\mathfrak{u}$ is exactly $\mathfrak{b}$. Note that it
coincides with $\mathfrak{b}$ as a vector space. It remains to check
that the product on $\mathfrak{b}$ induced by the Lie structure of
$\mathfrak{u}$ (see (\ref{pro on a}) and (\ref{pro on b})), which
is denoted by ``$\cdot$'' below, coincides with the original product.
This can be done by multiplying various components of $\mathfrak{u}$.
We illustrate these calculations by checking the following three products:
$a_{1}^{-}\cdot a_{2}^{-}$, $b\cdot b'$ and $a^{-}\cdot b$. Let
$x^{+},x_{1}^{+},x_{2}^{+}\in sl_{n}$, $v\in V$, $v'\in V'$, $a^{-},a_{1}^{-},a_{2}^{-}\in A^{\pm}$,
$b\in B$, $b'\in B'$. We have 
\[
[x_{1}^{+}\otimes a_{1}^{-},x_{2}^{+}\otimes a_{2}^{-}]=[\left[\begin{array}{cc}
x_{1}^{+}\otimes a_{1}^{-} & 0\\
0 & 0
\end{array}\right],\left[\begin{array}{cc}
x_{2}^{+}\otimes a_{2}^{-} & 0\\
0 & 0
\end{array}\right]]=\left[\begin{array}{cc}
x_{1}^{+}x_{2}^{+}\otimes a_{1}^{-}a_{2}^{-}-x_{2}^{+}x_{1}^{+}\otimes a_{2}^{-}a_{1}^{-} & 0\\
0 & 0
\end{array}\right]
\]
Since $a_{1}^{-},a_{2}^{-}\in e_{1}\mathfrak{a}e_{1}\oplus e_{2}\mathfrak{a}e_{2}$,
we have $[a_{1}^{-},a_{2}^{-}],a_{1}^{-}\circ a_{2}^{-}\in e_{1}\mathfrak{a}e_{1}\oplus e_{2}\mathfrak{a}e_{2}$.
As $\eta([a_{1}^{-},a_{2}^{-}])=-[a_{1}^{-},a_{2}^{-}]$ and $\eta(a_{1}^{-}\circ a_{2}^{-})=a_{1}^{-}\circ a_{2}^{-}$,
we have $[a_{1}^{-},a_{2}^{-}]\in A^{-}$ and $a_{1}^{-}\circ a_{2}^{-}\in A^{+}$.
Then $x_{1}^{+}x_{2}^{+}\otimes a_{1}^{-}a_{2}^{-}-x_{2}^{+}x_{1}^{+}\otimes a_{2}^{-}a_{1}^{-}=(x_{1}^{+}x_{2}^{+}-x_{2}^{+}x_{1}^{+})\otimes\frac{[a_{1}^{-},a_{2}^{-}]_{A^{-}}}{2}+(x_{1}^{+}x_{2}^{+}+x_{2}^{+}x_{1}^{+}-\frac{2}{n}\tra(x_{1}^{+}x_{2}^{+})I)\otimes\frac{(a_{1}^{-}\circ a_{2}^{-})_{A^{+}}}{2}+(x_{1}^{+}\mid x_{2}^{+})I\otimes[a_{1}^{-},a_{2}^{-}]_{A^{-}}.$
Therefore, 
\[
[x_{1}^{+}\otimes a_{1}^{-},x_{2}^{+}\otimes a_{2}^{-}]=x_{1}^{+}\circ x_{2}^{+}\otimes\frac{[a_{1}^{-},a_{2}^{-}]_{A^{-}}}{2}+[x_{1}^{+},x_{2}^{+}]\otimes\frac{(a_{1}^{-}\circ a_{2}^{-})_{A^{+}}}{2}+(x_{1}^{+}\mid x_{2}^{+})I\otimes[a_{1}^{-},a_{2}^{-}]_{A^{-}}.
\]
where $(x_{1}^{+}\mid x_{2}^{+})I\otimes[a_{1}^{-},a_{2}^{-}]_{A^{-}}\in D$.
Thus, $a_{1}^{-}\cdot a_{2}^{-}=\frac{[a_{1}^{-},a_{2}^{-}]_{A^{-}}}{2}+\frac{(a_{1}^{-}\circ a_{2}^{-})_{A^{+}}}{2}=\frac{1}{2}((a_{1}^{-}a_{2}^{-}-a_{2}^{-}a_{1}^{-})+(a_{1}^{-}a_{2}^{-}+a_{2}^{-}a_{1}^{-}))=a_{1}^{-}a_{2}^{-}$,
as required.  Similarly, we check $b\cdot b'$:

\begin{gather*}
[v\otimes b,v'\otimes b']=[\left[\begin{array}{cc}
0 & v\otimes b\\
v^{t}\otimes b & 0
\end{array}\right],\left[\begin{array}{cc}
0 & v'\otimes b'\\
v'^{t}\otimes b' & 0
\end{array}\right]]=\left[\begin{array}{cc}
v(v')^{t}\otimes bb'-v'v^{t}\otimes b'b & 0\\
0 & (v)^{t}v'\otimes[b,b']_{A^{-}}
\end{array}\right].
\end{gather*}
Indeed, $b\in\mathcal{B}e_{2}$ and $b'\in\mathcal{B}e_{1}$, so $[b,b'],b\circ b'\in e_{1}\mathfrak{a}e_{1}\oplus e_{2}\mathfrak{a}e_{2}$.
Since $\eta([b,b'])=-[b,b']$ and $\eta(b\circ b')=b\circ b'$, we
have $[b,b']\in A^{-}$ and $b\circ b'\in A^{+}$. Then $v(v')^{t}\otimes bb'-v'v^{t}\otimes b'b=(v(v')^{t}-v'v^{t})\otimes\frac{[b,b']_{A^{-}}}{2}+(v(v')^{t}+v'v^{t}-\frac{2}{n}\tra(v'v^{t})I)\otimes\frac{(b\circ b')_{A^{+}}}{2}+\frac{1}{n}\tra(v'v^{t})\frac{[b,b']_{A^{-}}}{2}.$
Therefore 
\[
[v\otimes b,v'\otimes b']=v'\circ v\otimes\frac{[b,b']_{A^{-}}}{2}+[v',v]\otimes\frac{(b\circ b')_{A^{+}}}{2}+\tra(v(v')^{t})\left[\begin{array}{cc}
\frac{1}{n}I\otimes[b,b']_{A^{-}} & 0\\
0 & 1\otimes[b,b']_{A^{-}}
\end{array}\right].
\]
where $\tra(v(v')^{t})\left[\begin{array}{cc}
\frac{1}{n}I\otimes[b,b']_{A^{-}} & 0\\
0 & 1\otimes[b,b']_{A^{-}}
\end{array}\right]\in D$. Thus, $b\cdot b'=\frac{[b,b']_{A^{-}}}{2}+\frac{(b\circ b')_{A^{+}}}{2}=bb'$,
as required. Since $(x^{+})^{t}=x^{+}$, $\eta(a^{-})=-a^{-}$ and
$(v\otimes b)^{t}(x^{+}\otimes a^{-})^{t}=v^{t}(x^{+})^{t}\otimes ba^{-}=-(x^{+}v\otimes a^{-}b)^{t}$,
we get $[x^{+}\otimes a^{-},v\otimes b]=\left[\begin{array}{cc}
0 & x^{+}v\otimes a^{-}b\\
(x^{+}v)^{t}\otimes a^{-}b & 0
\end{array}\right]=x^{+}v\otimes a^{-}b.$ Thus, $a^{-}\cdot b=a^{-}b$. So the product on $\mathfrak{b}$ determined
by ($\ref{main for}$), (\ref{pro on a}) and (\ref{pro on b}) coincides
with the original product on $\mathfrak{b}$, as required. We also
note that Pierce decomposition (\ref{eq:pd}) for $\mathfrak{b}$
implies that $\mathfrak{u}$ satisfies the conditions (\ref{eq:MainAssumptions}).
We summarize these facts in the following proposition.
\end{example}

\begin{prop}
\label{model ex} Let $n\ge5$ and let $\mathfrak{a}$ and $\mathcal{B}$
be as in Example \ref{exa:Model x}. Let $\mathfrak{u}$ be the $\Theta_{n}$-graded
unitary Lie algebra of the hermitian form $\xi=w\bot-\chi$ on the
$\mathcal{\mathfrak{a}}$-module $\mathfrak{a}^{n}\oplus\mathcal{B}$.
Then $\mathfrak{u}$ \textup{is }$\Theta_{n}$-graded with coordinate
algebra $\mathfrak{b}$. Moreover, $\mathfrak{u}$ satisfies the conditions
(\ref{eq:MainAssumptions}) in the case $n=5,6$.
\end{prop}

Let $L$ be as in Theorem \ref{structure} . By Proposition \ref{derivation 2}
and (\ref{eq:kk}) $\langle A^{+},A^{+}\rangle$, $\langle A^{-},A^{-}\rangle$,
$\langle B,B'\rangle$, $\langle C,C'\rangle$ and $\langle E,E'\rangle$
are ideals of the Lie algebra $D$, $D$ acts by derivations on $\mathfrak{b}$
and leaves all subspaces $A^{+},A^{-},B,B',\dots,E,E'$ invariant
and 
\[
D=\langle\mathfrak{b},\mathfrak{b}\rangle=\langle A^{+},A^{+}\rangle+\langle A^{-},A^{-}\rangle+\langle B,B'\rangle+\langle C,C'\rangle+\langle E,E'\rangle.
\]

For $\alpha,\beta\in\mathfrak{b}$, denote by $D_{\alpha,\beta}$
as follows: if $\alpha\in X$ and $\beta\notin X'$ with $X=B,C,E$
or $\alpha\in A^{+}$ and $\beta\in A^{-}$ then $D_{\alpha,\beta}=D_{\beta,\alpha}=0$;
otherwise, $D_{\alpha,\beta}$ is the $\bbF$-linear map $\gamma\mapsto\langle\alpha,\beta\rangle\gamma$
on $\mathfrak{b}$ as defined in (\ref{eq:<a,a>a}) (e.g. $D_{\alpha,\beta}(\gamma)=[[\alpha,\beta]_{A^{-}},\gamma]$
if $\alpha,\beta,\gamma\in\mathfrak{a}$). Note that the map $D_{\alpha,\beta}$
depends only on the algebra $\mathfrak{b}$ and doesn't depend on
the choice of the specific $\Theta_{n}$-graded Lie algebra $L$ with
coordinate algebra $\mathfrak{b}$, so by Proposition \ref{derivation 2},
$D_{\alpha,\beta}$ is a derivation of $\mathfrak{b}$. More exactly,
$D_{\alpha,\beta}\in\Der_{*}(\mathfrak{b}):=\left\{ d\in\Der_{\bbF}(\mathfrak{b})\mid dX\subseteq X\text{ for }X=A^{+},A^{-},B,\cdots,E'\right\} $,
which is a Lie subalgebra of $\Der_{\bbF}(\mathfrak{b})$. This can
also be checked by straightforward calculations. Set $D_{\mathfrak{b},\mathfrak{b}}=span\{D_{\alpha,\beta}\mid\alpha,\beta\in\mathfrak{b}\}\subseteq\Der_{*}(\mathfrak{b})$.
Let $\varphi:D\rightarrow\Der_{*}(\mathfrak{b})$, $\varphi(d)\beta=d\beta$.
Then $\varphi(D)=D_{\mathfrak{b},\mathfrak{b}}$ and the center $Z(L)$
of $L$ is equal to the kernel of $\varphi$. 
\begin{lem}
\textup{\label{ideal B,b}} $[\psi,D_{\alpha_{1},\alpha_{2}}]=D_{\psi\alpha_{1},\alpha_{2}}+D_{\alpha_{1},\psi\alpha_{2}}$,
for all $\alpha_{1},\alpha_{2}\in\mathfrak{b}$ and $\psi\in\Der_{*}(\mathfrak{b})$.\textup{
In particular, $D_{\mathfrak{b},\mathfrak{b}}$ is an ideal in $\Der_{*}(\mathfrak{b})$.}
\end{lem}

\begin{proof}
 This is checked by straightforward calculations using Proposition
\ref{derivation rule}.  To illustrate this, suppose $\alpha_{1},\alpha_{2}\in\mathfrak{a}$
and $\delta\in\mathfrak{b}$. If $\delta\in\mathfrak{a}$, then, as
required, 
\begin{align*}
[\psi,D_{\alpha_{1},\alpha_{2}}](\delta) & =\psi D_{\alpha_{1},\alpha_{2}}(\delta)-D_{\alpha_{1},\alpha_{2}}\psi(\delta)=\psi([[\alpha_{1},\alpha_{2}]_{A^{-}},\delta])-[[\alpha_{1},\alpha_{2}]_{A^{-}},\psi(\delta)]\\
 & =\psi([\alpha_{1},\alpha_{2}]_{A^{-}}\delta)-\psi(\delta[\alpha_{1},\alpha_{2}]_{A^{-}})-[\alpha_{1},\alpha_{2}]_{A^{-}}.\psi(\delta)+\psi(\delta)[\alpha_{1},\alpha_{2}]_{A^{-}}\\
 & =\psi([\alpha_{1},\alpha_{2}]_{A^{-}})\delta+[\alpha_{1},\alpha_{2}]_{A^{-}}\psi(\delta)-\psi(\delta)[\alpha_{1},\alpha_{2}]_{A^{-}}\\
 & -\delta\psi([\alpha_{1},\alpha_{2}]_{A^{-}})-[\alpha_{1},\alpha_{2}]_{A^{-}}.\psi(\delta)+\psi(\delta)[\alpha_{1},\alpha_{2}]_{A^{-}}\\
 & =\psi([\alpha_{1},\alpha_{2}]_{A^{-}})\delta-\delta\psi([\alpha_{1},\alpha_{2}]_{A^{-}})=[\psi([\alpha_{1},\alpha_{2}]_{A^{-}}),\delta]\\
 & =[(\psi\alpha_{1})\alpha_{2}+\alpha_{1}(\psi\alpha_{2})-(\psi\alpha_{2})\alpha_{1}-\alpha_{2}(\psi\alpha_{1}),\delta]\\
 & =[(\psi\alpha_{1})\alpha_{2}-\alpha_{2}(\psi\alpha),\delta]+[\alpha_{1}(\psi\alpha_{2})-(\psi\alpha_{2})\alpha_{1},\delta]\\
 & =[[\psi\alpha_{1},\alpha_{2}]_{A^{-}},\delta]+[[\alpha_{1},\psi\alpha_{2}]_{A^{-}},\delta]=D_{\psi\alpha_{1},\alpha_{2}}+D_{\alpha_{1},\psi\alpha_{2}}(\delta).
\end{align*}
If $\delta\in B\oplus B'$, then, as required, 
\begin{align*}
[\psi,D_{\alpha_{1},\alpha_{2}}](\delta) & =\psi D_{\alpha_{1},\alpha_{2}}(\delta)-D_{\alpha_{1},\alpha_{2}}\psi(\delta)=\psi([\alpha_{1},\alpha_{2}]_{A^{-}}\delta)-[\alpha_{1},\alpha_{2}]_{A^{-}}\psi(\delta)\\
 & =\psi([\alpha_{1},\alpha_{2}]_{A^{-}})\delta+[\alpha_{1},\alpha_{2}]_{A^{-}}\psi(\delta)-[\alpha_{1},\alpha_{2}]_{A^{-}}.\psi(\delta)=\psi([\alpha_{1},\alpha_{2}]_{A^{-}})\delta\\
 & =\psi(\alpha_{1}\alpha_{2})\delta-\psi(\alpha_{2}\alpha_{1})\delta=((\psi\alpha_{1})\alpha_{2}-\alpha_{2}(\psi\alpha_{1})+\alpha_{1}(\psi\alpha_{2})-(\psi\alpha_{2})\alpha_{1}))\delta\\
 & =[\psi\alpha_{1},\alpha_{2}]_{A^{-}}\delta+[\alpha_{1},\psi\alpha_{2}]_{A^{-}}\delta=D_{\psi\alpha_{1},\alpha_{2}}+D_{\alpha_{1},\psi\alpha_{2}}(\delta).
\end{align*}
\end{proof}
\begin{thm}
\label{L(b) center-1} Let $n\ge5$ and let $\mathfrak{b}$ and $\mathfrak{u}$
be as in Example \ref{exa:Model x}. Define the algebra 
\[
\mathcal{L}(\mathfrak{b}):=(\mathfrak{g^{+}}\otimes A^{-})\oplus(\mathfrak{g}^{-}\otimes A^{+})\oplus(V\otimes B)\oplus\cdots\oplus(\Lambda'\otimes E')\oplus D_{\mathfrak{b},\mathfrak{b}}
\]
 with multiplication as in (\ref{main for}) with $D$ replaced by
$D_{\mathfrak{b},\mathfrak{b}}$ and $\langle\alpha,\beta\rangle$
replaced by $D_{\alpha,\beta}$. Then the following hold.

$(1)$ $\mathcal{L}(\mathfrak{b})$ is a Lie algebra isomorphic to
$\mathfrak{u}/Z(\mathfrak{u})$ where $Z(\mathfrak{u})$ is the center
of $\mathfrak{u}$. 

$(2)$ $\mathcal{L}(\mathfrak{b})$ is $\Theta_{n}$-graded with coordinate
algebra $\mathfrak{b}$.

$(3)$ Every $\Theta_{n}$-graded Lie algebra with coordinate algebra
$\mathfrak{b}$ is a cover of $\mathcal{L}(\mathfrak{b})$.
\end{thm}

\begin{proof}
(1) Define a linear map $f:\mathfrak{u}\rightarrow\mathcal{L}(\mathfrak{b})$
by $f(x)=x$, for all $x\in(\mathfrak{g^{+}}\otimes A^{-})\oplus(\mathfrak{g}^{-}\otimes A^{+})\oplus\cdots\oplus(\Lambda'\otimes E')$
and $f(\langle\alpha,\beta\rangle)=D_{\alpha,\beta}$, for all homogeneous
$\alpha,\beta\in\mathfrak{b}$. It is clear that $f$ is a surjective
map. We claim that $f$ is a Lie algebra homomorphism, i.e. $f([x,y])=[f(x),f(y)]$
for all homogeneous $x,y\in\mathfrak{u}$. This is clear if $x\not\in D$
or $y\not\in D$. If both $x,y\in D$, by using Lemma \ref{ideal B,b},
we get 
\begin{align*}
f([\langle\alpha_{1},\alpha_{2}\rangle,\langle\beta_{1},\beta_{2}\rangle]) & =f(\langle D_{\alpha_{1},\alpha_{2}}\beta_{1},\beta_{2}\rangle+\langle\beta_{1},D_{\alpha_{1},\alpha_{2}}\beta_{2}\rangle)=D_{D_{\alpha_{1},\alpha_{2}}\beta_{1},\beta_{2}}+D_{\beta_{1},D_{\alpha_{1},\alpha_{2}}\beta_{2}}\\
 & =[D_{\alpha_{1},\alpha_{2}},D_{\beta_{1},\beta_{2}}]=[f(\langle\alpha_{1},\alpha_{2}\rangle),f(\langle\beta_{1},\beta_{2}\rangle)],
\end{align*}
as required. It follows from (\ref{main for}) that $\ker(f)=Z(\mathfrak{u})$,
so $\mathcal{L}(\mathfrak{b})$ is a Lie algebra isomorphic to $\mathfrak{u}/Z(\mathfrak{u})$. 

(2) By construction, it is clear that $\mathcal{L}(\mathfrak{b})$
is $\Theta_{n}$-graded with coordinate algebra $\mathfrak{b}$.

(3) Let $L$ be a $\Theta_{n}$-graded Lie algebra with coordinate
algebra $\mathfrak{b}$. By replacing $\mathfrak{u}$ by $L$ in (1),
we get $\mathcal{L}(\mathfrak{b})\cong L/Z(L)$. 
\end{proof}
Next theorem completes the classification of $\Theta_{n}$-graded
Lie algebras up to central extensions.
\begin{thm}[Classification of $\Theta_{n}$-graded Lie algebras]
\label{model theorem}  Let $n\ge5$ and let $L$ be a perfect Lie
algebra. Then $L$ is $(\Theta_{n},\mathfrak{g})$-graded (and satisfies
the conditions (\ref{eq:MainAssumptions}) if $n=5,6$) if and only
if there exist an associative algebra $\mathfrak{a}$ with involution
$\eta$, identity element $1^{+}$ and two orthogonal idempotents
$e_{1}$ and $e_{2}$ such that $1^{+}=e_{1}+e_{2}$ and $e_{2}=\eta(e_{1})$,
a unital associative right $\mathfrak{a}$-module $\mathcal{B}$ with
a hermitian form $\chi$ with values in $\mathfrak{a}$ such that
$L$ is centrally isogenous to the $(\Theta_{n},\mathfrak{g})$-graded
unitary Lie algebra $\mathfrak{u}$ of the hermitian form $\xi=w\bot-\chi$
on the right $\mathcal{\mathfrak{a}}$-module $\mathfrak{a}^{n}\oplus\mathcal{B}$
(see Example \ref{exa:Model x}).
\end{thm}

\begin{proof}
The ``if'' part follows from Proposition \ref{model ex} and Corollary
\ref{cor:iclass}. To prove the ``only if'', suppose that $L$ is
as in the theorem. By Theorem \ref{structure} and Proposition \ref{pdec},
$L$ has coordinate algebra $\mathfrak{b=a}+\mathcal{B}$ with $\mathfrak{a}$
being associative containing two orthogonal idempotents $e_{1}$ and
$e_{2}$ with the above properties. By Proposition \ref{model ex},
the $(\Theta_{n},\mathfrak{g})$-graded unitary Lie algebra $\mathfrak{u}$
has the same coordinate algebra. By Theorem \ref{L(b) center-1},
$L/Z(L')\cong\mathcal{L}(\mathfrak{b})\cong\mathfrak{u}/Z(\mathfrak{u})$,
so $L$ and $\mathfrak{u}$ are centrally isogenous.
\end{proof}
\begin{rem}
There is another approach to classification of weight-graded Lie algebras
by using so-called structurable algebras (non-associative unital algebras
with involution satisfying certain identities), see for example \cite[Appendix]{allison2002lie}
and \cite{allison1978class,allison2003structurable}. Any structurable
algebra $A$ gives rise to a Lie algebra $K(A)$ via the so-called
Tits-Kantor-Koecher construction. By imposing some extra conditions
on $A$, one can make the Lie algebra $K(A)$ weight-graded and then
describe its coordinate algebra $\mathfrak{b}$ in terms of the structurable
algebra $A$, see for example \cite[Example 6.36]{allison2002lie}.
This approach is more technical but probably unavoidable in the case
of the grading subalgebras of small rank as the coordinate algebras
$\mathfrak{b}$ become much more difficult to characterize, see for
example \cite[Ch. 6]{allison2002lie}.
\end{rem}

\subsection{\label{universal section} Universal central extensions}

In Theorem \ref{L(b) center-1} we defined the centerless $(\Theta_{n},\mathfrak{g})$-graded
Lie algebra $\mathcal{L}(\mathfrak{b}):=(\mathfrak{g^{+}}\otimes A^{-})\oplus(\mathfrak{g}^{-}\otimes A^{+})\oplus(V\otimes B)\oplus\cdots\oplus(\Lambda'\otimes E')\oplus D_{\mathfrak{b},\mathfrak{b}}$
with coordinate algebra $\mathfrak{b}$ and multiplication as in (\ref{main for})
with $D$ replaced by $D_{\mathfrak{b},\mathfrak{b}}$ and $\langle\alpha,\beta\rangle$
replaced by $D_{\alpha,\beta}$. In this subsection we compute the
universal central extension $\widehat{\mathcal{L}(\mathfrak{b})}$
of $\mathcal{L}(\mathfrak{b})$ and we show that for every $\Theta_{n}$-graded
Lie algebra $L$ there is a subspace $X$ of the center of $\widehat{\mathcal{L}(\mathfrak{b})}$
such that $L$ is isomorphic to $\mathcal{L}(\mathfrak{b},X):=\widehat{\mathcal{L}(\mathfrak{b})}/X$.
We prove that the center of $\widehat{\mathcal{L}(\mathfrak{b})}$
is $\hfb(\mathfrak{b})$ (the full skew-dihedral homology group of
$\mathfrak{b}$). This finishes the classification of $\Theta_{n}$-graded
Lie algebras up to isomorphism. 

Recall that $\Der_{*}(\mathfrak{b}):=\left\{ d\in\Der(\mathfrak{b})\mid dX\subseteq X\text{ for }X=A^{+},A^{-},B,\cdots,E'\right\} $
and $D_{\mathfrak{b},\mathfrak{b}}=span\{D_{\alpha,\beta}\mid\alpha,\beta\in\mathfrak{b}\}\subseteq\Der_{*}(\mathfrak{b})$.
The subspace $D_{\mathfrak{b},\mathfrak{b}}$ is a Lie subalgebra
(and ideal) of $\Der_{*}(\mathfrak{b})$ and $D_{\mathfrak{b},\mathfrak{b}}(X)\subseteq X\text{ }$
for $X=A^{+}$, $A^{-}$, $B$, $\cdots$, $E'$. By definition, $D_{x,y}=0$
if $x\in X$ and $y\notin X'$ with $X=B,C,E$ or $x\in A^{+}$ and
$y\in A^{-}$. From condition $(\Gamma3)$ in Definition \ref{def of gamma}
get $D_{\mathfrak{b},\mathfrak{b}}=D_{A^{+},A^{+}}+D_{A^{-},A^{-}}+D_{B,B'}+D_{C,C'}+D_{E,E'}$. 
\begin{prop}
$D_{\alpha,\beta}+D_{\beta,\alpha}=0$ and $D_{\alpha\beta,\gamma}+D_{\beta\gamma,\alpha}+D_{\gamma\alpha,\beta}=0$
for all $\alpha,\beta,\gamma\in\mathfrak{b}$.
\end{prop}

\begin{proof}
From anti-commutativity of the bracket of of $\mathcal{L}(\mathfrak{b})$
and the fact that $\tra(xy)=\tra(yx)$, $\tra(uv'^{t})=\tra(v'u^{t})$,
for all $n\times n$ matrices $x$ and $y$ and $v\in V$ and $v'\in V'$,
we deduce that $D_{\alpha,\beta}=-D_{\beta,\alpha}$ for all $\alpha,\beta\in\mathfrak{b}$.
It remains to show that $D_{\alpha\beta,\gamma}+D_{\beta\gamma,\alpha}+D_{\gamma\alpha,\beta}=0$.
This can be proved by making various choices of $z_{1}\otimes\alpha$,
$z_{2}\otimes\beta$, $z_{3}\otimes\gamma\in(\mathfrak{g^{+}}\otimes A^{-})\cup(\mathfrak{g}^{-}\otimes A^{+})\cup(V\otimes B)\cup\dots\cup(\Lambda'\otimes E')$
and calculating the corresponding Jacoby identity. As illustration,
consider $\alpha=a^{-}\in A^{-}$, $\beta=b'\in B'$, $\gamma=b\in B$.
Write the Jacoby identity for $(E_{1,2}+E_{2,1})\otimes a^{-}$, $e_{2}\otimes b'$
and $e_{1}\otimes b$, then use Lemma \ref{bb'} and evaluate the
$D_{\mathfrak{b},\mathfrak{b}}$-component to get  $D_{\delta,a^{-}}+D_{b'a^{-},b}+D_{a^{-}b,b'}=0$
where $\delta=\frac{1}{2}[b,b']_{A^{-}}$. Since $\delta=bb'-\frac{1}{2}(b\circ b')_{A^{+}}$
and $D_{\frac{1}{2}(b\circ b')_{A^{+}},a^{-}}=0$, we get,  $D_{bb',a^{-}}+D_{b'a^{-},b}+D_{a^{-}b,b'}=0$,
as required.
\end{proof}
Let $I$ be the subspace of $\mathfrak{b}\otimes\mathfrak{b}$ spanned
by the elements 
\begin{alignat}{1}
 & \alpha\otimes\beta+\beta\otimes\alpha,\quad\gamma\alpha\otimes\beta+\beta\gamma\otimes\alpha+\alpha\beta\otimes\gamma,\quad x\otimes y\label{I skew}
\end{alignat}
where $\alpha,\beta,\gamma\in\mathfrak{b}$, $x\in X$ and $y\notin X'$
with $X=B,C,E$ or $x\in A^{+}$ and $y\in A^{-}$. Recall that $D_{\mathfrak{b},\mathfrak{b}}$
is a Lie subalgebra of $\Der_{*}(\mathfrak{b})$, so $\mathfrak{b}$
and $\mathfrak{b}\otimes\mathfrak{b}$ are $D_{\mathfrak{b},\mathfrak{b}}$-modules.
It is easy to see that the space $I$ is invariant under $D_{\mathfrak{b},\mathfrak{b}}$,
and so the quotient space $\left\{ \mathfrak{b},\mathfrak{b}\right\} :=\mathfrak{b}\otimes\mathfrak{b}/I$
is a $D_{\mathfrak{b},\mathfrak{b}}$-module under the induced action:
\[
D_{\alpha_{1},\alpha_{2}}.\left\{ \beta_{1},\beta_{2}\right\} :=\left\{ D_{\alpha_{1},\alpha_{2}}\beta_{1},\beta_{2}\right\} +\left\{ \beta_{1},D_{\alpha_{1},\alpha_{2}}\beta_{2}\right\} 
\]
where $\left\{ \alpha,\beta\right\} :=\alpha\otimes\beta+I$ in $\left\{ \mathfrak{b},\mathfrak{b}\right\} $.
Then the relations in (\ref{I skew}) translate to say $\{\alpha,\beta\}=-\{\beta,\alpha\}$,
$\{\gamma\alpha,\beta\}+\{\beta\gamma,\alpha\}+\{\alpha\beta,\gamma\}=0$
and $\{x,y\}=0$. The mapping $\mathfrak{b}\otimes\mathfrak{b}\rightarrow D_{\mathfrak{b},\mathfrak{b}}$,
$\alpha\otimes\beta\mapsto D_{\alpha,\beta}$ contains $I$ in the
kernel. We define the induced mapping $\rho:\left\{ \mathfrak{b},\mathfrak{b}\right\} \rightarrow D_{\mathfrak{b},\mathfrak{b}}$
by $\rho(\left\{ \alpha,\beta\right\} )=D_{\alpha,\beta}$. 
\begin{prop}
\label{braket and surjective} (1) The space $\left\{ \mathfrak{b},\mathfrak{b}\right\} $
is a Lie algebra with the multiplication $[\left\{ \alpha_{1},\alpha_{2}\right\} ,\left\{ \beta_{1},\beta_{2}\right\} ]=\left\{ D_{\alpha_{1},\alpha_{2}}\beta_{1},\beta_{2}\right\} +\left\{ \beta_{1},D_{\alpha_{1},\alpha_{2}}\beta_{2}\right\} ,$
for all $\alpha_{1},\alpha_{2},\beta_{1},\beta_{2}\in\mathfrak{b}$.

(2) The mapping $\rho:\left\{ \mathfrak{b},\mathfrak{b}\right\} \rightarrow D_{\mathfrak{b},\mathfrak{b}}$
given by $\rho(\left\{ \alpha,\beta\right\} )=D_{\alpha,\beta}$ is
a surjective Lie algebra homomorphism.
\end{prop}

\begin{proof}
This can be checked by making various choices of elements in $\mathfrak{b}$
and calculating the corresponding derivations by using Proposition
\ref{derivation rule}, see \cite[4.8-4.10]{allison2000central},
\cite[5.24]{allison2002lie} and \cite[Proposition 5.3.4]{yaseen2018generalized}. 
\end{proof}
Propositions \ref{derivation 2} and \ref{braket and surjective}
imply the following.
\begin{prop}
$\mathfrak{b}$ is a module for the Lie algebra $\left\{ \mathfrak{b},\mathfrak{b}\right\} $
with action defined by $\{\alpha,\beta\}.\gamma=\rho(\{\alpha,\beta\})\gamma=D_{\alpha,\beta}\gamma$
for $\{\alpha,\beta\}\in\left\{ \mathfrak{b},\mathfrak{b}\right\} $,
$\gamma\in\mathfrak{b}.$ This action stabilizes the subspaces $A^{+},A^{-},B,\cdots,E'$.
\end{prop}

\begin{defn}
\cite[5.26]{allison2002lie} The \emph{full skew-dihedral homology
group} of $\mathfrak{b}$ is
\[
\hfb(\mathfrak{b}):=\ker\rho=\left\{ \sum_{i}\left\{ \alpha_{i},\beta_{i}\right\} \in\left\{ \mathfrak{b},\mathfrak{b}\right\} \mid\sum_{i}D_{\alpha_{i},\beta_{i}}=0\right\} .
\]
\end{defn}

\begin{thm}
\label{universal covering} Let $n\ge5$ and let $\mathfrak{a}$ and
$\mathcal{B}$ be as in Example \ref{exa:Model x}. Let 
\[
\widehat{\mathcal{L}(\mathfrak{b})}:=(\mathfrak{g^{+}}\otimes A^{-})\oplus(\mathfrak{g}^{-}\otimes A^{+})\oplus\cdots\oplus(\Lambda'\otimes E')\oplus\left\{ \mathfrak{b},\mathfrak{b}\right\} 
\]
 be the algebra with multiplication defined by (\ref{main for}) with
$D$ replaced by $\left\{ \mathfrak{b},\mathfrak{b}\right\} $ and
$\langle\alpha,\beta\rangle$ replaced by $\left\{ \alpha,\beta\right\} $.\textup{
}Consider the map $f:\widehat{\mathcal{L}(\mathfrak{b})}\rightarrow\mathcal{L}(\mathfrak{b})$
given by $f(x)=x$ for all $x\in(\mathfrak{g}\otimes A)\oplus\cdots\oplus(\Lambda'\otimes E')$
and $f(\left\{ \alpha,\beta\right\} )=D_{\alpha,\beta}$ for all $\left\{ \alpha,\beta\right\} \in\left\{ \mathfrak{b},\mathfrak{b}\right\} $.
Then $(\widehat{\mathcal{L}(\mathfrak{b})},f)$ is the universal covering
algebra of $\mathcal{L}(\mathfrak{b})$ and the center of \textup{$\widehat{\mathcal{L}(\mathfrak{b})}$}
is $\hfb(\mathfrak{b})$.
\end{thm}

\begin{proof}
 This is similar to \cite[Theorem 4.13]{allison2000central} and
\cite[Theorem 5.34]{allison2002lie}. First, we need to check that
$\widehat{\mathcal{L}(\mathfrak{b})}$ is a Lie algebra under the
above multiplication. Note that the products in (\ref{main for})
are bilinear and antisymmetric. It remains to check \foreignlanguage{english}{$\widehat{\mathcal{L}(\mathfrak{b})}$}
satisfies the Jacobi identity. Observe that if at least $2$ of the
$3$ factors are from $(\mathfrak{g}\otimes A)\oplus\cdots\oplus(\Lambda'\otimes E')$,
then the products behave as in $\mathcal{L}(\mathfrak{b})$. The only
difference is that the $\left\{ \mathfrak{b},\mathfrak{b}\right\} $-component
of the products involves expressions such as $\{\alpha_{1},\alpha_{2}\}$
rather than $D_{\alpha_{1},\alpha_{2}}$. But when such a term acts
on $\mathfrak{b}$, the action of the two is the same. When all of
them belong to $\left\{ \mathfrak{b},\mathfrak{b}\right\} $, by Proposition
\ref{braket and surjective}, the Jacobi identity holds. When exactly
2 of the 3 factors belongs to $\left\{ \mathfrak{b},\mathfrak{b}\right\} $
then we can use the fact that the products of the form $[\{\alpha_{1},\alpha_{2}\},\{\beta_{1},\beta_{2}\}]$
are represented as $[D_{\alpha_{1},\alpha_{2}},D_{\beta_{1},\beta_{2}}]$,
see Proposition \ref{braket and surjective}. As illustration, we
consider $\{\alpha_{1},\alpha_{2}\},\{\beta_{1},\beta_{2}\}\in\left\{ \mathfrak{a},\mathfrak{a}\right\} $
and $x\otimes\alpha\in(\mathfrak{g^{+}}\otimes A^{-})\oplus(\mathfrak{g}^{-}\otimes A^{+})\oplus\cdots\oplus(\Lambda\otimes E)\oplus(\Lambda'\otimes E')$.
Using Proposition \ref{derivation rule} and the associativity of
$\mathcal{\mathfrak{a}}$ we get 
\begin{align*}
 & [[\left\{ \alpha_{1},\alpha_{2}\right\} ,\left\{ \beta_{1},\beta_{2}\right\} ],x\otimes\alpha]=[\left\{ D_{\alpha_{1},\alpha_{2}}\beta_{1},\beta_{2}\right\} +\left\{ \beta_{1},D_{\alpha,\alpha_{2}}\beta_{2}\right\} ,x\otimes\alpha]\\
 & \quad\quad=[\left\{ [[\alpha_{1},\alpha_{2}],\beta_{1}],\beta_{2}\right\} ,x\otimes\alpha]+[\left\{ \beta_{1},[[\alpha_{1},\alpha_{2}],\beta_{2}]\right\} ,x\otimes\alpha]\\
 & \quad\quad=x\otimes([[[[\alpha_{1},\alpha_{2}],\beta_{1}],\beta_{2}],\alpha]+[[\beta_{1},[[\alpha_{1},\alpha_{2}],\beta_{2}]],\alpha])=x\otimes[[[\alpha_{1},\alpha_{2}],[\beta_{1}\beta_{2}]],\alpha]\\
 & \quad\quad=x\otimes([[\alpha_{1},\alpha_{2}],[[\beta_{1},\beta_{2}],\alpha]]+[[[\alpha_{1},\alpha_{2}],\alpha],[\beta_{1},\beta_{2}]])\\
 & \quad\quad=[\left\{ \alpha_{1},\alpha_{2}\right\} ,x\otimes[[\beta_{1},\beta_{2}],\alpha]]+[x\otimes[[\alpha_{1},\alpha_{2}],\alpha],\left\{ \beta_{1},\beta_{2}\right\} ]\\
 & \quad\quad=[\left\{ \alpha_{1},\alpha_{2}\right\} ,[\left\{ \beta_{1},\beta_{2}\right\} ,x\otimes\alpha]+[[\left\{ \alpha_{1},\alpha_{2}\right\} ,x\otimes\alpha],\left\{ \beta_{1},\beta_{2}\right\} ]
\end{align*}
Therefore $\widehat{\mathcal{L}(\mathfrak{b})}$ with the above multiplication
is a Lie algebra. By its construction, $\widehat{\mathcal{L}(\mathfrak{b})}$
is graded by the same root system as $\mathcal{L}(\mathfrak{b})$
and it is perfect. By Proposition \ref{braket and surjective}, $f$
is a surjective Lie algebra homomorphism and $\ker f=\left\{ \sum_{i}\left\{ \alpha_{i},\beta_{i}\right\} \in\left\{ \mathfrak{b},\mathfrak{b}\right\} \mid\sum_{i}D_{\alpha_{i},\beta_{i}}=0\right\} .$
Thus, $(\widehat{\mathcal{L}(\mathfrak{b})},f)$ is a central extension
of $L$. We have $\ker f\subseteq Z(\widehat{\mathcal{L}(\mathfrak{b})})$
and it easy to check that $Z(\widehat{\mathcal{L}(\mathfrak{b})})\subseteq\ker f$,
so $Z(\widehat{\mathcal{L}(\mathfrak{b})})=\ker f=\hfb(\mathfrak{b}),$
as required.

To see that $f:\widehat{\mathcal{L}(\mathfrak{b})}\rightarrow\mathcal{L}(\mathfrak{b})$
is universal, suppose that $f:\widetilde{\mathcal{L}(\mathfrak{b})}\rightarrow\mathcal{L}(\mathfrak{b})$
is a central extension of $L$. By Lemma \ref{lifting 1 }, we can
lift $\mathcal{L}(\mathfrak{b})$ to a subspace of $\widetilde{\mathcal{L}(\mathfrak{b})}$,
which we identify with $\mathcal{L}(\mathfrak{b})$, so that the corresponding
2-cocycle satisfies $\zeta(\mathfrak{g},\mathcal{L}(\mathfrak{b}))=0$.
Then, by Theorem \ref{lift 2}, we may assume that the corresponding
2-cocycle is obtained from a 2-cocycle $\epsilon$ of $\mathfrak{b}$
as in (\ref{eq:lift2 equation}). The 2-cocycle $\epsilon$ induces
a mapping $\tilde{\epsilon}:\left\{ \mathfrak{b},\mathfrak{b}\right\} \rightarrow\mathbb{E}$
with $\left\{ \alpha,\beta\right\} \mapsto\epsilon(\alpha,\beta)\in\mathbb{E}$.
Thus, there is a homomorphism $\varphi:\widehat{\mathcal{L}(\mathfrak{b})}\rightarrow\widetilde{\mathcal{L}(\mathfrak{b})}$
with $\varphi(x)=x$ for all $x\in(\mathfrak{g}\otimes A)\oplus\cdots\oplus(\Lambda'\otimes E')$
and $\varphi(\left\{ \alpha,\beta\right\} )=D_{\alpha,\beta}+\tilde{\epsilon}(\alpha,\beta)$
for all $\left\{ \alpha,\beta\right\} \in\left\{ \mathfrak{b},\mathfrak{b}\right\} $.
Hence $\widehat{\mathcal{L}(\mathfrak{b})}$ is the universal covering
algebra of $\mathcal{L}(\mathfrak{b})$, as required.
\end{proof}
Let $X$ be a subspace of $\hfb(\mathfrak{b})=Z(\widehat{\mathcal{L}(\mathfrak{b})})$.
Consider the quotient space $\prec\mathfrak{b},\mathfrak{b}\succ=\left\{ \mathfrak{b},\mathfrak{b}\right\} /X$
and set $\prec\alpha,\beta\succ=\left\{ \alpha,\beta\right\} +X$
in $\left\{ \mathfrak{b},\mathfrak{b}\right\} /X$. Let 
\begin{equation}
\mathcal{L}(\mathfrak{b},X):=(\mathfrak{g}\otimes A)\oplus\cdots\oplus(\Lambda'\otimes E')\oplus\prec\mathfrak{b},\mathfrak{b}\succ\label{eq:X2}
\end{equation}
 be the algebra with multiplication same as $\mathcal{L}(\mathfrak{b})$
with $D_{\alpha,\beta}$ replaced by $\prec\alpha,\beta\succ$. Then
we have the following. 
\begin{thm}
\label{main universal}(1) $\mathcal{L}(\mathfrak{b},X)$ is a $(\Theta_{n},\mathfrak{g})$-graded
Lie algebra with coordinate algebra $\mathfrak{b}$.

(2) Every $\Theta_{n}$-graded Lie algebra with coordinate algebra
$\mathfrak{b}$ is isomorphic to $\mathcal{L}(\mathfrak{b},X)$ for
some subspace $X$ of $\hfb(\mathfrak{b})$.
\end{thm}

\begin{proof}
This is proved by using the same arguments as in \cite[Theorem 4.20]{allison2000central}
and \cite[Theorem 5.35]{allison2002lie}. 
\end{proof}
\global\long\def\bibname{References}%

\bibliographystyle{abbrv}
\bibliography{references}

\appendix
\printindex{}

\end{document}